\documentclass[12pt]{amsart}
\usepackage{latexsym,color,fancyhdr,amssymb,amsmath,amsthm,graphicx,listings,comment}
\definecolor{color1}{rgb}{0.8,1,0.8}
\usepackage[section]{placeins}
\newtheorem{thm}{Theorem} 
\newtheorem{coro}{Corollary} 
\newtheorem{lemma}{Lemma} 
\newtheorem{propo}{Proposition} 
 
\newcommand{\R}{\mathbb{R}} \newcommand{\Z}{\mathbb{Z}}  \newcommand{\N}{\mathbb{N}} 
\newcommand{\HH}{\mathbb{H}} \newcommand{\C}{\mathbb{C}}
  
\def\B#1#2{{#1\choose #2}}
\def\ssection#1{ \begin{center} \fcolorbox{color1}{color1}{ \parbox{17.2cm}{\bf{ \Large{ #1 }}}} \end{center} }
\let\paragraph\subsection
\definecolor{bbyellow}{rgb}{1.0,1.0,0.8}
\def\tttheorem#1{ \begin{center} \fcolorbox{bbyellow}{bbyellow}{ \parbox{17.2cm}{\bf{ \Large{ #1 }}}} \end{center} }
\title{The Cohomology for Wu Characteristics}
\author{Oliver Knill}
\date{March 14, 2018}
\address{Department of Mathematics \\ Harvard University \\ Cambridge, MA, 02138, USA }


\begin{document}
\maketitle

\begin{abstract}
While the Euler characteristic $\chi(G)=\omega_1(G) = \sum_x \omega(x)$ super counts simplices, the 
Wu characteristics $\omega_k(G) = \sum_{x_1 \sim x_2 \dots \sim x_k} \omega(x_1) \cdots \omega(x_k)$ 
super counts simultaneously pairwise interacting $k$-tuples of simplices in a finite abstract
simplicial complex $G$. More generally, one can define the $k$-intersection number 
$\omega_k(G_1, \dots G_k)$, which is the same sum but where $x_i \in G_i$.
For every $k \geq 1$ we define a cohomology $H^p_k(G_1,\dots,G_k)$ compatible 
with $\omega_k$. It is invariant under Barycentric subdivison. 
This interaction cohomology allows to distinguish spaces which simplicial 
cohomology can not: for $k=2$, it can identify algebraically the M\"obius 
strip and the cylinder. The vector spaces $H^p_k(G)$ are defined by explicit exterior 
derivatives $d_k$ which generalize the incidence matrices for simplicial cohomology. 
The cohomology satisfies the Kuenneth formula: for every $k$, 
the Poincar\'e polynomials $p_k(t)$ are ring homomorphisms from the strong ring to the 
ring of polynomials in $t$.
The case $k=1$ for Euler characteristic is the familiar simplicial cohomology $H^p_1(G)=H^p(G)$. 
On any interaction level $k$, there is now a Dirac operator $D=d_k+d_k^*$. 
The block diagonal Laplacian $L=D^2$ leads to the generalized Hodge 
correspondence $b_p(G)={\rm dim}(H^p_k(G)) = {\rm dim}({\rm ker}(L_p))$ and Euler-Poincar\'e
$\omega_k(G) = \sum_p (-1)^p {\rm dim}(H^p_k(G))$ for Wu characteristic and more generally
$\omega_k(G_1, \dots G_k) = \sum_p (-1)^p {\rm dim}(H^p_k(G_1,\dots, G_k))$.
Also, like for traditional simplicial cohomology, an isospectral Lax deformation 
$\dot{D} = [B(D),D]$, with $B(t)=d(t)-d^*(t)-ib(t)$, $D(t)=d(t)+d(t)^* + b(t)$
can deform the exterior derivative $d$ which belongs to the interaction cohomology.
Also the Brouwer-Lefschetz fixed point theorem 
generalizes to all Wu characteristics: given an endomorphism $T$ of $G$, 
the super trace of its induced map on k'th cohomology defines a Lefschetz number $L_k(T)$. 
The Brouwer index $i_{T,k}(x_1, \dots, x_k) = \prod_{j=1}^k \omega(x_j) {\rm sign}(T|x_j)$ 
attached to simplex tuple which is invariant under $T$ leads to the formula
$L_k(T) = \sum_{T(x)=x} i_{T,k}(x)$. For $T=Id$, the Lefschetz number $L_k(Id)$
is equal to the $k$'th Wu characteristic $\omega_k(G)$ of the graph $G$
and the Lefschetz formula reduces to the Euler-Poincar\'e formula for Wu characteristic. 
Also this generalizes to the case, where automorphisms $T_k$ act on $G_k$:
there is a Lefschetz number $L_k(T_1,\dots, T_k)$ and indices 
$\prod_{j=1}^k \omega(x_j) {\rm sign}(T_j|x_j)$. 
For $k=1$, it is the known Lefschetz formula for Euler characteristic. While Gauss-Bonnet
for $\omega_k$ can be seen as a particular case of a discrete Atiyah-Singer result, the Lefschetz formula
is an particular case of a discrete Atiyah-Bott result. But unlike for Euler characteristic $k=1$, 
the elliptic differential complexes for $k>1$ are not yet associated to any constructions 
in the continuum.
\end{abstract}

\section{Introduction}

\paragraph{}
An abstract finite simplicial complex $G$ is a finite set of non-empty subsets that is closed
under the process of taking non-empty subsets. It defines a new complex $G_1$,
where $G_1$ is the set of subsets of the power set $2^G$ which are pairwise contained in each other.
This is called the Barycentric refinement of $G$. 
Quantities which do not change under Barycentric refinements are {\bf combinatorial invariants} \cite{Bott52}.
Examples of such quantities are the {\bf Euler characteristic}, the {\bf Wu characteristics} \cite{Wu1959,valuation} 
or {\bf Green function values} \cite{Spheregeometry}. An important class of
complexes are Whitney complexes of a finite simple graph in which $G$
consists of all non-empty subsets of the vertex set, in which all elements
are connected to each other. A Barycentric
refinement $G_1$ of a complex $G$ is always a Whitney complex. 
The graph has the sets of $G$ as vertices and connects two if one 
is contained in the other. 

\paragraph{}
While Euler characteristic $\chi(G)=\sum_x \omega(x)$ super counts the self interactions of 
simplices within the geometry, the Wu characteristic $\omega_k$ super counts 
the intersections of $k$ different simplices 
$$   \omega_k(G) = \sum_{x_1 \sim \dots \sim x_k} \omega(x_1) \cdots \omega(x_k) \; , $$ 
where $x_1 \sim \dots \sim x_k$ means that 
$(x_1,x_2, \cdots x_k) \in G^k = G \times G \times \cdots \times G$ all pairwise intersect.
While originally introduced for polytopes by Wu in 1959, we use it for simplicial complexes 
and in particular for Whitney complexes of graphs. Especially important is the quadratic
case 
$$ \sum_{x \sim y} \omega(x) \omega(y) $$
which produces an Ising type nearest neighborhood functional if we think of even 
or odd dimensional simplices as spins of a two-state statistical mechanics system. 
We call the cohomology for Wu characteristic an {\bf interaction 
cohomology} because the name "intersection cohomology" is usedin the continuum for
a cohomology theory by Goresky-MacPherson.

\paragraph{}
As in \cite{valuation}, where we proved Gauss-Bonnet, Poincar\'e-Hopf, 
and product formulas for Wu characteristics in the case of Whitney complexes, we
can work in the slightly more general frame work of finite abstract simplicial complexes.
If we look at set of subsets of $G \times G \cdots \times G$ of a simplicial complex $G$ 
such that all $k$ elements pairwise intersect, we obtain a $k$-intersection complex.
For $k=1$ the complex is $G$ itself, the connection-oblivious complex.
For $k=2$ the intersection relations give the interaction graph $G'$ which was studied in 
\cite{Unimodularity,ListeningCohomology,StrongRing,DehnSommerville,Helmholtz,Spheregeometry}. 
It contains the Barycentric refinement graph $G_1$. 
For every $k>1$ we get an example of an ``ordinary cohomology" generalizing the
simplicial cohomology. 

\paragraph{}
A major reason to study the Wu characteristic, is that 
each $\omega_k$ is multiplicative with respect to the graph product \cite{KnillKuenneth}
$G,H \to (G \times H)_1$, in which the pairs $(x,y)$ are vertices with $x \in G$
and $y \in H$ and where two vertices are connected if one is contained in the other. 
It is already multiplicative with respect to the Cartesian product $G \times H$, which
unlike $(G \times H)_1$ is not a simplicial complex any more. 
The multiplicative property makes the $\omega_k$ functionals a natural counting tool, 
en par with Euler characteristic $\chi(G)=\omega_1(G)$, which is the special case $k=1$. The 
corresponding cohomologies belong to a generalized calculus which does
not yet seem to exist in the continuum. The multiplicative property of the functionals
is especially natural when looking at the Cartesian product $G \times H$ in the strong ring
generated by simplicial complexes in which the simplicial complexes are the multiplicative
primes and where the objects are signed to have an additive group. 
The subring of zero-dimensional complexes can be identified with the ring of integers. 

\paragraph{}
Traditionally, graphs are seen as zero-dimensional simplicial complexes. 
A more generous point of view is to see a graph $G=(V,E)$ as a structure on which 
simplicial complexes can live, similarly as topologies, sheaves or $\sigma$-algebras 
are placed on sets. The $1$-dimensional complex $V \cup E$ is only one of the possible
simplicial complexes. An other example 
is the graphic matroid defined by trees in the graph. 
But the most important complex imposed on a graph is the {\bf Whitney complex}.
It is defined by the set of complete subgraphs of $G$.
Without specifying further, we always associate a graph equipped with its Whitney complex.

\paragraph{}
The class of all finite simple graphs $G=(V,E)$ is a distributive Boolean 
lattice with union $(V_1 \cup V_2,E_1 \cup E_2)$
and intersection $(V_1 \cap V_2,E_1 \cap E_2)$. The same is true for finite
abstract simplicial complexes, where also the union and intersection remain
both simplicial complexes. But unlike the Boolean lattice of sets,
the Boolean lattices of graphs or complexes are each not a ring: 
the set theoretical addition $G \Delta H = G \cup H \setminus G \cap H$ is in general only in  
the larger category of chains. But if we take the disjoint union of graphs as "addition" 
and the Cartesian multiplication of sets as multiplication, we get a ring \cite{StrongRing}. 

\paragraph{}
All of the Wu characteristics $\omega_k$ are functions 
on this ring. They are of fundamental importance because they are not only additive 
$\omega_k(A \dot{\cup} B) = \omega_k(A) + \omega_k(B)$ but also multiplicative
$\omega_k(A \times B) = \omega_k(A) \omega_k(B)$. In other words, 
all the Wu characteristic numbers $\omega_k$ and especially the Euler characteristics 
$\chi=\omega_1$ are {\bf ring homomorphisms} from the strong ring to the integers. 
This algebraic property already makes the Wu characteristics important quantities 
to study. 

\paragraph{}
Euler characteristic is distinguished in having additionally the property of being 
invariant under homotopy transformations. It is unique among linear valuations 
with this property. The lack of homotopy invariance for the Wu characteristic numbers 
for $k \geq 2$ is not a handicap, in contrary: it allows for finer invariants: 
while for all {\bf trees} for example, the Euler characteristic is $1$, the Wu characteristic 
$\omega(G)=\omega_2(G)$ is more interesting: it is $\omega(G)=5-4n+M(G)$, 
where $n=\sum_{v \in V} 1$ is the vertex cardinality and $M$ is the 
{\bf Zagreb index} $M(G)=\sum_v {\rm deg}(v)^2$. 
This {\bf tree formula} relating the  Wu characteristic with the Zagreb index has been pointed 
out to us by Tamas Reti. 

\paragraph{}
There are other natural operations on simplicial complexes. 
One is the {\bf Zykov addition} which is in the continuum known as the {\bf join}. 
Take the disjoint union $G,H$ of the graphs
and additionally connect any pair of vertices from different graphs. The generating
function $f_G(x)=1+\sum_{k=0} v_k x^{k+1}$ satisfies $f_{G+H}(x)=f_G(x) f_H(x)$ and
$\chi(G)=1-f_G(-1)$ from which one can deduce that $i(G)=1-\chi(G)$ is multiplicative:
$i(G+H)=i(G) i(H)$. Evaluating $f$ at $x=1$ instead shows that the total number of
simplices $f_G(x)-1 = \sum_k v_k$ is multiplicative. We don't use the join much here. 

\paragraph{}
Looking at finite abstract simplicial complexes modulo Barycentric refinements already captures
some kind topology. But Barycentric refinements are still rather rigid and do not include
all notions of topology in which one wants to have a sense of rubber geometry. One especially
wants to have dimension preserved.  We introduced 
therefore a notion of homeomorphism in \cite{KnillTopology}. It uses
both homotopy and dimension. The Barycentric refinement of a graph $G$ is homeomorphic
to $G$ in that notion. This notion of homeomorphism could be ported over to simplicial complexes
because the Barycentric refinement of an abstract simplicial complex is a Whitney complex of a
graph. From a topological perspective there is no loss of generality by looking at graphs rather 
than simplicial complexes. We believe that the interaction cohomology considered here is the 
same for homeomorphic graphs but this is not yet proven.

\paragraph{}
As topological spaces come in various forms, notably manifolds or varieties, one can also define classes
of simplicial complexes or classes of graphs. 
In the continuum, one has historically started to investigate first very regular structures
like Euclidean spaces, surfaces, then introduced manifolds, varieties, schemes etc.
In the discrete, simplicial complexes can be seen as general structure and subclasses play the role of manifolds,
or varieties. In this respect, Euler characteristic has a differential topological flair,
while higher Wu characteristics are more differential topological because second order difference operations 
appear in the curvatures. 

\paragraph{}
Complexes for which all unit spheres are $(d-1)$ spheres are discrete $d$-manifolds. The definitions
are inductive: the empty complex $0$ is the $(-1)$ sphere, a $d$-complex is a complex for which every unit sphere $S(v)$
is a $(d-1)$-sphere. A $d$-sphere is a $d$-graph which after removing one vertex becomes
contractible. A graph $G=(V,E)$ is contractible, if there is a vertex $v$ such that
both $S(v)$ and the graph generated by $V \setminus \{v\}$ is contractible. We have seen that
for $d$-graphs, all Wu characteristics are the same. Wu characteristic is interesting for 
discrete varieties like the figure $8$ graph. Discrete varieties of dimension $d$ are complexes for which 
every unit sphere is a $(d-1)$-complex. There can now be singularities, as the unit spheres are not spheres
any more. In the figure 8 graph for example, the unit sphere of the origin is $4$, the $0$-dimensional graph 
with 4 points. (When writing $4$, we identify zero dimensional complexes with $\N$.)

\section{Calculus} 

\paragraph{}
Attaching an {\bf orientation} to a simplicial complex $G$ means choosing a base permutation on 
each element of $G$. The choice of orientation is mostly irrelevant and corresponds to a choice of basis. 
Calculus on complexes deals with functions $F$ on $G$ which change sign
if the orientation of a simplex is changed.  Functions on the set of $k$-dimensional simplices are
the {\bf $k$-forms}. It produces {\bf simplicial cohomology} with the exterior derivative 
defined as $dF(x) = F(\delta x)$, where $\delta x$ is the ordered boundary complex 
of a simplex (which is not a simplicial complex in general). While
$G=\delta K_3 = \{ \{1,2\},\{2,3\},\{3,1\},\{1\},\{2\},\{3\} \}$ is a 1-sphere
(because $G_1=C_6$ being a circular graph with 6 vertices), the boundary of a star
graph is not even a simplicial complex any more. 
Summing a $k$-forms over a collection $A$ of $k$-simplices gives {\bf integration} 
$\int_A F \; dx$.

\paragraph{}
Differentiation and integration are related by {\bf Stokes theorem} $\int_G dF \; dx = \int_{\delta G}$
which follows by extending integration linearly from 
simplices to simplicial complexes and chain. The boundary of a simplicial complex is not a simplicial complex any 
more (for the star complex $G=\{ \{1,2\},\{1,3\},\{1,4\},\{1\},\{2\},\{3\},\{4\} \}$ for example, 
$\delta G=\{ -3 \{1\}, \{2\}, \{3\}, \{4\} \}$ is only a chain). It is only for orientable discrete
manifolds with boundary that we can achieve $\delta G$ to be a simplicial complex and actually a 
closed discrete manifold without boundary, leading to the classical theorems in calculus. 

\paragraph{}
In calculus we are used to split up the exterior derivatives $d_p: \Lambda^p \to \Lambda^{p+1}$. 
It is more convenient to just stick them together to one single derivative $d: \Lambda \to \Lambda$. 
For a simplicial complex $G$ with $n$ elements, 
the exterior derivative matrix $d$ on the union of all forms together is given by a $n \times n$ matrix.
The corresponding {\bf Dirac operator} $D=d+d^*$ defines the {\bf Hodge Laplacian} $L=D^2$ 
which is block diagonal, each block $L_k$ acting as a symmetric Laplacian on $k$-forms. 

\paragraph{}
The operator $D$ not only simplifies the notation; it is also useful: 
the solution to the wave equation $u_{tt} = L u$ on the graph for example is given by an explicit 
{\bf d'Alembert solution formula} $u(t) = \cos(D t) u(0) +  \sin(Dt) D^{-1} u'(0)$, where $D^{-1}$ is the 
pseudo inverse of $D$. Classically, in the case 
$D=-i \partial_x$, this is the {\bf Taylor formula} $u(t) = e^{i D t} u(0) = \exp(\partial_x t) u$. In the case
of a line graph, this is the {\bf Newton-Gregory formula} (= discrete Taylor formula) 
$f(x+t)=\sum_k f^{k}(x) t^k/k!$, where $t^k$ are
deformed polynomials. On a general graph, it is a Feynman-Kac path integral, which is in the discrete a 
finite weighted sum over all possible paths of length $t$. (See \cite{ArchimedesFunctions} for an exhibit). 

\paragraph{}
A general solution of the {\bf Schr\"odinger evolution} $e^{D t} \psi(0)$ with complex valued
initial condition $\psi(0)=u(0)+D^{-1} u'(0)$ is identical to the real wave evolution which so
because in an abstract way Lorentz invariant in space-time (where time is continuous). 
Also the heat evolution $e^{-t L} u(0)$ is of practical use, as it produces
a numerical method to find harmonic forms representing cohomology classes. The reason is
that under the evolution, all the eigenmodes of the positive eigenvalues get damped exponentially fast,
so that the attractor on each space of $k$-forms is the space of harmonic $k$-vectors.

\paragraph{}
The first block of $L$, the matrix $L_0$ acting on scalar functions is known under the name
{\bf Kirchhoff matrix}. The eigenspace $E_k(0)$ of the eigenvalue $0$ of
$L_k$ defines the linear space of harmonic forms whose dimension agrees with the dimension of 
the cohomology group $H^k(G)$. By McKean-Singer, for any $t$, the super trace of $e^{-t L}$ is the 
Euler characteristic $\chi(G)$ which is $\chi_1 \cdot v(G)$, where $\chi_1=(1,-1,1,\dots)$
and the {\bf $f$-vector} $v(G)=(v_0(G),v_1(G),\dots )$ which encodes the number $v_k(G)$ 
of complete subgraphs of dimension $k$. 
In general, $b_k(G) = {\rm dim}(H^k(G))$ is equal to the number of zero eigenvalues of 
$L_k$. This is part of discrete Hodge theory. We also know that the collection of the nonzero
eigenvalues of all Fermionic Laplacians $L_{2k+1}$ is equal to the collection of the nonzero
eigenvalues of all Bosonic Laplacians $L_{2k}$. This {\bf super symmetry} leads to 
the {\bf McKean-Singer formula}
$$  \chi(G) = {\rm str} \exp(-t L) \; . $$ 
For $t=0$, this is the combinatorial definition $\chi(G)={\rm str}(1)$ of the Euler characteristic. For 
$t \to \infty$, it gives the super sum of the Betti numbers. 
If $G=(V,E)$ is a finite simple graph, let $v = (v_0, \dots,v_d)$ denote the 
{\bf $f$-vector} of $G$, where $v_i$ is the number of $i$-dimensional subgraphs of $G$. 

\paragraph{}
If $v=(v_0,v_1, \dots, v_r)$ is the $f$-vector of $G$, then 
$n=v_0+\cdots + v_r=v(G) \cdot (1, \dots, 1)$ simplicies in $G$. 
The number $\chi(G)=v(G) \cdot (1,-1,\dots, \pm 1)$ 
is the {\bf Euler characteristic} of $G$. We know $\chi(G \times H) = \chi(G) \chi(H)$
and also $\chi(G)=\chi(G_1)$, where $G_1$ is the Barycentric refinement.
The graph $G_1$ is a subgraph of the {\bf connection graph}, where two simplices 
are connected if they intersect. Let $V(G)$ be the {\bf $f$-matrix} of $G$ counting in the entry 
$V_{ij}$ how many $i$-simplices intersect with $j$-simplices. 
More generally, we have a symmetric $V_{i_1,\dots,i_k}$ which 
tells, how many $i_1$ simplices, $i_2$ simplices to $i_k$ simplices intersect. 
For any $k$, the Euler characteristic satisfies the {\bf Euler-Poincar\'e formula}
$\chi_k(G) = b_0(G)-b_1(G)+b_2(G) \cdots$, 
where $b_j(G) = {\rm dim}(H^j(G))$ is the $k$'th {\bf Betti number} of $G$ \cite{Eckmann1}. 
We will below generalize this Euler-Poincar\'e formula from 
$\chi(G) = \omega_1(G)$ to all the Wu characteristics $\omega_k(G)$. 

\section{Interaction cohomology}

\paragraph{}
In this section we define a cohomology compatible with the Wu characteristic 
$$  \omega(G) = \sum_{x \sim y} \omega(x) \omega(y)  \; , $$ 
a sum taken over all intersecting simplices in $G$, where $\omega(x) = (-1)^{{\rm dim}(x)}$. 
Wu characteristic is a quadratic version
of the Euler characteristic $\chi(G) = \sum_x \omega(x)$. Using $\chi_1=(1,-1,1,\dots)$ and the
$f$-vector $v(G)$ and $f$-matrix $V(G)$ we can write $\chi(G) = \chi_1 \cdot v(G)$ and 
$\omega(G) = \chi_1 \cdot (V(G) \cdot \chi_1)$.
Cubic and higher order Wu characteristics are defined in a similar way. 
All these characteristic numbers are multi-linear valuations
in the sense of \cite{valuation}. One can also generalize the setup and define
$$ \omega(G,H) = \sum_{x \in G, y \in H, x \sim y} \omega(x) \omega(y) \; . $$
It is a {\bf Wu intersection number}. 
This can be especially interesting if $H$ is a geometric object embedded in $G$. 

\paragraph{}
{\bf Example:} Let $G$ be the complete graph
$G=K_n$. By subtracting the number of all non-intersecting pairs of simplices
from all pairs of simplices, we get the $f$-matrix 
$V_{ij} = \B{n}{i} \B{n}{j}-\frac{n!}{i! j! (n-i-j)!}$
which gives $\omega(K_{d+1})= \sum_{i,j} V_{ij} (-1)^i (-1)^j =  (-1)^{d}$.
If $f_G$ is the Stanley-Reisner polynomial of $G$, then 
$\omega(G)=f_G^2(-1,\dots,-1) - f^2_G(-1,\dots,-1)$. We can therefore write 
$$  \chi(G) = \sum_x \omega(x), \omega(G) = \sum_{x \sim y} \omega(x) \omega(y)  \; . $$
The higher order versions $\omega_k(G) = \sum_{x_1 \sim \cdots \sim x_k} \prod_{j=1}^k \omega_j$
are computed in the same way. We have $\omega_k(K_n)=\chi(K_n)$ for odd $n$ and $\omega_k(K_n)_=1$
for even $n$. 

\paragraph{}
When looking at valuations on {\bf oriented simplices}, we require $F(A)=-F(\overline{A})$,
where $\overline{A}$ denotes the simplex with opposite orientation. We think of these 
valuations as {\bf discrete differential forms}. Multi-linear versions of such valuations 
play the role of quadratic or higher order differential forms. Given a $k$-linear form $F$, 
it leads to the {\bf exterior derivative} 
$$ dF(x_1, \dots, x_k) = \sum_j (-1)^{\sum_{i<j} {\rm dim}(x_i)} F(A_1, \dots, dA_j, \dots,A_k) \; , $$ 
where $dA$ is the boundary chain of $G$. In the case $k=1$, this is the usual exterior derivative
$$ dF(x) = F(dx)  \; . $$
For $k=2$, it is 
$$ dF(x,y) = F(dx,y)+(-1)^{{\rm dim}(x)} F(x,dy) \; . $$
Because $ddA=0$, the operation $d$ can serve as an exterior derivative: we get a {\bf cohomology}
$$ H^p(G) = {\rm ker}(d_p)/{\rm im}(d_{p-1}) \; . $$
We call it the {\bf intersection cohomology} or the {\bf quadratic intersection cohomology} or 
simply the {\bf quadratic simplicial cohomology}.  It deals with 
interacting $k$-tuples of simplices. All these cohomologies $H^p_k(G_1, \dots, G_k)$ generalize 
simplicial cohomology, which is the case $k=1$. 

\paragraph{}
Define $b_{k,p}(G)$ to be the dimension of the $p$'th cohomology group $H^p_k(G)$. 
It is the $p$'th {\bf Betti number on the level $k$}. More generally, if 
$G_1, \dots, G_k$ are simplicial complexes, define $b_p(G_1,\dots, G_k)$ to be the 
dimension of the cohomology group $H^p(G_1,\dots, G_k)$. 

\paragraph{}
Functions on simplex $k$-tuples of pairwise intersecting simplices satisfying 
$$ \sum_{i=1}^k {\rm dim}(x_i) = p $$
are called {\bf $p$-forms} $f$ in the $k$-interaction cohomology. As for simplicial 
cohomology, if $df=0$, we have {\bf cocycles} and if $f=dg$ we have 
{\bf coboundaries}.  \\

{\bf Examples.} \\
{\bf 1)} If $G$ is the octahedron graph, then, for $k=1$ there are only $0$-forms, 
$1$-forms and $2$-forms. The Betti vector for simpicial cohomology is 
$b(G)=(1,0,1)$ which super sums to $\chi(G)=2$. 
For $k=2$, there are only non-trivial cohomology classes for 
$2$-forms, $3$-forms as well as $4$-forms. 
The Betti vector for quadratic interaction 
cohomology is $b_2(G) = (0,0,1,0,1)$. The same Betti vectors appear for the 
icosahedron or Barycentric refinements of these 2-spheres. For the Barycentric refinement
of the icosahedron, the Dirac operator is already a $9602 \times 9602$ matrix. \\

{\bf 2)} The following example shows how can use cohomology to distinguish topologically 
not-equivalent complexes, which have equivalent simplicial cohomology.
The pair has been found by computer search. 
Let $G$ be the Whitney complex of of the graph with edges 
$E=\{$ (1, 2), (1, 3), (1, 6), (1, 7), (1, 8), (1, 9), (2, 3), (2, 6), (2, 8), (2, 9),
(3, 7), (3, 9), (4, 5), (4, 6), (4, 8), (4, 9), (5, 6), (7, 8)
$\}$. Let $H$ be the Whitney complex with edges $E=\{$ 
(1, 2), (1, 3), (1, 4), (1, 5), (2, 4), (2, 5), (2, 6), (2, 8), (2, 9), (3, 4),
(3, 5), (3, 9), (4, 6), (4, 8), (6, 8), (6, 9), (7, 8), (7, 9)
$\}$, the f-vector is $(9, 18, 9, 1)$ and the f-matrix is
$[[9, 36, 27, 4], [36, 136, 96, 13], [27, 96, 65, 8], [4, 13, 8, 1]]$.
The Betti vector is the same $(1, 2, 0, 0)$ so
linear simplicial cohomology of $G$ and $H$ are the same. 
The quadratic Wu cohomology of $G$ and $H$ however differ: it is
$b_2(G) = (0, 0, 6, 3, 0, 0, 0)$ and $b_2(H) = (0, 0, 4, 1, 0, 0, 0)$ for $H$. \\

{\bf 3)} This is an interesting prototype computation which we have written down in 2016
(see the appendix). For the {\bf cylinder} $G$ given by the Whitney complex of the graph 
with edge set $E\{$ (1, 2), (2, 3), (3, 4), (4, 1), (5, 6), (6, 7), (7, 8), (8, 5),
(1, 5), (5, 2), (2, 6), (6, 3), (3, 7), (7, 4), (4, 8), (8, 1)
$\}$ and the {\bf M\"obius strip} given as the Whitney complex of the graph with 
edge set $E=\{$
(1, 2), (2, 3), (3, 4), (4, 5), (5, 6), (6, 7), (7, 1), (1, 4),
(2, 5), (3, 6), (4, 7), (5, 1), (6, 2), (7, 3) $\}$
we have identical simplicial cohomology but different quadratic cohomology. The cylinder has
the Betti vector $(0, 0, 1, 1, 0)$, the M\"obius strip has the trivial quadratic Betti vector
$(0,0,0,0,0)$. In the appendix computation, all matrices are explicitely written down. 
An other appendix gives the code. 

\paragraph{}
In the same way as in the case $k=1$, we have a cohomological 
reformulation of Wu-characteristic. This Euler-Poincar\'e formula
holds whenever we have a cohomology given by a concrete exterior
derivative \cite{Eckmann1,Edelsbrunner,brouwergraph}:

\begin{thm}[Euler-Poincar\'e]
For any $k \geq 1$, we have 
$\omega_k(G) = \sum_{p=0}^{\infty} (-1)^p {\rm dim}(H_k^p(G))$. 
\end{thm}
\begin{proof}
Here is the direct linear algebra proof:
Assume the vector space $C_m$ of $m$-forms on $G$ has dimension $v_m$.
Denote by $Z_m= {\rm ker}(d)$ the kernel of of $d$ of dimension $z_m$.
Let $R_m={\rm ran}(d)$ the range of $d$ of dimension $r_m$.
From the rank-nullity theorem $\dim(\ker(d_m)) + \dim({\rm ran}(d_m))=v_m$, we get
$z_m = v_m-r_m$. From the definition of the cohomology groups 
$H^m(G)=Z_m(G)/R_{m-1}(G)$ of dimension $b_m$ we get 
$ b_m = z_m-r_{m-1}$. Adding these two equations gives 
$$ v_m-b_m = r_{m-1}+r_m  \; . $$
Summing over $m$ (using $r_{-1}=0, r_m=0$ for $m>m_0$)
$$ \sum_{m=0}^{\infty} (-1)^m (v_m-b_m) = \sum_{m=0}^{\infty} (-1)^m (r_{m-1}+r_m) = 0 \;  $$
which implies $\sum_{m=0}^{\infty} (-1)^m v_m = \sum_{m=0}^{\infty} (-1)^m b_m$.
\end{proof}

\paragraph{}
As we will see below, there is a more elegant deformation proof. It uses the fact that
$D$ produces a symmetry between non-zero eigenvalues of $D^2$ restricted to even forms
and non-zero eigenvalues of $L=D^2$ restricted to odd forms. This implies ${\rm str}(L^n)=0$
for all $n>0$ and so the Mc-Kean-Singer relations ${\rm str}(\exp(-t L))=0$. For $t=0$, this
is the definition of $\omega_k$. In the limit $t \to \infty$ we get to the projection onto 
Harmonic forms, where 
$$  \lim_{t \to \infty} {\rm str}(\exp(-i t L)) = \sum_{p=0}^{\infty} (-1)^p {\rm dim}(H_k^p(G)) \; . $$

\section{Example computations}

\paragraph{}
For $G=\{\{1,2\},\{2,3\},\{1\},\{2\},\{3\} \}$, there are 15 interacting pairs $x \sim y$ of 
simplices:
$$  \left\{
                  \begin{array}{cc}
                   \{3\} & \{3\} \\
                   \{2\} & \{2\} \\
                   \{1\} & \{1\} \\
                   \{3\} & \{2,3\} \\
                   \{2\} & \{2,3\} \\
                   \{2\} & \{1,2\} \\
                   \{1\} & \{1,2\} \\
                   \{2,3\} & \{3\} \\
                   \{2,3\} & \{2\} \\
                   \{1,2\} & \{2\} \\
                   \{1,2\} & \{1\} \\
                   \{2,3\} & \{2,3\} \\
                   \{2,3\} & \{1,2\} \\
                   \{1,2\} & \{2,3\} \\
                   \{1,2\} & \{1,2\} \\
                  \end{array}
                  \right\} \;  . $$
The Hodge Laplacian for the Wu characteristic $\omega(G)=\omega_2(G,G)$ is 
$$ L=D^2 = \left[
                 \begin{array}{ccccccccccccccc}
                  2 & 0 & 0 & 0 & 0 & 0 & 0 & 0 & 0 & 0 & 0 & 0 & 0 & 0 & 0 \\
                  0 & 4 & 0 & 0 & 0 & 0 & 0 & 0 & 0 & 0 & 0 & 0 & 0 & 0 & 0 \\
                  0 & 0 & 2 & 0 & 0 & 0 & 0 & 0 & 0 & 0 & 0 & 0 & 0 & 0 & 0 \\
                  0 & 0 & 0 & 2 & -1 & 0 & 0 & 0 & -1 & 0 & 0 & 0 & 0 & 0 & 0 \\
                  0 & 0 & 0 & -1 & 3 & -1 & 0 & -1 & 0 & 0 & 0 & 0 & 0 & 0 & 0 \\
                  0 & 0 & 0 & 0 & -1 & 3 & -1 & 0 & 0 & 0 & -1 & 0 & 0 & 0 & 0 \\
                  0 & 0 & 0 & 0 & 0 & -1 & 2 & 0 & 0 & -1 & 0 & 0 & 0 & 0 & 0 \\
                  0 & 0 & 0 & 0 & -1 & 0 & 0 & 2 & -1 & 0 & 0 & 0 & 0 & 0 & 0 \\
                  0 & 0 & 0 & -1 & 0 & 0 & 0 & -1 & 3 & -1 & 0 & 0 & 0 & 0 & 0 \\
                  0 & 0 & 0 & 0 & 0 & 0 & -1 & 0 & -1 & 3 & -1 & 0 & 0 & 0 & 0 \\
                  0 & 0 & 0 & 0 & 0 & -1 & 0 & 0 & 0 & -1 & 2 & 0 & 0 & 0 & 0 \\
                  0 & 0 & 0 & 0 & 0 & 0 & 0 & 0 & 0 & 0 & 0 & 4 & -1 & -1 & 0 \\
                  0 & 0 & 0 & 0 & 0 & 0 & 0 & 0 & 0 & 0 & 0 & -1 & 2 & 0 & -1 \\
                  0 & 0 & 0 & 0 & 0 & 0 & 0 & 0 & 0 & 0 & 0 & -1 & 0 & 2 & -1 \\
                  0 & 0 & 0 & 0 & 0 & 0 & 0 & 0 & 0 & 0 & 0 & 0 & -1 & -1 & 4 \\
                 \end{array}
                 \right] \; .$$
It consists of three blocks: a $3 \times 3$ block $L_0$, a 
$8 \times 8$ block $L_1$ and a $4 \times 4$ block $L_2$. 
We have $b=(b_0,b_1,b_2,\dots)=(0,1,0,\dots)$ and Wu characteristic 
$\omega(G) = -1$. Similarly as we can compute
the Euler characteristic $\chi(G)=2-3=-1$ from the $f$-vector $v=(3,2)$ as $v \cdot e$ with 
$e=(1,-1)$, we can compute the Wu characteristic from the $f$-matrix 
$V(G)=\left[ \begin{array}{cc} 3 & 4 \\ 4 & 4 \\ \end{array} \right]$ by 
$\omega(G) = e \cdot (V(G) \cdot e)$. The only non-vanishing cohomology is $H^1(G,G)$. \\

Let us look at the case, where $H=\{ \{1\} \}$ is a boundary point. Now, the interacting simplices are
$\left\{ \begin{array}{cc} \{1\} & \{1\} \\ \{1,2\} & \{1\} \\ \end{array} \right\}$. The Hodge operator
is the $2 \times 2$ matrix $L=I_2$ which has two $1 \times 1$ blocks $1$. The cohomology is trivial.
If $L=\{ \{2\} \}$ is the middle point, then there are 3 interacting simplices 
$\left\{ \begin{array}{cc} \{2\} & \{2\} \\ \{2,3\} & \{2\} \\ \{1,2\} & \{2\} \\ \end{array} \right\}$. 
And the Hodge operator $L(G,H)=D(G,H)^2$ is given by 
$$ L = \left[ \begin{array}{ccc} 2 & 0 & 0 \\ 0 & 1 & -1 \\ 0 & -1 & 1 \\ \end{array} \right]  \; . $$
The Betti vector is $b(G,H)=(0,1,0,\dots)$. A basis of the cohomology is the vector on $0+1$-forms
which is constant $1$. \\

\paragraph{}
If $G$ is a $2$-sphere given by an octahedron
$G=\{\{1\},\{2\},\{3\},\{4\},\{5\},\{6\},\{1,2\},\{1,3\},\{1,4\},\{1,5\},\{2,3\}
    ,\{2,4\},\{2,6\},\{3,5\},\{3,6\},\{4,5\},\{4,6\},\{5,6\},\{1,2,3\},\{1,2,4\},\{1,3,5\},\{1,
    4,5\},\{2,3,6\},\{2,4,6\},\{3,5,6\},\{4,5,6\}\}$. If $H=\{1\}$ is a single point, then
the interacting simplices are $\left\{
                   \begin{array}{cc}
                    \{1\} & \{1\} \\
                    \{1,5\} & \{1\} \\
                    \{1,4\} & \{1\} \\
                    \{1,3\} & \{1\} \\
                    \{1,2\} & \{1\} \\
                    \{1,4,5\} & \{1\} \\
                    \{1,3,5\} & \{1\} \\
                    \{1,2,4\} & \{1\} \\
                    \{1,2,3\} & \{1\} \\
                   \end{array}
                   \right\}$. The Hodge Laplacian to this pair $(G,H)$ is 
$$ L = \left[
                   \begin{array}{ccccccccc}
                    4 & 0 & 0 & 0 & 0 & 0 & 0 & 0 & 0 \\
                    0 & 3 & 0 & 0 & 1 & 0 & 0 & 0 & 0 \\
                    0 & 0 & 3 & 1 & 0 & 0 & 0 & 0 & 0 \\
                    0 & 0 & 1 & 3 & 0 & 0 & 0 & 0 & 0 \\
                    0 & 1 & 0 & 0 & 3 & 0 & 0 & 0 & 0 \\
                    0 & 0 & 0 & 0 & 0 & 2 & 1 & -1 & 0 \\
                    0 & 0 & 0 & 0 & 0 & 1 & 2 & 0 & -1 \\
                    0 & 0 & 0 & 0 & 0 & -1 & 0 & 2 & 1 \\
                    0 & 0 & 0 & 0 & 0 & 0 & -1 & 1 & 2 \\
                   \end{array}
                   \right] \; . $$
Its null-space of $L$ is spanned by a single vector $(0, 0, 0, 0, 0, -1, 1, -1, 1)$ so that
the Betti vector is $(0,0,1)$. \\

\paragraph{}
If $G$ is a $3$-sphere, obtained by suspending the octahedron, then $\chi(G)=\omega(G)=0$
and $b_{k=1}(G) = (1,0,0,1)$ and $b_{k=2}(G) =(0,0,0,1,0,0,1)$. Also for a $4$-sphere,
like a suspension of a $3$-sphere $G$, the simplicial cohomology has just shifted up
$b_{k=2}(G) = (0,0,0,0,1,0,0,0,1)$. 

\paragraph{}
Here is a table of some linear, quadratic and cubic cohomology computations: 

\begin{tabular}{|ll|ll|ll|l|} \hline
 $\chi=\omega_1$ & Betti  &  $\omega=\omega_2$& Betti      &$\omega_3$ & Betti            & Complex      \\ \hline
  1 &  (1)      &   1      &(1)                &1          & (1)                          & $K_1$ (point)   \\
  1 &  (1,0)    &  -1      &(0,1,0)            &1          & (0,0,1,0)                    & $K_2$ (1-ball) \\
  1 &  (1,0,0)  &   1      &(0,0,1,0,0)        &1          & (0,0,0,0,1,0,0)              & $K_3$ (triangle) \\
  1 &  (1,0,0,0)&  -1      &(0,0,0,1,0,0,0)    &1          & (0,0,0,0,0,0,1,0,0,0)        & $K_4$ (tetrahedron) \\
  0 &  (1,1)    &   0      &(0,1,1)            &0          & (0,0,1,1)                    & $C_4$ (circle)  \\
  2 &  (1,0,1)  &   2      &(0,0,1,0,1)        &2          & (0,0,0,0,1,0,1)              & Octahedron   \\
  2 &  (1,0,1)  &   2      &(0,0,1,0,1)        &2          & (0,0,0,0,1,0,1)              & Icosahedron  \\
  0 &  (1,0,0,1) &  0      &(0,0,0,1,0,0,1)    &0          & (0,0,0,0,0,0,1,0,0,1)        & 3-sphere     \\
  2 &  (1,0,0,0,1)& 2      &(0,0,0,0,1,0,0,0,1)&2          & (0,0,0,0,0,0,0,0,1,0,0,0,1)  & 4-sphere     \\
  1 &  (1,0,0)  &   1      &(0,0,1,0,0)        &1          & (0,0,0,0,1,0,0)              & 2-ball       \\
  1 &  (1,0)    &   1      &(0,0,1)           &-5          & (0,0,0,5)                    & 3-star       \\
  1 &  (1,0)    &   5      &(0,0,5)          &-23          & (0,0,0,23)                   & 4-star       \\
  1 &  (1,0)    &  11      &(0,0,11)         &-59          & (0,0,0,59)                   & 5-star       \\
  1 &  (1,0)    &   1      &(0,0,0,0,1)       &25          & (0,0,0,0,0,0,25)             & 3-star x 3-star   \\
 -1 &  (1,2)    &   7      &(0,0,7)          &-25          & (0,0,0,25)                   & Figure 8     \\
 -2 &  (1,3)    &  22      &(0,0,22)        &-122          & (0,0,0,122)                  & 3 bouquet    \\
 -3 &  (1,4)    &  45      &(0,0,35)        &-339          & (0,0,0,339)                  & 4 bouquet    \\
 -4 &  (1,5)    &  76      &(0,0,76)        &-724          & (0,0,0,724)                  & 5 bouquet    \\
  1 &  (1,0)    &  3       &(0,0,3,0,0)       &-5          & (0,0,0,6,1,0,0)              & Rabbit       \\
  0 &  (1,1)    &  2       &(0,0,2,0,0)        &0          & (0,0,0,1,1,0,0)              & House        \\
 -4 &  (1,5)    & 20       &(0,0,20)         &-52          & (0,0,0,52)                   & Cube         \\
 -16 & (14,30)  &112       &(0,0,112)       &-400          & (0,0,0,400)                  & Tesseract    \\
  0  & (1,1,0)  &  0       &(0,0,0,0,0)        &0          & (0,0,0,0,1,1,0)              & Moebius strip   \\
  0  & (1,1,0)  &  0       &(0,0,1,1,0)        &0          & (0,0,0,0,1,1,0)              & Cylinder     \\
  1  & (1,0,0)  &  1       &(0,0,0,0,1)        &1          & (0,0,0,0,0,0,1)              & Projective plane  \\
  0  & (1,1,0)  &  0       &(0,0,0,1,1)        &0          & (0,0,0,0,0,1,1)              & Klein bottle \\ \hline
\end{tabular}

\paragraph{} Let us look at more examples of quadratic interaction cohomology $H^p(G,H)$
with two different complexes $G,H$. The cohomology groups $H^p(G,H)$ are trivial vector spaces
if $G \cap H = \emptyset$, as there are then no interaction. The Moebius strip is a rare 
example of trivial interaction cohomology $H^p(G,G)$. 

\begin{tabular}{|l|l|l|l|} \hline
 $\omega=\omega_2$ & Betti vector    & Complex G  & Complex H  \\ \hline
 $-1$              & (0,1,0)         & interval   &  interval  \\ \hline
 $-1$              & (0,1)           & interval   &  point inside \\ \hline
 $0$               & (0,0)           & interval   &  boundary point  \\ \hline
 $0$               & (0,1,1)         & circle     &  circle \\ \hline
 $-1$              & (0,1)           & circle     &  point  \\ \hline
 $-2$              & (0,2)           & circle     &  two points  \\ \hline
 $-1$              & (0,1,0)         & circle     &  sub interval \\ \hline
 $11$              & (0,0,11)        & star(5)    &  star(5)    \\ \hline
 $4$               & (0,4)           & star(5)    &  central point \\ \hline
 $0$               & (0,0)           & star(5)    &  boundary point\\ \hline
 $-1$              & (0,1,0)         & star(5)    &  radius interval  \\ \hline
 $0$               & (0,0,1,1)       & 2D disk       & Circle in interior \\ \hline
 $1$               & (0,0,1,0)       & 2D disk       & Circle touching boundary \\ \hline
 $1$               & (0,0,1,0)       & 2D disk       & 2D disc  \\ \hline
 $1$               & (0,0,1)         & 2D disk       & Point inside \\ \hline
 $0$               & (0,0,0)         & 2D disk       & Point at boundary \\ \hline
\end{tabular}

\begin{figure}
\scalebox{0.32}{\includegraphics{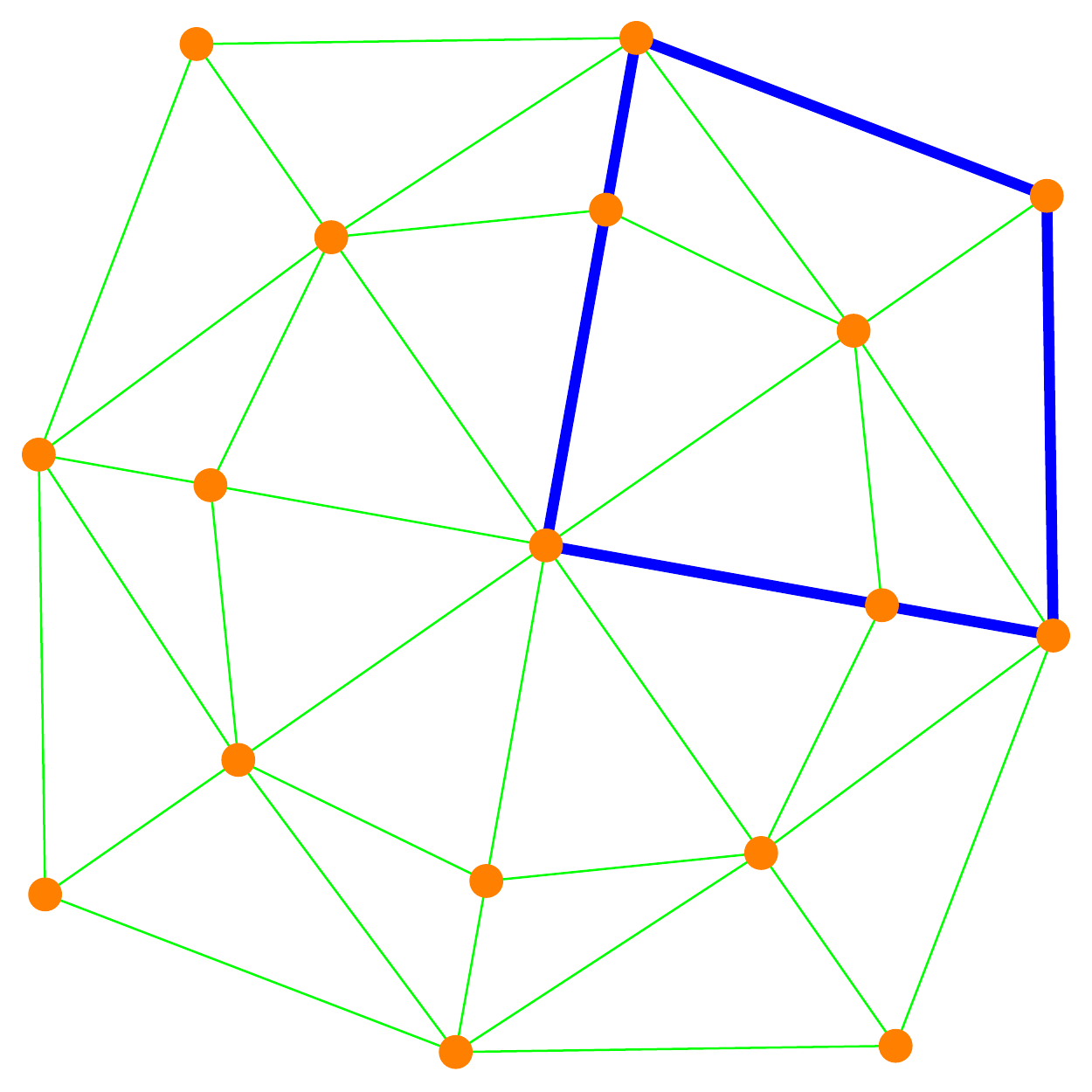}} 
\scalebox{0.32}{\includegraphics{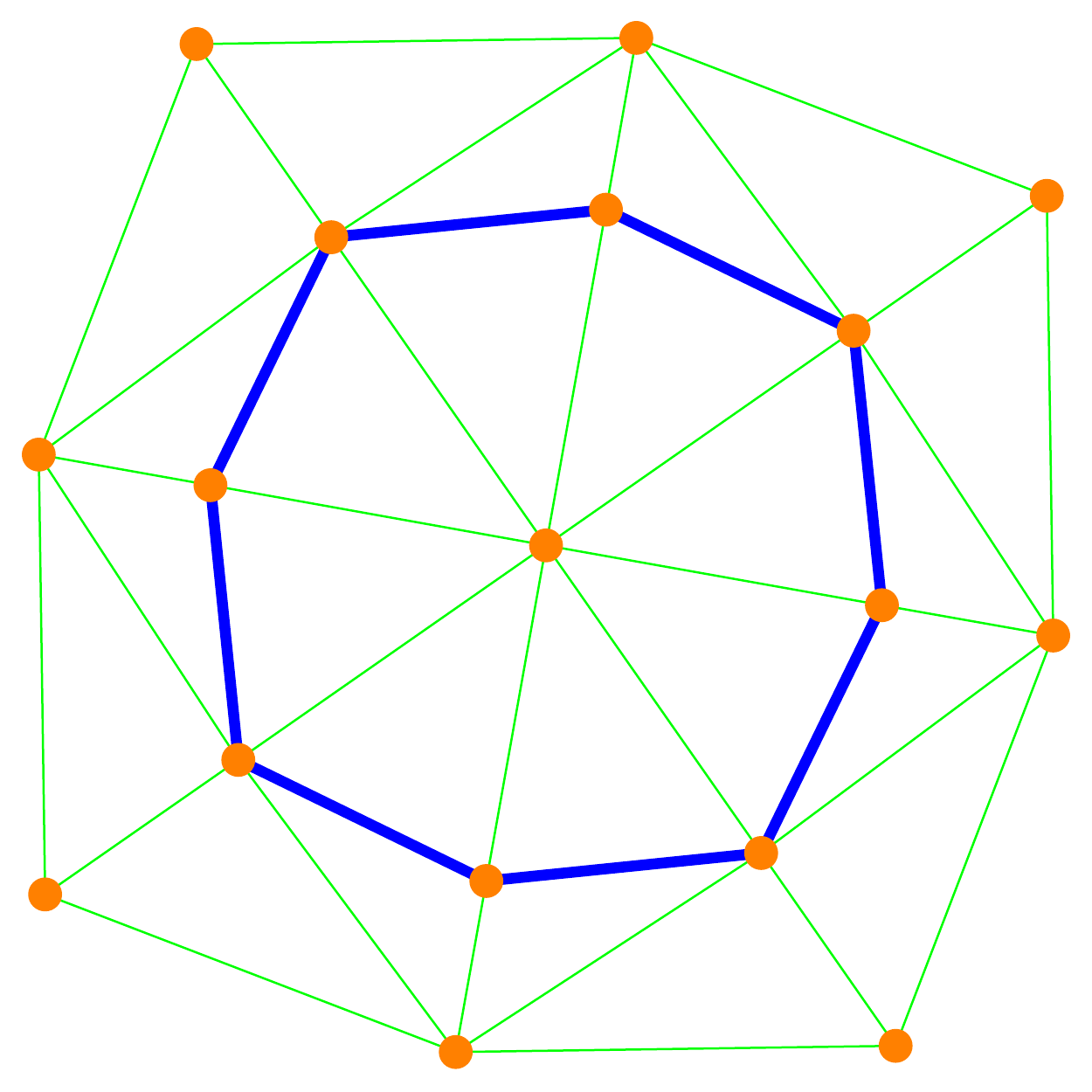}}
\scalebox{0.32}{\includegraphics{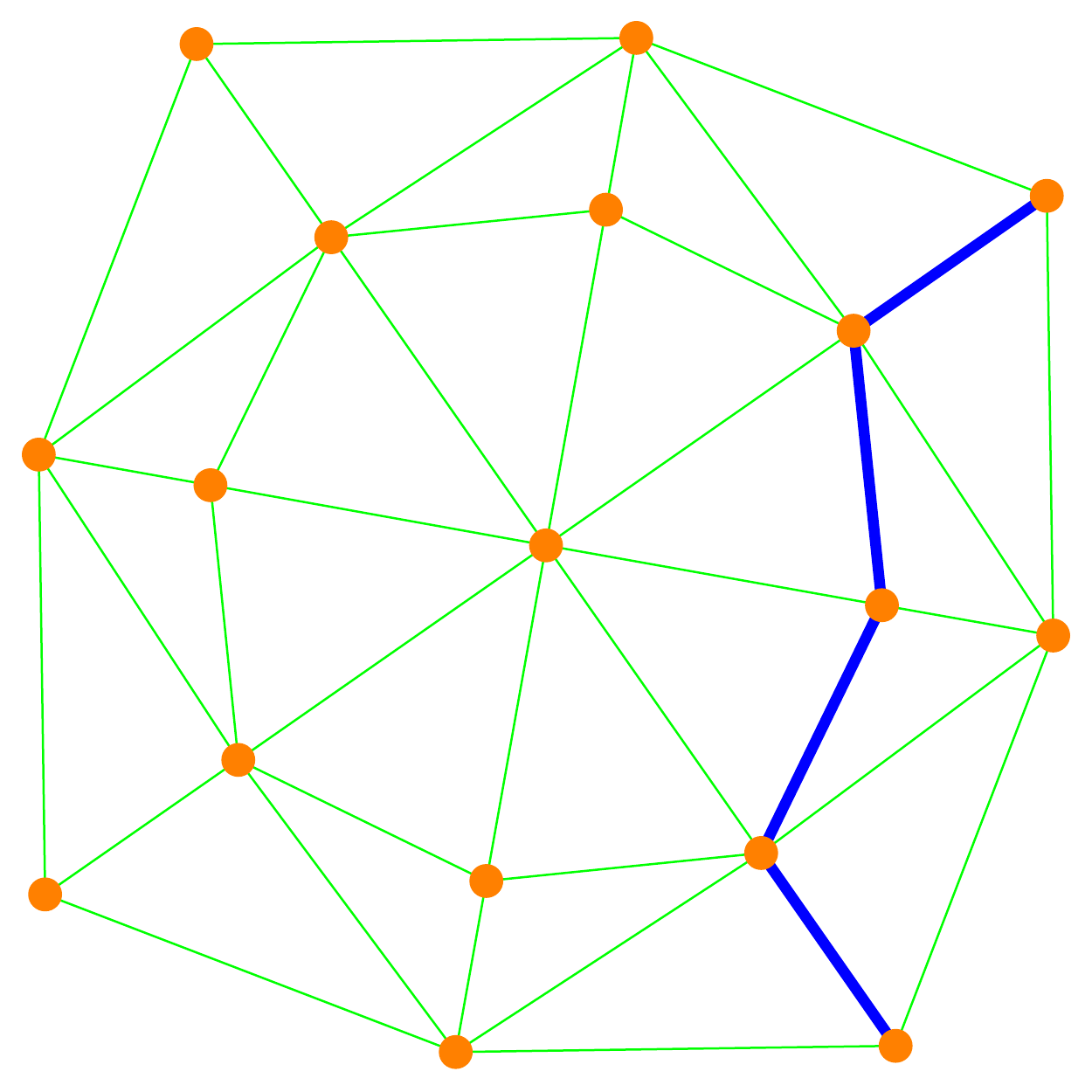}}
\caption{
\label{wheel}
The background complex $G$ is a Barycentric refinement of the wheel graph. 
The $k=2$ cohomology has the Betti vector $b(G)=(0,0,1,0,0)$.
To the very left and very right we see unit sphere subgraphs $H$ 
which touch the boundary. The Betti vector for the cohomology $H(G,H)$ 
is $b(G,H)=(0, 0, 1, 0)$ which matches with $\omega(G,H)=1$. In the middle we see
a circular subgraph $H$ which does not hit the boundary. In that case
the cohomology is $b(G,H)=(0, 0, 1, 1)$ leading to the Wu intersection number
$\omega(G,H)=0$.
}
\end{figure}

\paragraph{}
We can have seen examples, where $H$ was a subcomplex of $G$. 
It is also possible that $G,H$ just intersect. In that case, only a neighborhood 
of the intersection set matters. We have computed Wu characteristic of 
nicely intersecting one dimensional graphs. To analyze this, we only need to 
understand a single intersection cross, where the cohomology is $b(G,H) = (0,0,1)$. 
Now, if $G$ is a circle and $H$ is a circle and the two circles intersect nicely, 
we have $b(G,H) = (0,0,2)$ as there are two intersection points. 

\begin{figure}
\scalebox{0.32}{\includegraphics{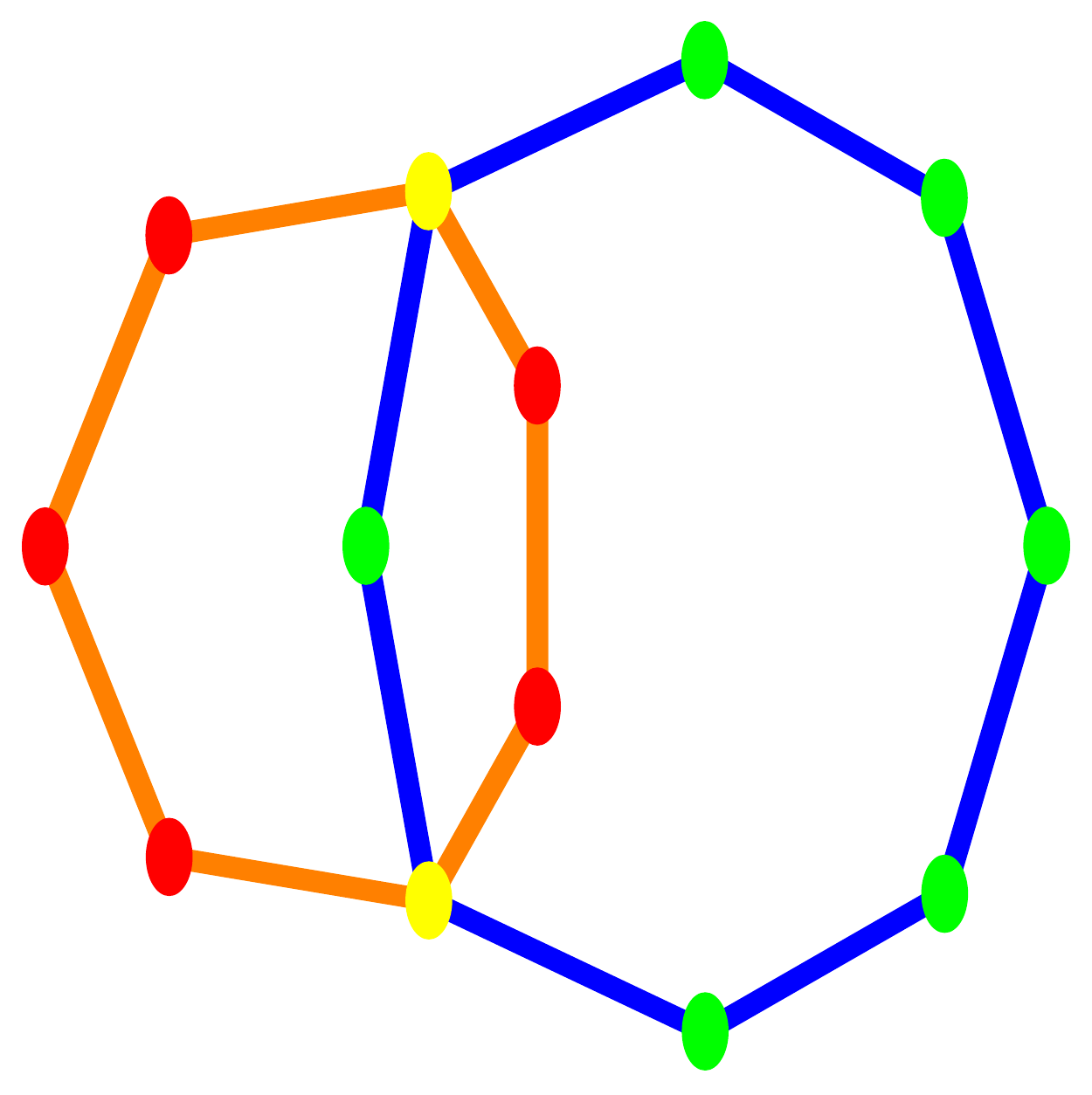}} 
\caption{
\label{twocircles}
The intersection of two circles $G,H$ has the Betti vector 
$b(G,H)=(0,0,2)$ and $\omega(G,H)=2$. 
}
\end{figure}

\section{Hodge relations}

\paragraph{}
The Euler-Poincar\'e relations equating combinatorial and cohomological expressions for $H_k^p(G)$ 
can be proven dynamically using the heat flow $\exp(-t L)$. 
To do so, it is important to build a basis of the
cohomology using harmonic forms. For $t=0$, we have by definition the Euler characteristic and for $t \to \infty$, the
super trace of the identity on harmonic forms. 
These Hodge relations play there an important role. It is good to revisit them in the
more general frame work of interaction cohomology. The proofs are the same than for $k=1$ however.

\paragraph{}
The Hodge theorem states that $H^p(G_1,\dots,G_k)$ is isomorphic to the space 
$E_0^p(G_1, \dots< G_k)$ of harmonic $p$-forms. 

\begin{lemma}
a) $L$ commutes with $d$ and $d^*$, \\
b) $L f = 0$ is equivalent to $df=0$ and $d^* f = 0$ \\
c) $f=dg$ and $h=d^*k$ are orthogonal. \\
d) ${\rm im}(d) \cup {\rm im}(d^*)$ span ${\rm im}(L)$.
\end{lemma}

\begin{proof}
a) follows from from $d^2=(d^*)^2=0$ so that $d^* L = L d^* = d^* d d^*$ and $d L = L d = d d^* d$.\\
b) if $Lf=0$ then $0=\langle f,Lf \rangle=\langle d^*f,d^f \rangle + \langle df,df \rangle$
shows that both $d^* f=0$ and $df=0$. The other direction is clear.  \\
c) $\langle f,h \rangle = \langle dg,d^*k \rangle 
   = \langle d dg,k \rangle = \langle g,d^* d^* k \rangle=0$.  \\
d) Write $g=Lf = d (d^* f) + d^* (d f)$.
\end{proof}

This shows that any $p$-form can be written as a sum of an exact, a co-exact and 
a harmonic $p$-form:

\begin{lemma}[Hodge decomposition]
There is an orthogonal decomposition
$\Omega = {\rm im}(L) + {\rm ker}(L) = {\rm im}(d) + {\rm im}(d^*) + {\rm ker}(L)$.
Any $p$-form $g$ can be written as $g= df + d^* h + k$, where $k$ is harmonic.
\end{lemma}

\begin{proof}
The matrix $L: \Omega(G_1, \dots, G_k) \to \Omega(G_1, \dots, G_k)$ is symmetric so that the
image and kernel are perpendicular. By the lemma, the image of $L$ splits into two
orthogonal components ${\rm im}(d)$ and ${\rm im}(d^*)$.
\end{proof}

\begin{thm}[Hodge-Weyl]
The dimension of the vector space of harmonic $p$-forms 
is equal to $b_p(G_1, \dots, G_k)$. Every cohomology class has a unique harmonic representative.
\end{thm}
\begin{proof}
If $Lf=0$, then $df=0$ and so $L f = d^* d f=0$. Given a closed $n$-form $g$
then $dg=0$ and the Hodge decomposition shows $g=df + k$ so that $g$
differs from a harmonic form only by $df$ and so that this harmonic form
is in the same cohomology class than $f$. 
\end{proof}

The {\bf harmonic $k$-form} represent therefore cohomology classes for the cohomologies
$H^p(G_1,\dots, G_k)$. 

\section{Lefschetz fixed point formula}

\paragraph{}
In the case of Euler characteristic {\bf Lefschetz number} of an automorphism $T$ of a complex $G$
is the super trace of the induced linear map on simplicial cohomology. If $T$ is
the identity, then this is the definition of the cohomological Euler characteristic.
We want to show here that the Lefschetz fixed point formula relates more generally
the Lefschetz number with the fixed simplices of $T$.
If $T$ is the identity, then every simplex is fixed and the Lefschetz fixed point
formula is the Euler-Poincar\'e formula. Since we have a cohomology for Wu characteristic,
we can look at the super trace of the map induced on the cohomology groups and get
so a new Lefschetz number.

\paragraph{}
We also have fixed pairs of simplices $(x,y)$ and an {\bf $2$-interaction index}
$$ i_T(x,y) = \omega(x) \omega(y) {\rm sign}(T|x) {\rm sign}(T|y) \; . $$
Now, the analogue quadratic theorem is that
$L_2(T) = \sum_{x \sim y \in G, T(x)=x,T(y)=y} i_{x,y}(T)$,
where the sum is over all pairs of simplices, where one is fixed by $T$.
This is the Lefschetz fixed point formula for Wu characteristic.
In the case of the identity, every pair of intersecting simplices is a fixed point
and the Lefschetz fixed point formula reduces to the Euler-Poincar\' formula for
Wu characteristic.
For example, for an orientation reversing involution on the circular graph $G=C_n$
with $n \geq 4$. The space $H^{0,1}(G)$ is 6-dimensional as is $H^{1,1}(G)$.

\paragraph{}
An {\bf automorphism} $T$ of $G$ is a permutation on the finite set $V$ of zero-dimensional 
elements in $G$ such that $T$ induces a permutation on $G$. More generally, one could look
also at endomorphisms, maps which map simplices to simplices and preserve the order, but since
we are in a finite frame work, we can always restrict to the attractor $\bigcap_n T^n(G)$ on
which $T$ then induces an automorphism. It is therefore almost no restriction of generality
to look at automorphisms.  

\paragraph{}
Let ${\rm Aut}(G)$ be the group of automorphisms of $G$. Any $T \in Aut(G)$ induces
a linear map on each interaction cohomology group $H^p_k(G)$. 
To realize that the matrix belonging to the transformation take an orthonormal 
basis for the kernel of $L_k$ and define $A_{ij}=f_i(T) f_j$. 
If $L|H(G)$ denotes the induced map on cohomology,
define the {\bf Lefschetz number} $\chi_T(G) = {\rm str}(L|H(G))$. In the special case, when $T$ is 
the identity $T=1$, the Lefschetz number is the Wu characteristic. 
In the case $k=2$, a pair of intersecting simplices $(x,y)$ is called a {\bf fixed point} of $T$, 
if $T(x)=x,T(y)=y$. More generally, we call a $k$-tuple $(x_1, \dots, x_k)$ of pairwise
intersecting simplices to be a fixed point of $T$ if $T(x_j)=x_j$. 
Define the {\bf index} of a fixed point as
$$  i_T(x_1,\dots,x_k) = \prod_{j=1}^k (-1)^{\rm dim(x_j)} {\rm sign}(T|x_j) \; . $$

\paragraph{}
In the case $k=1$, the index is 
$$  i_T(x)= (-1)^{{\rm dim}(x)} {\rm sign}(T|x) \; .$$
In the case $k=2$, the index is 
$$  i_T(x,y)= (-1)^{{\rm dim}(x) + {\rm dim}(y)} {\rm sign}(T|x) {\rm sign}(T|y)  \; .$$
Here is the Brouwer-Lefschetz-Hopf fixed point formula generalizing \cite{brouwergraph}. 

\begin{thm}[Brouwer-Lefschetz fixed point theorem]
$\chi_T(G) = \sum_{T(x=x} i_T(x)$
\end{thm}

The heat proof works as before.  For $t=0$ we get the right hand side, for $t \to \infty$
we get the left hand side.

\begin{proof}
Combine the unitary transformation $U: f \to f(T)$ 
with the heat flow $V(t) = \exp(-t L)$ to get the time-dependent operator $W(t) = U V(t)$
on the linear space of discrete differential forms.
For $t=0$, the super trace of $W_t$ is $\sum_{x \sim y} i_T(x,y)$.
McKean-Singer shows that the super trace of any positive power $L^n$ of $L$ is always zero. 
This is the {\bf super symmetry} meaning that
the union of the positive eigenvalues on the even dimensional part of the interaction Laplacian L is the same
than the union of the positive eigenvalues on the odd dimensional part of L.
For $t \to \infty$, the heat flow kills the eigenspace to the positive eigenvalues and pushes
the operator $W(t)$ towards a projection onto the kernel. The projection operator $P$ onto
harmonic interaction differential forms is the attractor.
Since by Hodge theory the kernel of the Laplacian L<sub>m</sub> on
$m$-interaction forms parametrized representatives of the m-th cohomology group, the transformation W(infinity)
is nothing else than the induced linear transformation on cohomology. Its super trace is by definition the
Lefschetz number.
The super symmetry relation ${\rm str}(L^n)$ for $n >0$ assures that the super trace of $\exp(-t L)$ 
does not depend on t. The heat flow (mentioned for example in \cite{DiscreteAtiyahSingerBott}) washes away the positive spectrum and leaves only the bare harmonic part responsible for the cohomology 
description. For $T=1$, the heat flow interpolates between the
combinatorial definition of the Wu characteristic and the cohomological definition of the
Wu characteristic, similarly as in the case of the Euler characteristic.
\end{proof}

\paragraph{}
Also, as in the case $k=1$, we can average the Lefschetz number 
over a subgroup $A$ of the automorphism group of $G$ to get $\omega_k(G/A)$, 
where $G/A$ is the quotient chain. If we look at the Barycentric refinement, we can 
achieve that the quotient chain is again a simplicial complex. 
These are generalizations of Riemann-Hurwitz type results.

\section{Barycentric invariance}

\paragraph{}
Almost by definition, the Wu characteristic $\omega(G)$ is multiplicative 
$$  \omega(G \times H) = \omega(G) \omega(H) \;   $$ 
because 
$$ \sum_{(x,y) \sim (u,v)} \omega(x) \omega(y) \omega(u) \omega(v) 
  = (\sum_{x \sim u} \omega(x) \omega(u) ) (\sum_{y \sim v} \omega(y) \omega(v) ) \; . $$
The same product formula holds for the other cases $\omega_k(G \times H) = \omega_K(G) \omega_K(H)$. 
This relation has used the Cartesian product $G \times H$ of simplicial complexes
which is not a simplicial complex any more. In \cite{KnillKuenneth} we have looked
at the product $(G \times H)_1$, the Barycentric refinement of $G \times H$ so that
we would remain in the class of simplicial complexes. But because calculus can be done equally
well on the ring itself, there is no need any more to insist on having
simplicial complexes. Indeed, in the strong ring, the simplicial complexes are the multiplicative
primes. 

\begin{propo}
All the Wu characteristics $\omega_k$ are ring homomorphisms from the strong ring to the integers. 
\end{propo}
\begin{proof}
The additivity $\omega_k(G+H) = \omega_k(G + H)$ for the disjoint union is clear as all $\omega_k$
are multi-linear valuations and $\omega_k(0)=0$. 
The compatibility with multiplication follows from $\omega_k(1)=1$ for all $k$ and 
\begin{eqnarray*}
 \omega_k(G \times H) = 
   \sum_{(x_1, \dots, x_k) \sim (y_1, \dots, y_k)} \omega(x_1)\cdots \omega(x_k) \omega(y_1) \cdots \omega(y_k)  \\
   (\sum_{(x_1, \dots, x_k)} \omega(x_1)\cdots \omega(x_k)) (\sum_{(y_1, \dots, y_k)} \omega(y_1) \cdots \omega(y_k) ) \\
   = \omega_k(G) \omega_k(H) 
\end{eqnarray*} 
\end{proof} 

\paragraph{}
Given a simplicial complex $G$ defined over a set $V$ of ``vertices", zero dimensional sets in $G$, 
the complex $G_1$ is defined over the set $V_1=G$. We call a vertex $x$ in $G_1$ a {\bf $p$-vertex},
if it had been a $p$-dimensional simplex in $G$. 

\begin{thm}[Barycentric Invariance]
$H^p_k(G) = H^p_k(G_1)$.
\end{thm}
\begin{proof}
There is an explicit map $\phi$ from the space of forms on $G$ to the space of forms in $G_1$. 
Lets look at it first in the case $k=1$. \\
Given a $p$-simplex $(x_1, \dots, x_p)$ in $G_1$, there is maximally one of the vertices where $F$ 
is non-zero. Define $F(x_1, \dots, x_p) = F(x_k)/p!$ to be this value. 
On the other hand, given a form $F$ on $G_1$, we get a form on $G$. Given a simplex $x$ in $G$
it contains $p!$ simplices of the same dimension in $G_1$. Now just sum up the $F$ values
over these simplices. As $\phi$ and its inverse commute with the exterior derivative $d$ in 
the sense $d(G_1) \phi = \phi d(G)$, we have also $L(G_1) \phi = \phi L(G)$. 
Harmonic forms are mapped into Harmonic forms.  \\
More generally, the Barycentric invariance works for $H^p(G_1, \dots, G_k)$, which is isomorphic to 
$H^p((G_1)_1, \dots, (G_k)_1)$, where $(G_j)_1$ is the Barycentric refinement of 
$G_j$. The map $\phi$ from forms in $\Omega(G_1, \dots, G_k)$ to forms in 
$\Omega( (G_1)_1, \dots, (G_k)_1)$ just distributes the value $F(x_1, \dots, x_k)$ on 
pairwise intersecting simplices $x_1, \dots, x_k$ equally to all $k$-tuples of 
pairwise intersecting simplices $y_1, \dots, y_k)$ with $y_j$ being simplices of the same dimension
than $x_j$ within $x_j$. The inverse of $\phi$ sums all values up. These averaging procedures
commute with $d$. 
\end{proof}

\paragraph{}
We don't know yet whether there is a more robust statement. We would 
think for example that if $H_1,H_2$ are two complexes, where one is a Barycentric 
refinement of the other and both are embedded in the same way in a larger graph $G$, 
then $H^p(G,H_1)=H^p(G,H_2)$. Experiments indicated that this is true. We also have
hopes that if $H_1,H_2$ are two different knots embedded into a three dimensional discrete
sphere $G$, then we could distinguish the embedding by computing $H^p(G,H_1)$ and 
$H^p(G,H_2)$. We also need to look at the 
cohomology of the complement $H^p(G,G \setminus H_1)$ and $H^p(G,G \setminus H_2)$. 
Classically, one has used simplicial cohomology $H^p(G \setminus H_1), H^p(G \setminus H_2)$ 
already. 

\section{Kuenneth formula}

\paragraph{}
The Barycentric refinement invariance shows that that the cohomology 
of $(G \times H)_1$ is the same than the cohomology of $G \times H$. 
So, we could for $k=1$ refer to \cite{KnillKuenneth}  to get 

\begin{thm}[Kuenneth]
$H_k^p(G \times H) = \oplus_{i+j=p} H_k^i(G) \oplus H_k^j(H)$.
\end{thm}

\paragraph{}
The proof follows from Hodge and the proposition stated below. 
Instead of generalizing \cite{KnillKuenneth} from $k=1$ to larger $k$, 
it is better to prove the relation directly first for the Cartesian 
product (which is not a simplicial complex), and then invoke the Barycentric refinement, to recover \cite{KnillKuenneth}. 
Now this is just Hodge. We have to show that for every $k$, every harmonic $p$-form $f$
can be written as a pair $(g,h)$, where $g$ is an $i$-form and $h$ is a $j$ form. 

\paragraph{}
K\"unneth could not be more direct: 

\begin{propo} 
A basis for $H^p(G,H)$ is given by $L$-harmonic $p$-forms $v$. A basis for $L(G,H)$-Harmonic $p$-forms
is given by all the vectors $v \oplus w$, where $v$ is $L(G)$-harmonic $k$-form and 
$w$ is $ L(H)$-harmonic $m$-form so that $k+m=p$. 
\end{propo}

\paragraph{}
In \cite{KnillKuenneth} the construction of a chain homotopy was more convoluted as we dealt with
the Barycentric refinements (as at time, we felt we needed to work with simplicial complexes). 
Separating the two things and establishing the Barycentric invariance separately is more elegant. 

\paragraph{}
For a given $k$ and ring element $G$, define the {\bf Poincar\'e polynomial}
$$ p_{G,k}(t)=\sum_{l=0} b_{k,l}(G) t^l \; . $$
If $k=1$, then $p_G(t)=p_{G,1}(t)$ is the usual Poincar\'e polynomial. 

\begin{coro}
The map $G \to p_{G,k}(t)$ is a ring homomorphism from the strong connection ring to $\Z[t]$.
\end{coro}

\begin{proof} 
If $d_i$ are the exterior derivatives on $G_i$, we can write them as partial exterior
derivatives on the product space. We get from
$d f(x,y) = d_1 f(x,y) + (-1)^{{\rm dim}(x)} d_2(f(x,y)$
$$ d^* d f = d_1^* d_1 f + (-1)^{{\rm dim}(x)} d_1^* d_2 f 
           + (-1)^{{\rm dim}} d_2^* d_1 f + d_2^* d_2 f \; ,  $$
$$ d d^* f = d_1 d_1^* f + (-1)^{{\rm dim}(x)} d_1 d_2^* f 
           + (-1)^{{\rm dim}(x)} d_2 d_1^* f + d_2 d_2^* f \; . $$
Therefore $H f = H_1 f + H_2 f + (-1)^{{\rm dim}(x)} (d_1^*d_2 + d_1 d_2^* + d_2^* d_1 + d_2 d_1^*) f(x,y) )$.
Since Hodge gives an orthogonal decomposition
$$  {\rm im}(d_i),{\rm im}(d_i^*), {\rm ker}(H_i) = {\rm ker}(d_i) \cap {\rm ker}(d_i^*) \; , $$
there is a basis in which $H(v,w) = (H(G_1)(v), H(G_2)(w))$.
Every kernel element can be written as $(v,w)$, where $v$ is in the kernel of $H_1$ and $w$ 
is in the kernel of $H_2$.
\end{proof} 

\paragraph{}
In the special case $k=1$, the Kuenneth formula would follow also from \cite{KnillKuenneth}
because the product is $(G \times H)_1$ and the cohomology of the Barycentric refinement is the same.
It follows that the {\bf Euler-Poincar\'e formula} holds in general for elements $G$ in the
ring: the cohomological Euler characteristic $\sum_{k=0}^{\infty} b_k(G) (-1)^k$ is equal to
the combinatorial Euler characteristic $\sum_{k=0}^{\infty} v_k(G) (-1)^k$, where
$(v_0,v_1, \dots)$ is the $f$-vector of $G$. 
Note that the Cartesian product $G \times H$
of two signed simplicial complexes is still a signed discrete CW complex so that one can 
define $v_k(G \times H)$ as the number of $k$-dimensional cells in $G \times H$.

\section{Euler polynomial}

\paragraph{}
Every ring element $G$ in the strong ring has a {\bf $f$-vector} 
$(v_0,v_1,v_2, \dots, v_r)$. 
Unlike for a simplicial complex, the entries are now general integers and not non-negative
integers. This defines the {\bf Euler polynomial}
$$  e_{G}(t) = \sum_{l=0} v_l(G) t^l  \; .  $$
The Euler-Poincar\'e relations can be 
stated as $e_G(-1)=p_G(-1)$. This is a relation which holds for any ring element. 

\paragraph{}
The Poincar\'e polynomial relates with the Euler polynomial, which also has a nice
algebraic description:

\begin{propo}[Euler polynomial]
The map $G \to e_{G}(t)$ is a ring homomorphism from the strong connection ring to $\Z[t]$.
\end{propo}
\begin{proof}
When doing a disjoint union, the cardinalities add. For the Cartesian product, 
we take the set theoretical Cartesian product $G \times H$ of sets of sets. 
Now, if $x$ has dimension $p$ and $y$ has dimension $q$, then $x \times y$ has dimension $p+q$. 
\end{proof} 

\paragraph{}
There are higher dimensional versions of the Euler polynomial. Let $V=V_{k,l}$ denote the
{\bf $f$-vector} of $G$, where $V_{k,l}$ are the number of pairs $(x,y)$ of simplices 
which intersect and ${\rm dim}(x)=k$ and ${\rm dim}(y)=l$. We can now define multi-variate
Euler polynomials like the quadratic Euler polynomial 
$$  e_{G,2}(s,t) = \sum_{i,j} V_{i,j} s^i t^j  \; , $$
which is a generating function for the $f$-matrix. The definition of the Wu characteristic 
gives $\omega_2(G) = e_{G,2}(-1,-1)$. 

\paragraph{}
Examples. \\
1) For $G=K_2$, we have $e_G(t) = 2+t$ and $e_{G,2}(t,s) = 2+2t+2s+ts$
which leads to $\chi(G)=e_G(-1)=1$ and $\omega(G) = e_{G,2}(-1,-1) = -1$. \\
2) For $G=K_3$, we have $e_G(t) = 3+3t+t^2$ and 
$e_{G,2}(t,s) = s^2 t^2+3 s^2 t+3 s^2+3 s t^2+9 s t+6 s+3 t^2+6 t+3$
leading to $\omega(G) = e_{G,3}(-1,-1) = 1$. \\

\paragraph{}
Similarly, one can define {\bf $f$-tensors} for any $k$. They encode the 
cardinalities of interacting $k$-tuples of simplices.
The {\bf multivariate Euler polynomial} is then 
$$  e_{G,k}(t_1, \dots, t_k) = \sum_{j_1,j_2, \dots, j_k} V_{j_1,\dots,j_k} 
     t_1^{j_1} \cdots t_k^{j_k} \; . $$
For example, for $G=K_2$, we have $e_G(t_1,t_2,t_3)= 
2 + 2 t_1 + 2 t_2 + 2 t_3 + 2 t_1 t_2 + 2 t_1 t_3 + 2 t_2 t_3 + t_1 t_2 t_3$
which gives $\omega_3(G) = e_G(-1,-1,-1) = 1$. 

\begin{propo}[Euler polynomial]
The map $G \to e_{G,k}(t_1, \dots, t_k)$ is a ring homomorphism from the strong ring to 
$\Z[t_1, \dots, t_k]$.
\end{propo}
\begin{proof}
Again, the additive part is clear. For the multiplication,
if $(x_1 \times \cdots x_k)$ intersects $(y_1 \times \cdots y_k)$, this happens
if each $x_j$ intersects each $y_j$. The intersection of simplices $x_1,x_2, \dots x_k$
of dimension $r_1, \dots, r_k$ is represented algebraically by 
$t_1^{r_1} t_2^{r_2} \cdots t_k^{r_k}$. 
The intersection of simplices $y_1,y_2, \dots y_k$
of dimension $s_1, \dots, s_k$ is represented algebraically by 
$t_1^{s_1} t_2^{s_2} \cdots t_k^{s_k}$. If we take the product, we get
$t_1^{r_1+s_1} t_2^{r_2+s_2} \cdots t_k^{r_k+s_k}$ which represents the 
intersection of simplices $(x_1 \times y_1), (x_2 \times y_2), \dots (x_k \times y_k)$ 
of dimension $r_1+s_1,\dots, r_k+s_k$. 
\end{proof}

\paragraph{}
Having a multivariate ring homomorphism for the Euler polynomial suggests going back 
to the cohomology and redefining the cohomology so that the Poincar\'e polynomial 
becomes a multivariate polynomial too. This is indeed possible because the 
exterior derivative $d=d_1 + \cdots + d_k$ is a sum of {\bf partial exterior derivatives}
which each satisfy $d_j^2=0$. So, we get a {\bf cohomology Betti tensor} $b$ which is
in the case $k=2$ a {\bf cohomology Betti matrix}. One could define htne a
{\bf multivariate Poincar\'e polynomial} $p_{G,k}(t_1, \dots, t_k)$ and
the map $G \to p_{G,k}(t_1, \dots, t_k)$ is a ring homomorphism from the strong ring 
to $\Z[t_1, \dots, t_k]$. One can get the Poincar\'e polynomial (as we have 
defined it) by setting $p_{G,k}(t,t, \dots, t)$. This is how we have initially defined
the cohomology (see the appendix). Euler-Poincare $p_{G,k}(-1,-1,\dots,-1)=e_{G,k}(-1,-1,\dots,-1)$
still holds, but there is a problem: the Betti matrix is not invariant under
Barycentric refinements, similarly as the $f$-matrix is not invariant under Barycentric 
refinement. A sensible cohomology should be invariant under Barycentric refinement.

\section{Miscellaneous}

\paragraph{}
Any exterior derivative can be deformed via a Lax deformation 
\cite{IsospectralDirac,IsospectralDirac2}. There is a version which keeps
the exterior derivative $d$ real
$$ \dot{D} = [B(D),D],   B(t)=d(t)-d^*(t),  D(t)=d(t)+d(t)^* + b(t) \; .   $$
Then there is a version which allows the operators to become complex:
$$ \dot{D} = [B(D),D],   B(t)=d(t)-d^*(t)-ib(t),  D(t)=d(t)+d(t)^* + b(t)  \; . $$
This works in the same way also for the exterior derivative $d$ of 
interaction cohomology. In the first case, we have a scattering situation, in the 
second we get asymptotically to a linear wave equation. 

\paragraph{}
Having deformations in the division algebras $\R$ and $\C$, one can ask whether
it is also possible in the third and last associative division algebra, the quaternions 
$\HH$.  Given a real exterior derivative we can form $D=d+d^*$
and deform it with $D'=[B,D]$, where $B=d-d^*+I (b+b^*)$, where $I=oi+pj+qk$ is a space quaternion.
Given by three real commuting complexes $D_1,D_2,D_3$, we can start
with the initial condition $D=i D_1 + j D_2 + k D_3$. We initially have the Pythagorean relation
$L=D^2 = -L_1^2 - L_2^2 - L_3^2$.
Quaternions are best implemented as complex $2 \times 2$ matrices using Pauli matrices
$\sigma_i$.  The real number $1$ is represented by the identity matrix,
the standard imaginary square root of $-1$ is $i \sigma_3$, the quaternion
$j$ is $i \sigma_2$ and the quaternion $k$ is $i \sigma_1$. In other words,
the explicit translation from a quaternion $a+ib+jc+kd$ to a complex
$2 \times 2$ matrix $\left[ \begin{array}{cc} a+ib & c+id \\ -c+id & a-ib \end{array} \right]$.
The complex parameter $I$ is now just an arbitrary
element in the Lie algebra of $SU(2)$.

\paragraph{}
{\bf Example.} Take the smallest simplicial complex $G=\{ \{1\},\{2\},\{1,2\} \}$ for which
$D=\left[ \begin{array}{ccc}0 & 0 & -1 \\0 & 0 & 1 \\-1 &
1 & 0 \\ \end{array} \right]$. To evolve in the quaternion domain, take three copies and
glue them together to get the new quaternion valued Dirac operator: 
$$ D=\left[ \begin{array}{cccccc} 0 & 0 & 0 & 0 & -i & -1-i \\ 0 & 0 &
0 & 0 & 1-i & i \\ 0 & 0 & 0 & 0 & i & 1+i \\ 0 & 0 & 0 & 0 & -1+i &
-i \\ -i & -1-i & i & 1+i & 0 & 0 \\ 1-i & i & -1+i & -i & 0 & 0 \\
\end{array} \right] \; .  $$
Its Laplacian $L=D^2$ is just the original Hodge Laplacian,
where each entry $1$ is replaced with a $2 \times 2$ identity matrix.
Here is the deformed operator at time $t=0.1$ in the case $I=i$ (the quaternion $i$):
$$ D(t) = \left[ \begin{array}{cccccc}
0\,+0.36i&0.38\,+0.33i&0\,-0.36i&-0.38-0.33i&-0.19-0.85i& -0.86-0.86i\\
-0.38+0.33i&0\,-0.36i&0.38\,-0.33i&0\,+0.36i&0.86\,-0.86i &-0.19+0.85i\\
0\,-0.36i&-0.38-0.33i&0\,+0.36i&0.3 \,+0.33i&0.19\,+0.85i &0.86\,+0.86i\\
0.38\,-0.33i&0\,+0.36i&-0.38+0.33i&0\,-0.36i&-0.86+0.86i& 0.19\,-0.85i\\
0.19\,-0.85i&-0.86-0.86i&-0.19+0.85i&0.86\,+0.86i&0\,-0.72i &-0.67-0.77i\\
0.86\,-0.86i&0.19\,+0.85i&-0.86+0.86i&-0.19-0.85i&0.67\,-0.77
i&0\,+0.72i\\ \end{array} \right]  \; . $$

\paragraph{}
Let $G$ be a finite abstract simplicial complex with Barycentric refinement $G_1$
and let $n$ be the number of simplices of $G$. The {\bf connection graph} $G'$ of $G$ is
defined as the graph with vertex set $G$, and where two vertices are connected,
if they intersect. The graph $G_1$ is a subgraph of $G'$. Already
the case of a cyclic graph shows that $G'$ has in general a higher dimension
than $G$. Define the {\bf connection matrix} $L$ by $L(x,y)=1$ if $x \cap y \neq \emptyset$ and
$L(x,y)=0$ else.

\paragraph{}
The diagonal entries of $L$ correspond to self interactions, the
side diagonal elements are interactions between simplices.
With $S={\rm Diag}((-1)^{{\rm dim}(x)})$, define $J=S^{-1} I S$, where $I$ is the matrix containing only
$1$. The matrix satisfies $J_{ij}=(-1)^{i+j}$. It is a checkerboard matrix.
Define the {\bf Wu matrix} $W J$. We are interested in the matrix $W$ because it gives the Wu
characteristic: the Wu characteristic $\omega(G)$ of a finite simple graph is
equal to ${\rm tr}(W J)$.

\paragraph{}
A square matrix is called {\bf unimodular} if its determinant is either $-1$ or $1$.
This implies that all entries of the inverse are integer-valued.
While the Fredholm determinant of a general graph can be pretty
arbitrary, the Fredholm determinant of a connection graph is always
plus or minus $1$ implying that $L^{-1}$ is always an
integer matrix. Define the {\bf Fredholm characteristic}
$\psi(G) = {\rm det}(L)$ and the {\bf Fermi characteristic}
$\phi(G) = \prod_x \omega(x)$, where the sum is over all simplices in $G$
and $\omega(x) = (-1)^{{\rm dim}(x)}$.

\paragraph{}
For any finite simple graph, the Fredholm matrix $L$ of
its connection matrix is unimodular: $\psi(G)=\phi(G)$.  (We were writing
this paragraph of the current paper, when the unimodularity discovery happend,
delaying this by 2 years).

\paragraph{}
Finally, we notice that the graph $G_1'$ and $G_1$ are homotopic.
It is not true without a refinement step.
The octahedron graph $G$ has a connection graph $G'$ which is not homotopic to $G$:
the Euler characteristic of the connection graph is $0$. Its $f$-vector
is $(26, 180, 556, 918, 900, 560, 224, 54, 6)$. Its Betti vector is
$(1, 0, 0, 1, 0, 0, 0, 0, 0)$. It is topologically a 3-sphere, a Hopf
fibration of the 2-sphere as the graph has developed a
nontrivial three dimensional cohomology. It is not contractible.
If one removes all zero dimensional vertices, we still have the same
cohomology but a smaller graph with f-vector
$v(G)=(20, 132, 388, 582, 480, 224, 56, 6)$.

\paragraph{}
The Fredholm determinant of a Barycentric
refinement $G_1$ of a graph $G$ is always $1$.
There is an explicit linear map $f(G) \to Af = f(G_1)$ which gives
the $f$-vector of the Barycentric refinement $G_1$ from $G$. The image
always has an even number of odd-dimensional simplices.

\paragraph{}
There are various other connection matrices or connection tensors one
could look at. For any $k$ one can look at the tensor $L(x_1, \dots, x_k)$ which
is $1$ if the simplices $x_1, \dots, x_k$ pairwise intersect and $0$ else. But
it is the quadratic case which obviously is the most important one as we then
have a matrix which has a determinant and eigenvalues. An other structure we
have looked at but not found useful yet is to look at the set of intersecting
simplices $(x,y)$. If there are $n$ such pairs, look at the $n \times n$ matrix
which is $L((x,y),(u,v))=1$ if $(x,y)$ and $(u,v)$ intersect somewhere and $0$
else. One can similarly define such intersection matrices for $k>2$ also.
So far, we have not yet discovered a nice algebraic fact like the unimodularity
of the connection matrix.

\paragraph{}
In any case, whatever Laplacian is taken, we can look at it in the Barycentric limit
\cite{KnillBarycentric,KnillBarycentric2}.

\section{Historical remarks}

\paragraph{About simplicial complexes}

\paragraph{}
Abstract simplicial complexes were first defined in 1907 by Dehn and Heegaard
\cite{BurdeZieschang} and used further by Tietze, Brouwer, Steinitz, Veblen, Whitney, Weyl
and Kneser. In \cite{alexandroff}, Alexandroff calls a simplicial complex an
{\bf unrestricted skeleton complex}, whereas any finite sets is a {\bf skeleton complex}.
Still mostly used in Euclidean settings like \cite{Hatcher}, modern
topology textbooks like \cite{spanier,Maunder} or \cite{LeeTopologicalManifolds} use the
abstract version. A generalization of simplicial complexes are simplicial sets.

\paragraph{}
In \cite{Maunder}, the join of two simplicial complexes is defined. The definition for graphs goes
back to Zykov \cite{Zykov}. In \cite{JonssonSimplicial,Stanley86}, 
the empty set, the "void" is included in a complex.
This leads to reduced $f$-vectors $(v_{-1},v_0,...)$ with $v_{-1}=1$ and
{\bf reduced Euler characteristic} $\chi(G)-1$ known in enumerative combinatorics. The
number $1-\chi(G)$ is in one-dimensions the {\bf genus}. It is multiplicative when
taking joins of simplicial complexes.

\paragraph{About Wu characteristic}
Wu defines the characteristic numbers through the property that none of the simplices intersect. 
Lets call this $\tilde{\omega}(G)$.  For $k=2$, one can write this as $\chi(G)^2 - \tilde{\omega}(G)$
so that we have nothing new when defining this. For the higher order versions however, the relations
are no more so direct.  
We should also note that Wu, as custom for combinatorial topology in the 20'th century, 
dealt with polytopes. But since there are various definitions of polytopes \cite{Gruenbaum2003},
we don't use that language. 

\section{Questions}

\paragraph{}
Since we do not have homotopy invariance but invariance under Barycentric refinement, 
an important question is whether the cohomology is a topological invariant in the sense of 
\cite{KnillTopology}. Under which conditions is it true that if $H_1,H_2$ are homeomorphic
and $G \setminus H_1$ and $G \setminus H_2$ are topologically equivalent, then
then $H^p(G,H_1)=H^p(G,H_2)$? Is it true that $H^p(H_1) = H^p(H_2)$?  
We believe that both statements are true. The intuition comes from the believe that there is a 
continuum limit, where we can do deformations.
One attempt to construct a continuum cohomology is to look at the product complex of the usual de-Rham
complex on the product $M \times M$ of a manifold $M$, then take a limit of equivalence relations
where a pair $(f,g)$ of $k$ and $l$ forms with $k+l=p$ is restricted to the diagonal. If the Hodge
story goes over, then the homotopy invariance follows from elliptic regularity assuring that there
is a positive distance between the zero eigenvalues and the next larger eigenvalue. A 
small deformation of the geometry changes the eigenvalues continuously so that the kernel of the Hodge
operators on $p$-forms stay robust. This strategy could work while staying within the discrete: 
when doing Barycentric refinements, we can make the spectral effect of a homotopy deformation small.

\paragraph{}
Also not explored is what the Betti numbers of a $d$-complex with
or without boundary are. We have already a boundary formula
$\omega_2(G)=\omega(G) = \chi(G)-\chi(\delta{G})$ 
for Wu characteristic. Are there formulas for the Betti numbers in that case which 
relate the simplicial Betti numbers of $G$ and $\delta G$ with the Betti numbers of
the $k=2$ cohomology. 

\paragraph{}
We have seen that for the quadratic cohomology, there are examples of 
complexes like the M\"obius strip which have trivial cohomology in the sense that the Betti 
vector is 0. This is equivalent to the fact that the Hodge Laplacian is
invertible. This obviously does not happen for simplicial cohomology but
the M\"obius strip is an example for interaction cohomology. One can ask
which spaces have this property. Such a space necessarily has to have 
zero Wu characteristic. 

\paragraph{}
We have tried to compute the cohomology of the Poincar\'e Homology sphere
which is generated by 
$\{ \{1,2,4,9\}$, $\{1,2,4,15\}$, $\{1,2,6,14\}$, $\{1,2,6,15\}$,
$\{1,2,9,14\}$, $\{1,3,4,12\}$, $\{1,3,4,15\}$, $\{1,3,7,10\}$,
$\{1,3,7,12\}$, $\{1,3,10,15\}, \{1,4,9,12\}$, $\{1,5,6,13\}$,
$\{1,5,6,14\}$, $\{1,5,8,11\}$, $\{1,5,8,13\}$, $\{1,5,11,14\}$,
$\{1,6,13,15\}$, $\{1,7,8,10\}$, $\{1,7,8,11\}$, $\{1,7,11,12\},
\{1,8,10,13\}$, $\{1,9,11,12\}$, $\{1,9,11,14\}$, $\{1,10,13,15\}$,
$\{2,3,5,10\}$, $\{2,3,5,11\}$, $\{2,3,7,10\}$, $\{2,3,7,13\}$,
$\{2,3,11,13\}$, $\{2,4,9,13\}, \{2,4,11,13\}$, $\{2,4,11,15\}$,
$\{2,5,8,11\}$, $\{2,5,8,12\}$, $\{2,5,10,12\}$, $\{2,6,10,12\}$,
$\{2,6,10,14\}$, $\{2,6,12,15\}$, $\{2,7,9,13\}$, $\{2,7,9,14\},
\{2,7,10,14\}$, $\{2,8,11,15\}$, $\{2,8,12,15\}$, $\{3,4,5,14\}$,
$\{3,4,5,15\}$, $\{3,4,12,14\}$, $\{3,5,10,15\}$, $\{3,5,11,14\}$,
$\{3,7,12,13\}, \{3,11,13,14\}$, $\{3,12,13,14\}$, $\{4,5,6,7\}$,
$\{4,5,6,14\}$, $\{4,5,7,15\}$, $\{4,6,7,11\}$, $\{4,6,10,11\}$,
$\{4,6,10,14\}$, $\{4,7,11,15\}$, $\{4,8,9,12\}, \{4,8,9,13\}$,
$\{4,8,10,13\}$, $\{4,8,10,14\}$, $\{4,8,12,14\}$, $\{4,10,11,13\}$,
$\{5,6,7,13\}$, $\{5,7,9,13\}$, $\{5,7,9,15\}$, $\{5,8,9,12\}$,
$\{5,8,9,13\}, \{5,9,10,12\}$, $\{5,9,10,15\}$, $\{6,7,11,12\}$,
$\{6,7,12,13\}$, $\{6,10,11,12\}$, $\{6,12,13,15\}$, $\{7,8,10,14\}$,
$\{7,8,11,15\}$, $\{7,8,14,15\}, \{7,9,14,15\}$, $\{8,12,14,15\}$,
$\{9,10,11,12\}$, $\{9,10,11,16\}$, $\{9,10,15,16\}$, $\{9,11,14,16\}$,
$\{9,14,15,16\}$, $\{10,11,13,16\}$, $\{10,13,15,16\}, \{11,13,14,16\}$,
$\{12,13,14,15\}$, $\{13,14,15,16\} \}$. The corresponding Laplacian 
is a $70616 \times 70616$ matrix. Unfortunately, our computing resources
were not able yet to find the kernel of its blocks yet. It would be nice
to compare the interaction cohomology of the homology sphere with 
the interaction cohomology of the 3-sphere. 

\vfill

\pagebreak

\section*{Appendix: Code}

The Mathematica procedure "Coho2" computes in less than 12 lines
a basis for all the cohomology groups $H^p(G,H)$ for a pair of 
simplicial complexes $G,H$. We display then the 
Betti numbers $b_p(G,H)={\rm dim}(H^p(G,H))$. 
As in the text, the case $H^p_2(G)$ is the same than $H^p(G,G)$. 

\begin{tiny}
\lstset{language=Mathematica} \lstset{frameround=fttt}
\begin{lstlisting}[frame=single]
Coho2[G_,H_]:=Module[{},n=Length[G];m=Length[H];len[x_]:=Total[Map[Length,x]];U={}; 
  Do[If[Length[Intersection[G[[i]],H[[j]]]]>0,U=Append[U,{G[[i]],H[[j]]}]],{i,n},{j,m}];
  U=Sort[U,len[#1]<len[#2] & ];u=Length[U];l=Map[len,U]; w=Union[l]; 
  b=Prepend[Table[Max[Flatten[Position[l,w[[k]]]]],{k,Length[w]}],0]; h=Length[b]-1; 
  deriv1[{x_,y_}]:=Table[{Sort[Delete[x,k]],y},{k,Length[x]}];
  deriv2[{x_,y_}]:=Table[{x,Sort[Delete[y,k]]},{k,Length[y]}];
  d1=Table[0,{u},{u}]; Do[v=deriv1[U[[m]]]; If[Length[v]>0,
    Do[r=Position[U,v[[k]]]; If[r!={},d1[[m,r[[1,1]]]]=(-1)^k],{k,Length[v]}]],{m,u}];
  d2=Table[0,{u},{u}]; Do[v=deriv2[U[[m]]]; If[Length[v]>0,
    Do[r=Position[U,v[[k]]]; If[r!={},d2[[m,r[[1,1]]]]=(-1)^(Length[U[[m,1]]]+k)],
    {k,Length[v]}]],{m,u}]; d=d1+d2; Dirac=d+Transpose[d]; L=Dirac.Dirac; Map[NullSpace,
  Table[Table[L[[b[[k]]+i,b[[k]]+j]],{i,b[[k+1]]-b[[k]]},{j,b[[k+1]]-b[[k]]}],{k,h}]]];
Betti2[G_,H_]:=Map[Length,Coho2[G,H]];Coho2[G_]:=Coho2[G,G]; Betti2[G_]:=Betti2[G,G]; 
Generate[A_]:=Sort[Delete[Union[Sort[Flatten[Map[Subsets,A],1]]],1]];

moebius=Generate[{{1,2,5},{1,4,5},{1,4,7},{2,3,6},{2,5,6},{3,6,7},{4,3,7}}];
Print["Moebius Strip: ",Betti2[moebius]];
cylinder=Generate[{{1,2,5},{1,4,8},{1,5,8},{2,3,6},{2,5,6},{3,4,7},{3,6,7},{4,7,8}}];
Print["Cylinder: ",Betti2[cylinder]];
circle=Generate[{{1,2},{2,3},{3,4},{4,5},{5,1}}]; 
Print["Circle: ",Betti2[circle]]; 
figureeight=Generate[{{1,2},{1,4},{2,3},{2,5},{2,7},{3,4},{5,6},{6,7}}]; 
Print["Figure Eight: ",Betti2[figureeight]]; 

point=Generate[{{1}}]; 
Print["1-simplex: ",Betti2[point]]; 
segment=Generate[{{1,2}}]; 
Print["2-simplex: ",Betti2[segment]]; 
triangle=Generate[{{1,2,3}}]; 
Print["3-simplex: ",Betti2[triangle]]; 
tetrahedron=Generate[{{1,2,3,4}}]; 
Print["4-simplex: ",Betti2[tetrahedron]]; 

Print["Point in interval: ",Betti2[{{1,2},{2,3},{1},{2},{3}},{{2}}]];
Print["Point on boundary: ",Betti2[{{1,2},{2,3},{1},{2},{3}},{{1}}]];

house=Generate[{{1,2},{2,3},{3,4},{4,1},{2,3,5}}];
Print["House graph: ",Betti2[house]];

octahedron=Generate[{{1,2,3},{1,2,4},{1,3,5},{1,4,5},{2,3,6},{2,4,6},{3,5,6},{4,5,6}}];
Print["Octahedron: ",Betti2[octahedron]];
equator=Generate[{{2,3},{3,5},{5,4},{4,2}}];
Print["Circle in octahedron: ",Betti2[octahedron,equator]];

kleinbottle=Generate[{{1,2,3},{1,2,6},{1,3,5},{1,4,7},{1,4,8},{1,5,7},
{1,6,8},{2,3,7},{2,4,6},{2,4,8},{2,5,7},{2,5,8},{3,4,6},{3,4,7},{3,5,6},{5,6,8}}];
Print["Klein bottle: ",Betti2[kleinbottle]]; 

projectiveplane=Generate[{{1,2,5},{1,2,9},{1,4,5},{1,4,7},{1,8,7},{1,8,9},{2,3,6},
{2,3,10},{2,5,6},{2,9,10},{3,6,7},{3,10,11},{4,3,7},{4,3,11},{4,5,12},
{4,11,12},{5,6,13},{5,12,13},{6,7,14},{6,13,14},{8,7,14},{8,9,15},
{8,14,15},{9,10,15},{10,11,15},{11,12,15},{12,13,15},{13,14,15}}];
Print["Projective Plane: ",Betti2[projectiveplane]]; 

Print["The next computation deals with 4160x4160 matrices: patience please "]; 
threesphere=Generate[{{1,3,5,7},{1,3,5,8},{1,3,6,7},{1,3,6,8},
{1,4,5,7},{1,4,5,8},{1,4,6,7},{1,4,6,8},{3,5,7,2},{3,5,8,2},
{3,6,7,2},{3,6,8,2},{4,5,7,2},{4,5,8,2},{4,6,7,2},{4,6,8,2}}];
Print["3-sphere: ",Betti2[threesphere]];

RandomSimplComplex[n_,m_]:=Module[{A={},X=Range[n],k},Do[k:=1+Random[Integer,n-1];
  A=Append[A,Union[RandomChoice[X,k]]],{m}];Generate[A]];
Print["Random Complex: ",Betti2[RandomSimplComplex[6,10]]]
\end{lstlisting}
\end{tiny}

\hfill
\pagebreak

\section*{Appendix: math table talk: "Wu Characteristic"}

This was a Handout to "Wu Characteristic", a Harvard math table talk given on March 8, 2016
\cite{MathTableMarch}. It was formulated in the language of graphs. 
In that talk, the exterior derivative was not yet finalized. 
It gave a cohomology which was not yet invariant under Barycentric refinements. \\

{\bf Definitions.} 
Let $G=(V,E)$ be a {\bf finite simple graph} with vertex set $V$ and edge set $E$.
The {\bf $f$-vector} $v(G)=(v_0(G),v_1(G), \dots,v_d(G))$ 
contains as coordinates the number $v_k(G)$ of complete subgraphs $K_{k+1}$ of $G$. These
subgraphs are also called {\bf k-simplices} or {\bf cliques}. The {\bf $f$-matrix} $V(G)$ has the entries
$V_{ij}(G)$ counting the number of intersecting pairs $(x,y)$, where $x$ is a $i$-simplex
and $y$ is a $j$-simplex in $G$. The {\bf Euler characteristic} of $G$ is $\chi(G)=\sum_i (-1)^i v_i(G)$.
The {\bf Wu characteristic} of $G$ is $\omega(G) = \sum_{ij} (-1)^{i+j} V_{ij}(G)$. 
If $A,B$ are two graphs, the {\bf graph product} is a new graph which has as vertex set the set of pairs $(x,y)$, where
$x$ is a simplex in $A$ and $y$ is a simplex in $B$. Two such pairs $(x,y),(a,b)$ are connected in the product
if either $x \subset a,y \subset b$ or $a \subset x, b \subset y$. 
The product of $G$ with $K_1$ is called the {\bf Barycentric refinement} $G_1$ of $G$. Its vertices are
the simplices of $G$. Two simplices are connected in $G_1$ if one is contained in the other. If $W$ is a subset
of $V$, it {\bf generates the graph} $(W,F)$, where $F$ is the subset of $(a,b) \in E$ for 
which $a,b \in W$. The {\bf unit sphere} $S(v)$ of a vertex $v$ is the subgraph generated by the vertices 
connected to $v$. A function $f: V \to R$ satisfying $f(a) \neq f(b)$ for $(a,b) \in E$ is called a 
{\bf coloring}. The minimal range of a coloring is the {\bf chromatic number} of $G$. 
Define the {\bf Euler curvature} $\kappa(v) = \sum_{k=0}^\infty (-1)^k v_{k-1}(S(v))/(k+1)$
and the {\bf Poincar\'e-Hopf index} $i_f(v) = 1-\chi(S^-_f(v))$, where $S^-_f(x)$ is the graph generated
by $S^-_f(v)=\{ w \in S(v) \; | \; f(w)<f(v) \}$. Fix an orientation on {\bf facets} on $G$ simplices of 
maximal dimension. Let $\Omega^{-1}=\{0\},\Omega^k(G)$ be the set of functions $f(x_0,\dots,x_k)$ from the set of 
$k$-simplices of $G$ to $R$ which are anti-symmetric.
The {\bf exterior derivatives} $d_{k}f(x_0, \dots, x_k) = \sum_{j=0}^{k} (-1)^j f(x_0, \dots, \hat{x}_j, \dots x_k)$
defines linear transformations. The orientation fixes {\bf incidence matrices}. They can be collected 
together to a large $n \times n$ matrix $d$ with $n = \sum_{i=0}^d v_i$. The matrix $d$ is called 
the {\bf exterior derivative}. Since $d_{k+1} d_k=0$, the image of $d_k$ is contained in the kernel of $d_{k+1}$. 
The vector space $H^k(G) = {\rm ker}(d_k)/{\rm im}(d_{k-1})$ is the {\bf $k$'th cohomology group} of $G$. 
Its dimension $b_k(G)$ is called the {\bf $k$'th Betti number}. The symmetric matrix $D=d+d^*$ is the 
{\bf Dirac operator} of $G$. Its square $L=D^2$ is the {\bf Laplacian} of $G$. It decomposes into blocks $L_k$
the {\bf Laplacian on $k$-forms}. The {\bf super trace} of $L$ is ${\rm str}(L) = \sum_k (-1)^k {\rm tr}(L_k)$. 
A {\bf graph automorphism} $T$ is a permutation of $V$ such that if $(a,b) \in E$, then $(T(a),T(b)) \in E$. 
The set of graph automorphisms forms a {\bf automorphism group} ${\rm Aut}(G) = {\rm Aut}(G_1)$.
If $A$ is a subgroup of ${\rm Aut}(G)$, we can form $G_k/A$ which is for $k \geq 2$
again a graph. Think of $G$ as a {\bf branched cover} of $G/A$. For a simplex $x$, define its 
{\bf ramification index} 
$e_x=1-\sum_{1 \neq T \in A, T(x)=x} (-1)^{{\rm dim}(x)}$. 
Any $T \in {\rm Aut}(G)$ induces a linear map $T_k$ on $H^k(G)$. 
The {\bf super trace} $\chi_T(G) = \sum_{k} (-1)^k {\rm tr}(T_k)$ is the {\bf Lefschetz number} of $T$. For $T=1$
we have $\chi_{\rm Id}(G)=\chi(G)$. 
For a simplex $x=T(x)$, define the {\bf Brouwer index} $i_T(x) = (-1)^{{\rm dim}(x)} {\rm sign}(T|x)$
\cite{brouwergraph}.  An integer-valued function $X$ in $\Omega$ is called a {\bf divisor}.
The {\bf Euler characteristic} of $X$ is $\chi(X) = \chi(G)+\sum_v X(v)$. If $v_2(G)=0$,
it is $1-b_1+{\rm deg}(X)$, where $b_1=g$ is the {\bf genus}. 
If $f$ is a divisor, then $(f)=L f$ is called a {\bf principal divisor}. Two divisors are {\bf linearly equivalent}
if $X-Y=(f)$ for some $f$. A divisor $X$ is {\bf essential} if $X(v) \geq 0$ for all $v \in V$. 
The {\bf linear system} $|X|$ of $X$ is $\{ f \; | \;$ $X+(f)$ is essential $\}$.
The {\bf dimension} $l(X)$ of $X$ is $-1$ if $|X|=\emptyset$ and else max $k \geq 0$ so that for all $m \leq k$ 
and all $\chi(Y)=m$ the divisor $X-Y$ is essential. The {\bf canonical divisor} $K$
is for graph without triangles defined as $K(v)=-2\kappa(v)$, where $\kappa(v)$ is the curvature. 
The {\bf dimension} of a graph is $-1$ for the empty graph and inductively the average of the dimensions of the unit 
spheres plus $1$.
For the rabbit graph $G$, the f-vector is $v(G) = [5,5,1]$, the f-matrix is 
$V(G) = [[5, 10, 3],[10, 21, 5],[3, 5, 1]]$. Its dimension is
${\rm dim}(G)=7/5$, the chromatic number is $3$, the Euler Betti numbers are $b(G)=(1,0)$ 
which super sums to Euler characteristic $\chi(G)=\omega_1(G)=\sum_i (-1)^i v_i(G)=\sum_i (-1)^i b_i(G)=1$.
The Wu Betti numbers are $b(G)=(0,2,6,1,0)$. 
The super sum is the Wu characteristic $\omega_2(G)=\sum_{i,j} (-1)^{i+j} V(G)_{ij}=3$. 
The cubic Wu Betti numbers are $(0, 2, 16, 34, 16, 1, 0)$ which super sums to cubic Wu $\omega_3(G) = -5$.

{\bf Theorems} \\
Here are adaptations of theorems for Euler characteristic to graph theory. Take any
graph $G$. The left and right hand side are always {\bf concrete integers}. 
The theorem tells that they are equal and show: ``{\bf The Euler characteristic
is the most important quantity in all of mathematics}".

\tttheorem{Gauss-Bonnet:  $\chi(G) = \sum_{v \in V} \kappa(v)$}
\tttheorem{Poincar\'e-Hopf: $\chi(G) = \sum_{v \in V} i_f(v)$}
\tttheorem{Euler-Poincar\'e: $\chi(G) = \sum_i (-1)^i b_i(G)$}
\tttheorem{McKean-Singer: $\chi(G) = {\rm str}(e^{-t L})$}
\tttheorem{Steinitz-DeRham: $\chi(G \times H) = \chi(G) \chi(H)$}
\tttheorem{Brouwer-Lefschetz: $\chi_T(G)=\sum_{x=T(x)} i_T(x)$ }
\tttheorem{Riemann-Hurwitz: $\chi(G) = |A|  \; \chi(G/A) - \sum_{x} (e_x-1)$   }
\tttheorem{Riemann-Roch:  $\chi(X) = l(X) - l(K-X)$ }

All theorems except Riemann-Roch hold for general finite simple graphs. 
The complexity of the proofs are lower than in the continuum. 
In continuum geometry, lower dimensional parts of space are access using "sheaves" or 
"integral geometry" of "tensor calculus". 
For manifolds, a functional analytic frame works like elliptic regularity is needed to define
the {\bf heat flow} $e^{-tL}$.
We currently work on versions of these theorems to all the {\bf Wu characteristics}
$\omega_k(G)  = \sum_{x_1 \sim \cdots \sim x_k} \omega(x_1) \omega(x_2) \cdots \omega(x_k)$,
where the sum is over all ordered $k$ tuples of simplices in $G$ which intersect. Besides proofs,
we built also computer implementation for all notions allowing to compute things in concrete situations. There are 
Wu versions of curvature, Poincar\'e-Hopf index, Brouwer-Lefschetz index and an interaction cohomology. 
All theorems generalize. Only for Riemann-Roch, we did not complete the adaptation of {\bf Baker-Norine theory}
yet, but it looks very good (take Wu curvature for canonical divisor and use Wu-Poincar\'e-Hopf indices).
What is the significance of Wu characteristic? We don't know yet. The fact that important theorems
generalize, generates optimism that it can be significant in physics as an {\bf interaction functional}
for which extremal graphs have interesting properties. {\bf Tamas Reti} noticed already that
for a triangle free graph with $n$ vertices and $m$ edges, the Wu characteristic is 
$\omega(G)=n-5m+M(G)$, where $M(G)=\sum_v {\rm deg}(v)^2$ is the {\bf first Zagreb index} of $G$.
For Euler characteristic, we have guidance from the continuum.
This is no more the case for Wu characteristics. Nothing similar
appears in the continuum. Related might be {\bf intersection theory}, as Wu characteristic defines 
an {\bf intersection number} $\omega(A,B) = \sum_{x \sim y, x \subset A, y \subset B} \omega(x) \omega(y)$
for two subgraphs $A,B$ of $G$.
To generalize the parts using cohomology, we needed an adaptation of simplicial cohomology to 
a cohomology of interacting simplices. We call it {\bf interaction cohomology}:
one first defines discrete {\bf quadratic differential forms} $F(x,y)$ and
then an exterior derivative $dF(x,y) = F(dx,y)$. As $d^2=0$, one gets cohomologies the usual way. \\

{\bf The Rabbit} \\

\begin{center}
\scalebox{0.18}{\includegraphics{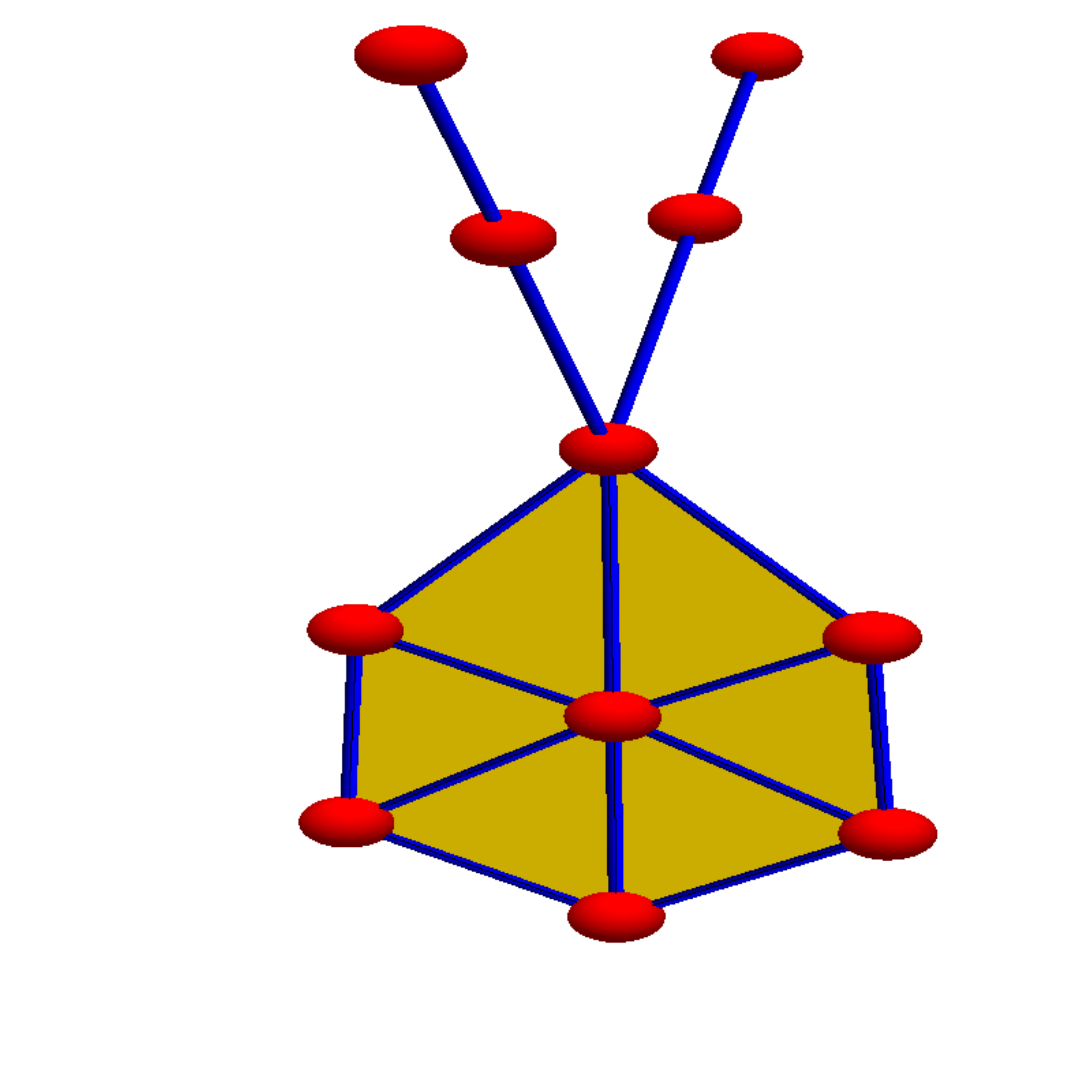}}
\scalebox{0.18}{\includegraphics{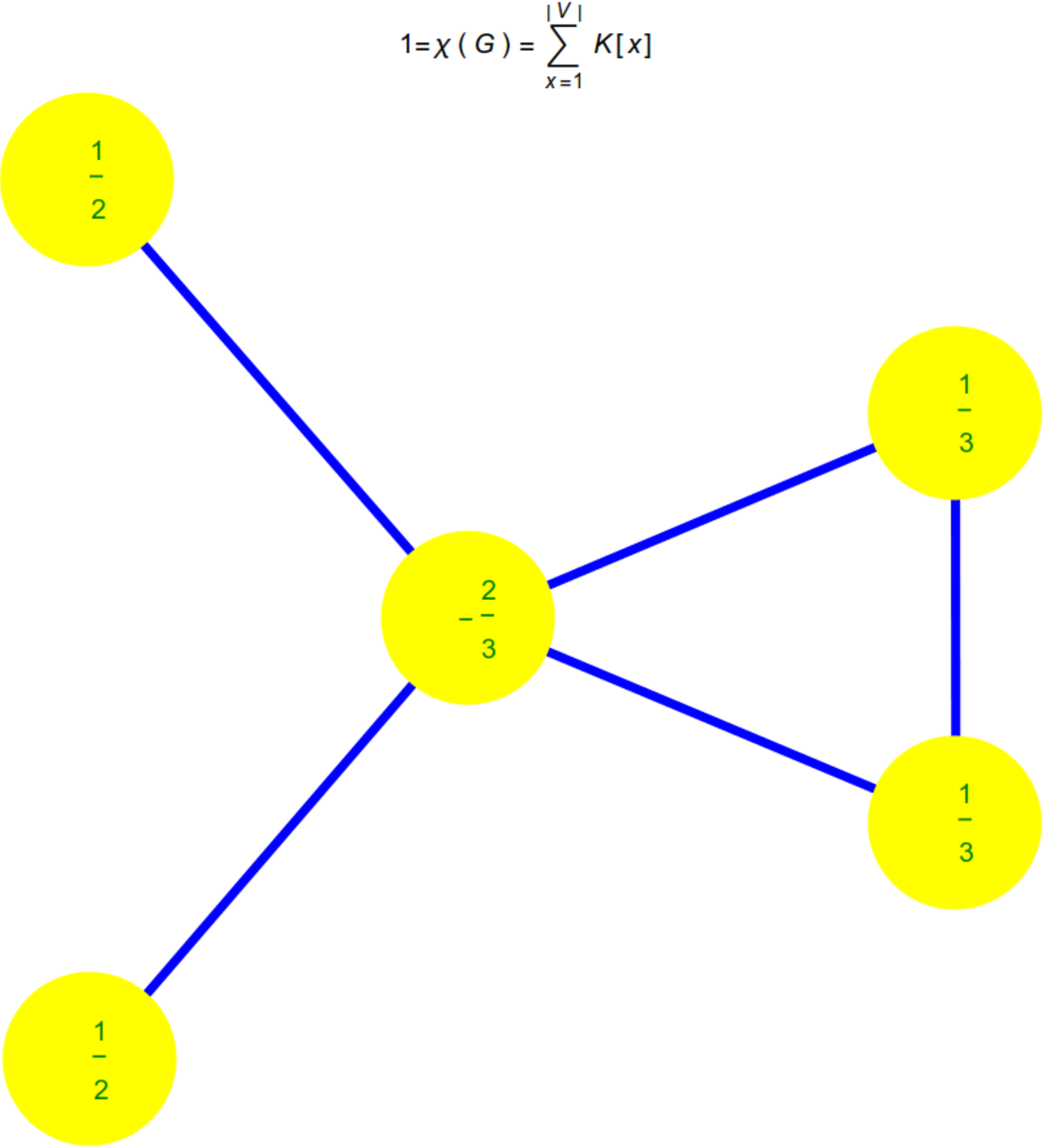}} \\
{\bf 1.} Barycentric refinement of the {\bf rabbit graph} $R$. {\bf 2.} Euler curvatures.
\end{center}

\begin{center}
\scalebox{0.18}{\includegraphics{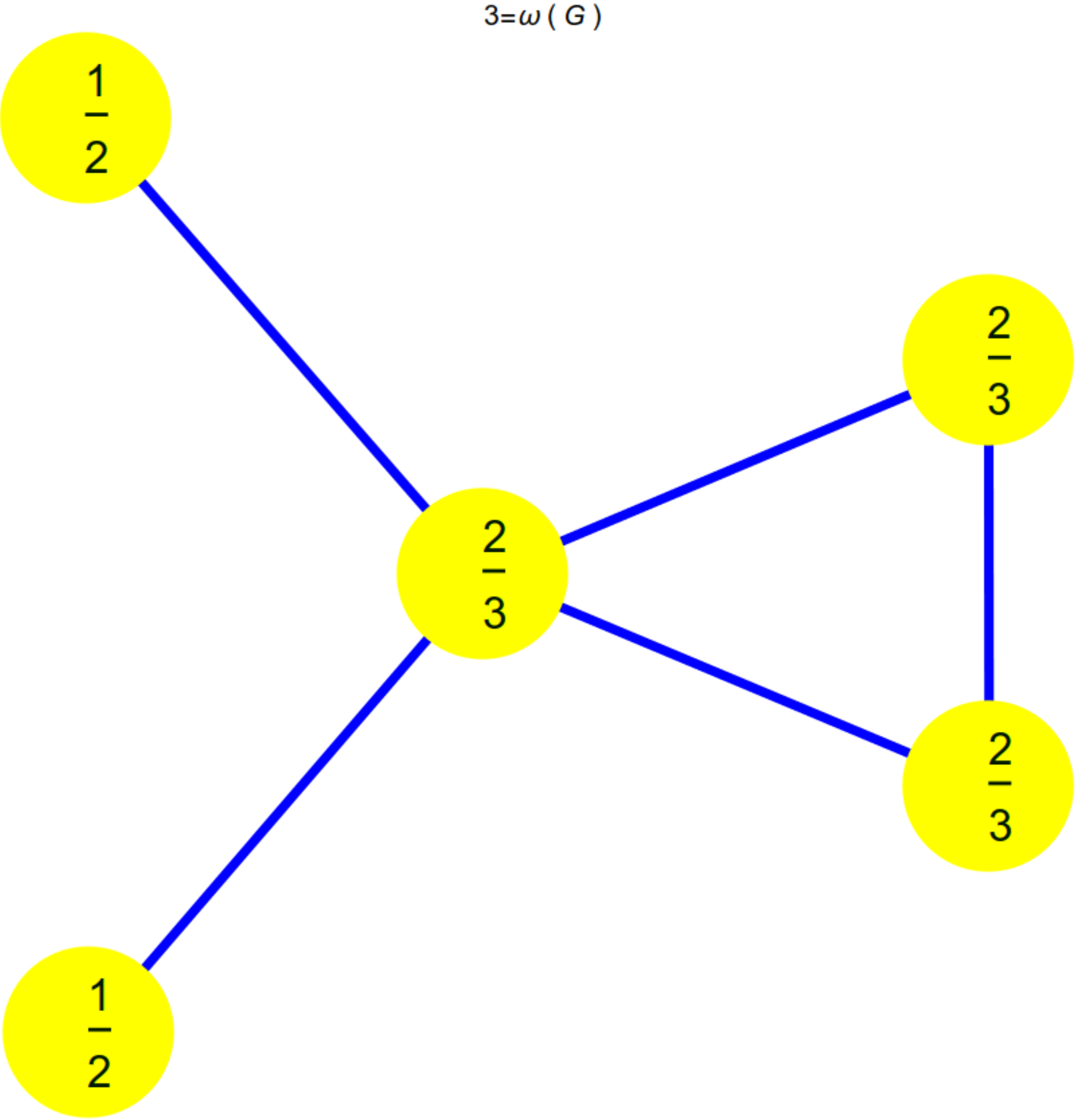}}
\scalebox{0.18}{\includegraphics{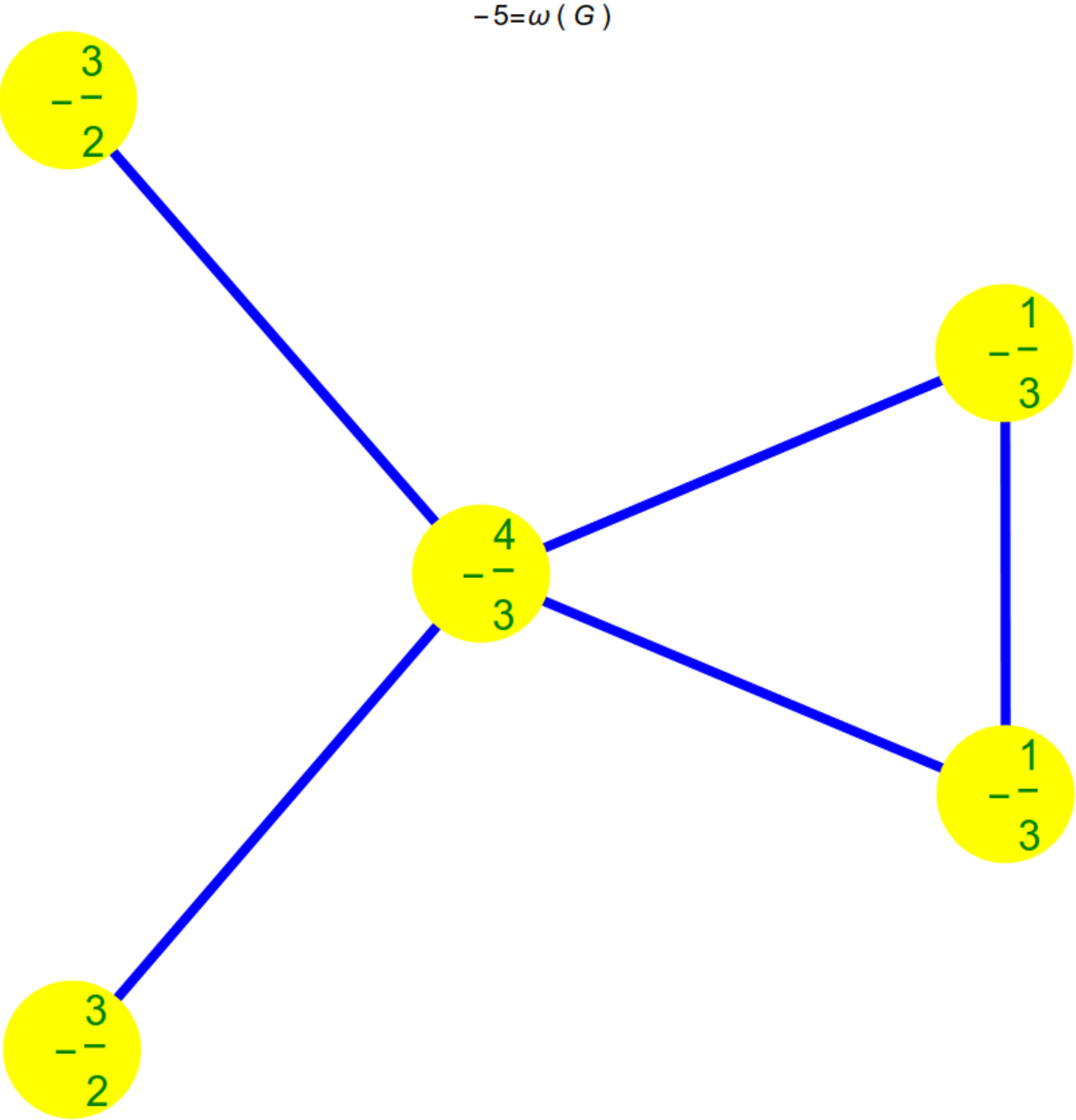}} \\
{\bf 3.} Wu curvatures adding up to $\omega$. {\bf 4.} Cubic Wu curvatures adding up to $\omega_3$.
\end{center} 

\begin{center}
\scalebox{0.18}{\includegraphics{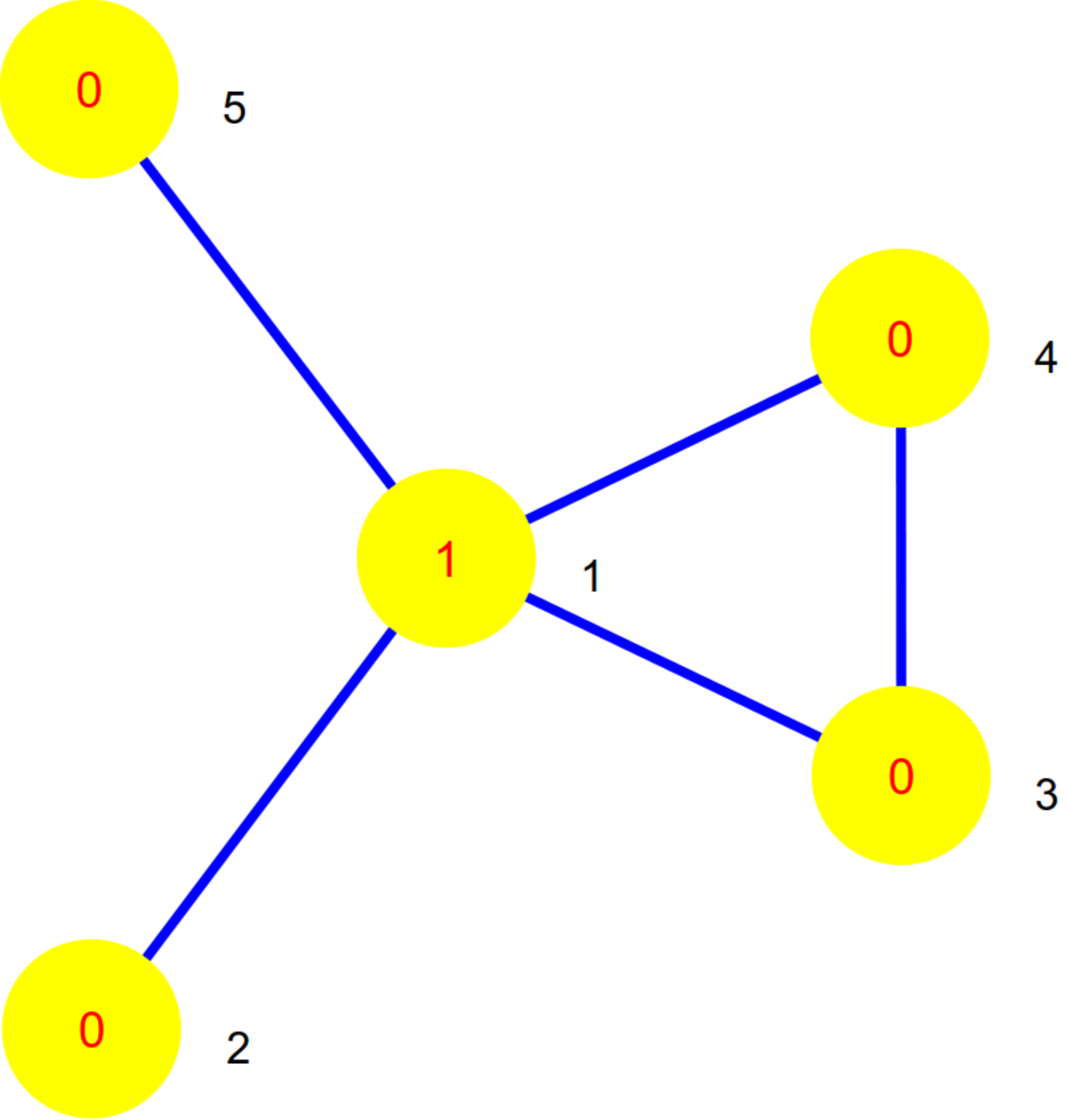}}
\scalebox{0.18}{\includegraphics{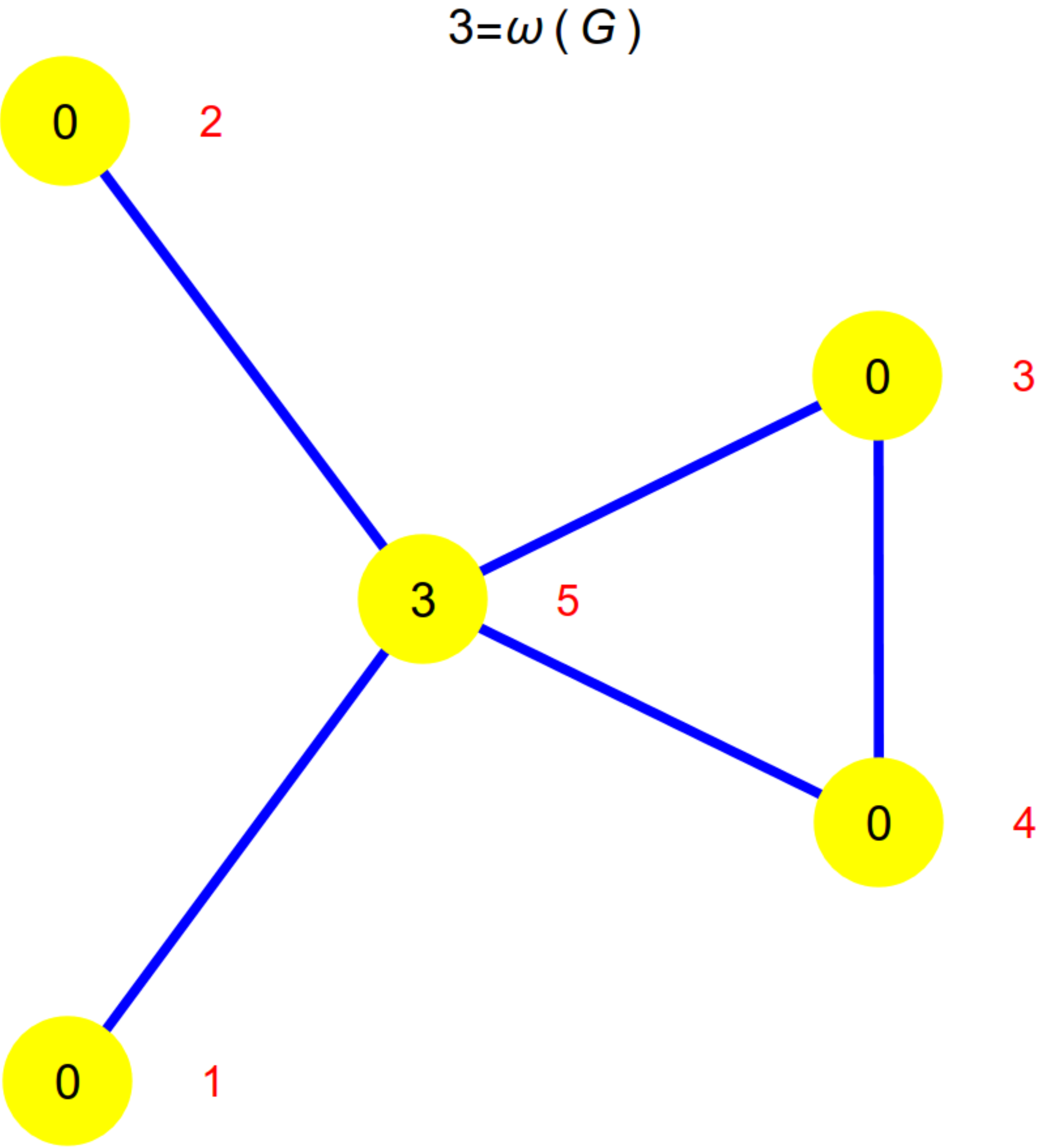}} \\
{\bf 5.} Poincar\'e-Hopf indices.  {\bf 6.} Wu-Poincar\'e-Hopf indices
\end{center} 

\pagebreak

\begin{center}
\scalebox{0.28}{\includegraphics{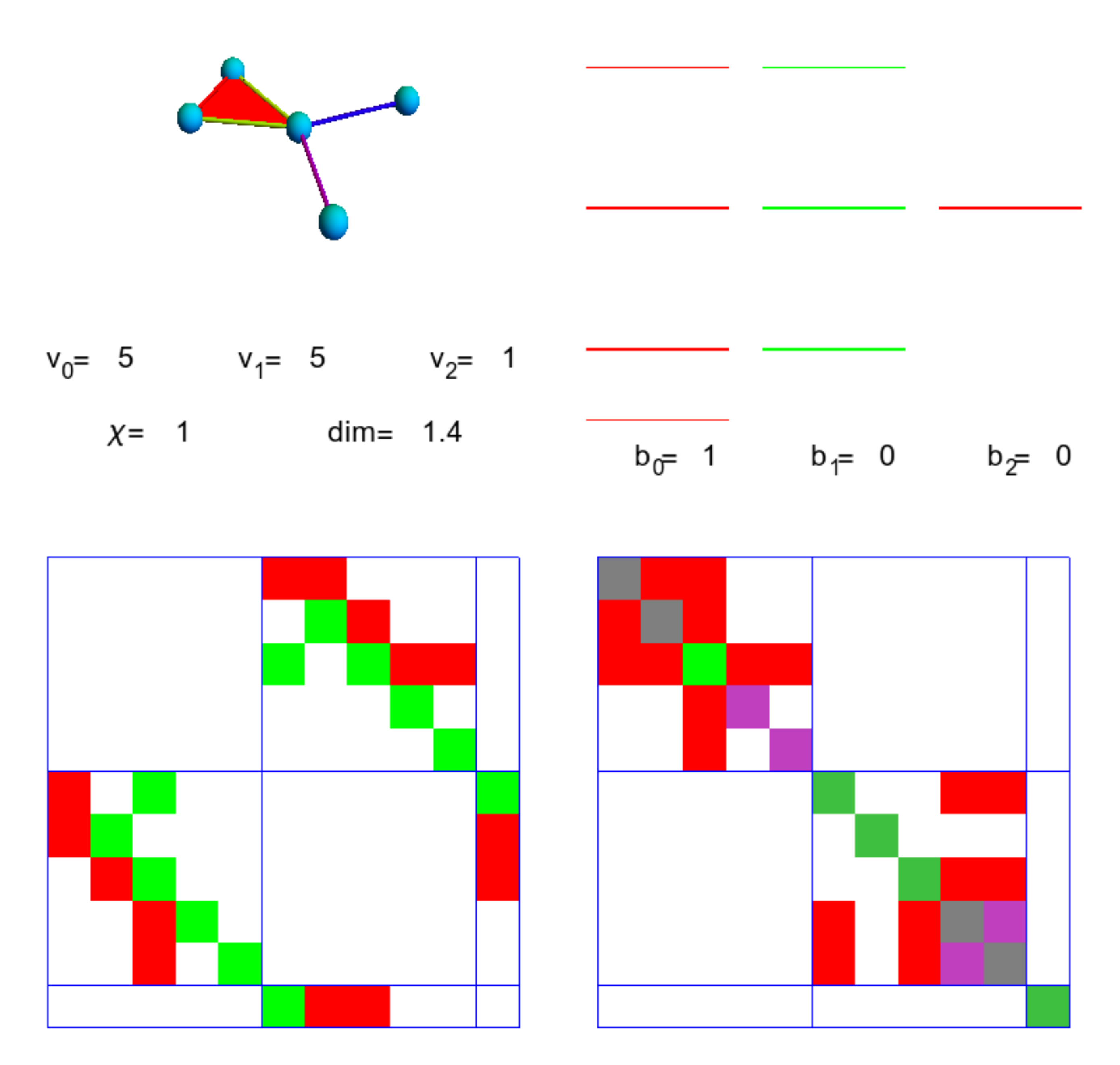}} \\
{\bf 7.} Spectrum $\sigma(L_0)=\{5,3,1,1,0\}$,$\sigma(L_1)= \{5, 3, 3, 1, 1\}$, 
$\sigma(L_2)=\{3\}$, Betti numbers=${\rm dim}({\rm ker}(L_k))$, Dirac operator $D$ and 
Laplacian $L$ of $R$ with 3 blocks. {\bf Super symmetry}: {\bf Bosonic spectrum} $\{5,3,1,1\} \cup \{3\}$
agrees with {\bf Fermionic spectrum} $\{5,3,3,1,1\}$.
\end{center} 

\begin{center}
\scalebox{0.26}{\includegraphics{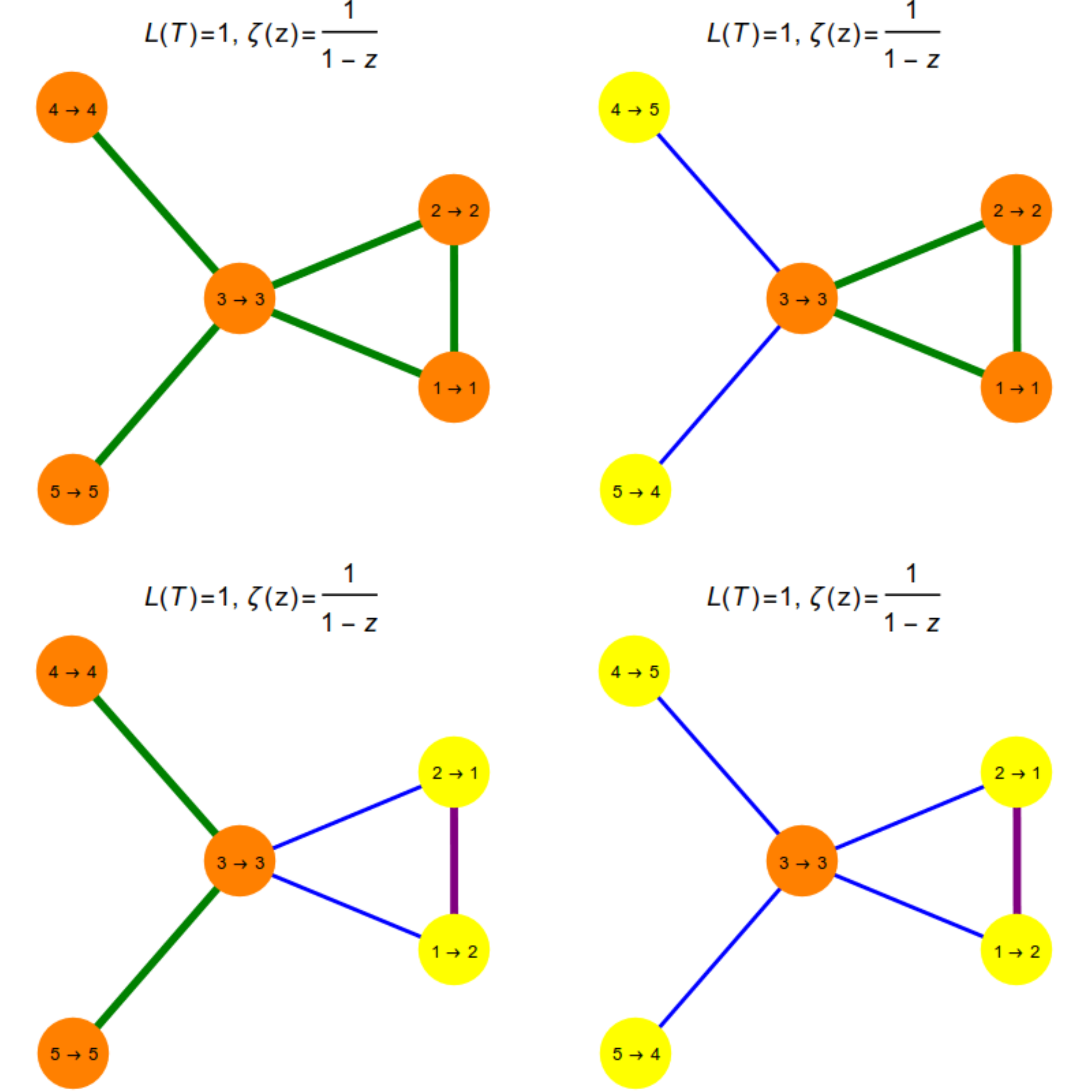}} \\
{\bf 8.} Lefschetz numbers of the 4 automorphisms in ${\rm Aut}(R)=Z_2 \times Z_2$.
\end{center} 

\begin{center}
\scalebox{0.18}{\includegraphics{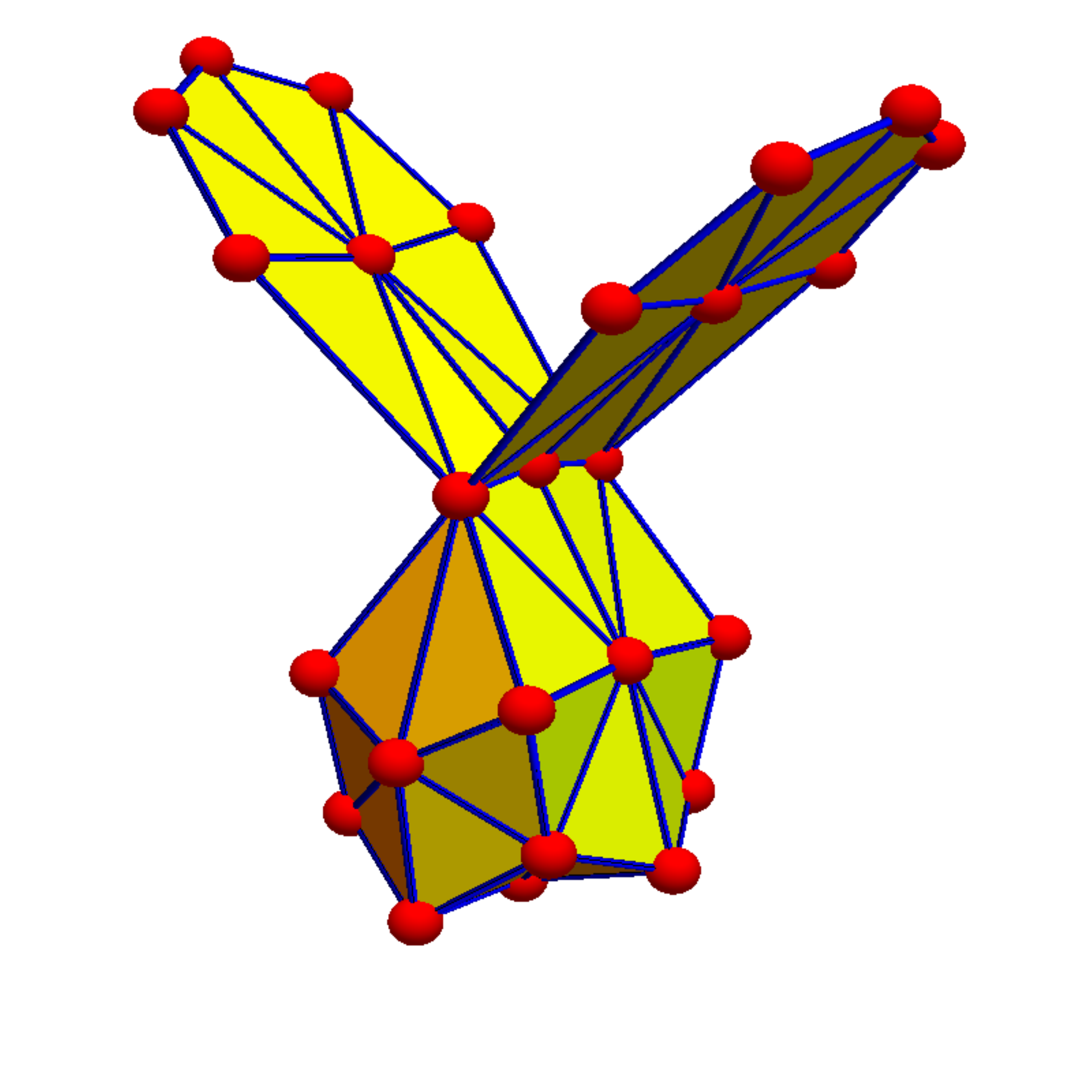}} 
\scalebox{0.18}{\includegraphics{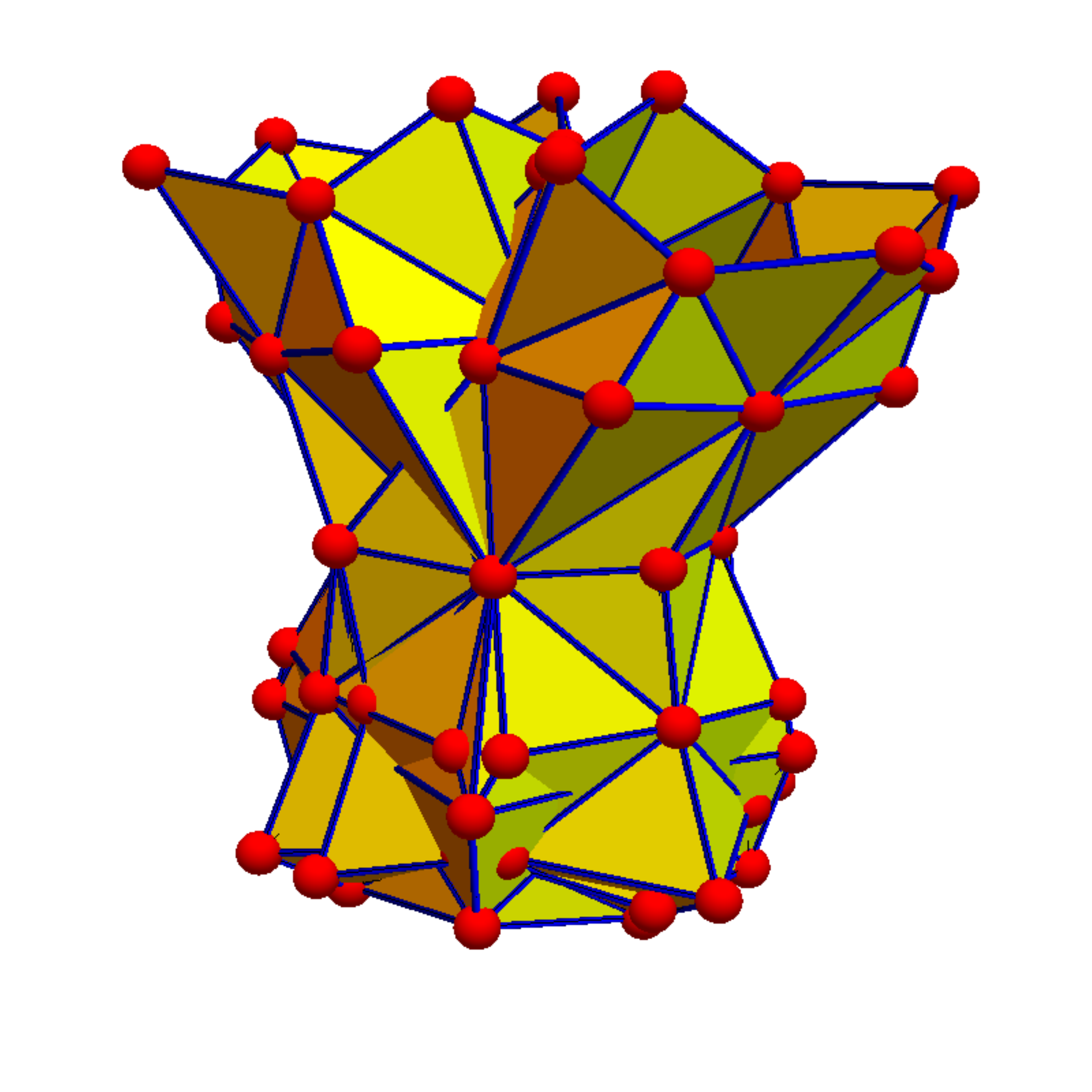}} \\
{\bf 9.} {\bf 3D rabbit} $R \times K_2$ with $\omega=-3$, ${\rm dim}=2.59433 \geq {\rm dim}(R)+1=2.4$.\\
{\bf 10.} curled rabit $G=R \times C_4$ with $\chi(G)=\omega(G)=0$, $b_0=1,b_1=1$. 
\end{center}

{\bf Calculus} \\
Fix an orientation on simplices. The exterior derivatives are the {\bf gradient}
$$ {\rm grad} = d_0 = \left[
                  \begin{array}{ccccc}
                   -1 & 0 & 1 & 0 & 0 \\
                   -1 & 1 & 0 & 0 & 0 \\
                   0 & -1 & 1 & 0 & 0 \\
                   0 & 0 & -1 & 1 & 0 \\
                   0 & 0 & -1 & 0 & 1 \\
                  \end{array}
                  \right] \; ,$$
the {\bf curl }
$$ {\rm curl} = d_1 = \left[
                  \begin{array}{ccccc}
                   1 & -1 & -1 & 0 & 0 \\
                  \end{array}
                  \right] $$
which satisfies ${\rm curl(grad}) = 0$ and the {\bf divergence }
$$  {\rm div} = d_0^* = \left[
                  \begin{array}{ccccc}
                   -1 & -1 & 0 & 0 & 0 \\
                   0 & 1 & -1 & 0 & 0 \\
                   1 & 0 & 1 & -1 & -1 \\
                   0 & 0 & 0 & 1 & 0 \\
                   0 & 0 & 0 & 0 & 1 \\
                  \end{array}
                  \right] \; . $$
The {\bf scalar Laplacian} or {\bf Kirchhoff matrix} $L_0 = d_0^* d_0 = {\rm div}({\rm grad})$ is
$$ L_0 =   \left[
                  \begin{array}{ccccc}
                   2 & -1 & -1 & 0 & 0 \\
                   -1 & 2 & -1 & 0 & 0 \\
                   -1 & -1 & 4 & -1 & -1 \\
                   0 & 0 & -1 & 1 & 0 \\
                   0 & 0 & -1 & 0 & 1 \\
                  \end{array} \right] \; . $$
It does not depend on the chosen orientation on simplices and is equal to $B-A$, where $B$ 
is the diagonal matrix containing the {\bf vertex degrees} and $A$ is the
{\bf adjacency matrix} of $G$. Has eigenvalues $\sigma(L_0) = \{ 5,3,1,1,0 \}$ with one-dimensional kernel spanned by 
$[1,1,1,1,1]^T$. Also the {\bf 1-form Laplacian} $L_1 = d_1^* d_1 + d_0 d_0^*$ does not depend on the
orientation of the simplices:
$$ L_1 = \left[
                  \begin{array}{ccccc}
                   3 & 0 & 0 & -1 & -1 \\
                   0 & 3 & 0 & 0 & 0 \\
                   0 & 0 & 3 & -1 & -1 \\
                   -1 & 0 & -1 & 2 & 1 \\
                   -1 & 0 & -1 & 1 & 2 \\
                  \end{array}
                  \right]  $$
with eigenvalues $\sigma(L_1) = \{ 5,3,3,1,1 \}$ with zero 
dimensional kernel. This reflects that the graph is {\bf simply connected}. \\

The $2$-form Laplacian $d_1 d_1^*$ 
$$ L_3  = \left[ \begin{array}{c} 3  \end{array} \right] \; $$
has a single eigenvalue $\{ 3 \}$. \\

{\bf Dirac operator} $D=d+d^*$ is
\begin{tiny}
$$ D = \left[
                  \begin{array}{ccccccccccc}
                   0 & 0 & 0 & 0 & 0 & -1 & -1 & 0 & 0 & 0 & 0 \\
                   0 & 0 & 0 & 0 & 0 & 0 & 1 & -1 & 0 & 0 & 0 \\
                   0 & 0 & 0 & 0 & 0 & 1 & 0 & 1 & -1 & -1 & 0 \\
                   0 & 0 & 0 & 0 & 0 & 0 & 0 & 0 & 1 & 0 & 0 \\
                   0 & 0 & 0 & 0 & 0 & 0 & 0 & 0 & 0 & 1 & 0 \\
                   -1 & 0 & 1 & 0 & 0 & 0 & 0 & 0 & 0 & 0 & 1 \\
                   -1 & 1 & 0 & 0 & 0 & 0 & 0 & 0 & 0 & 0 & -1 \\
                   0 & -1 & 1 & 0 & 0 & 0 & 0 & 0 & 0 & 0 & -1 \\
                   0 & 0 & -1 & 1 & 0 & 0 & 0 & 0 & 0 & 0 & 0 \\
                   0 & 0 & -1 & 0 & 1 & 0 & 0 & 0 & 0 & 0 & 0 \\
                   0 & 0 & 0 & 0 & 0 & 1 & -1 & -1 & 0 & 0 & 0 \\
                  \end{array} \right] \; . $$
\end{tiny}

The form Laplacian or {\bf Laplace Beltrami operator} $ L=D^2 = (d+d^*)^* = d d^* + d^* d$ is
\begin{tiny}
$$ L  = \left[ \begin{array}{ccccccccccc}
                   2 & -1 & -1 & 0 & 0 & 0 & 0 & 0 & 0 & 0 & 0 \\
                   -1 & 2 & -1 & 0 & 0 & 0 & 0 & 0 & 0 & 0 & 0 \\
                   -1 & -1 & 4 & -1 & -1 & 0 & 0 & 0 & 0 & 0 & 0 \\
                   0 & 0 & -1 & 1 & 0 & 0 & 0 & 0 & 0 & 0 & 0 \\
                   0 & 0 & -1 & 0 & 1 & 0 & 0 & 0 & 0 & 0 & 0 \\
                   0 & 0 & 0 & 0 & 0 & 3 & 0 & 0 & -1 & -1 & 0 \\
                   0 & 0 & 0 & 0 & 0 & 0 & 3 & 0 & 0 & 0 & 0 \\
                   0 & 0 & 0 & 0 & 0 & 0 & 0 & 3 & -1 & -1 & 0 \\
                   0 & 0 & 0 & 0 & 0 & -1 & 0 & -1 & 2 & 1 & 0 \\
                   0 & 0 & 0 & 0 & 0 & -1 & 0 & -1 & 1 & 2 & 0 \\
                   0 & 0 & 0 & 0 & 0 & 0 & 0 & 0 & 0 & 0 & 3 \\
                  \end{array} \right] $$
\end{tiny}

{\bf Partial difference equations} \\
Any continuum PDE can be considered on graphs. They work especially on the rabbit.
The Laplacian $L=(d+d^*)^2$ can be either for the exterior derivative $d$ for forms or 
quadratic forms. Let $\lambda_k$ denote the {\bf eigenvalues} of the form Laplacian $L$: \\

The {\bf Wave equation}
$$ u_{tt} = L u \;  $$
has {\bf d'Alembert solution} $u(t) = \cos(Dt) u(0) + \sin(Dt) D^{-1} u'(0)$.  Using
$u(0) = \sum_n u_n f_n$ and $v(0) = u'(0) = \sum_n v_n f_n$ with {\bf eigenvectors} $f_n$ of
$L$, one has the solution
$$ u(t) = \sum_n \cos(\sqrt{\lambda_n} t) u_n f_n + \sin(\sqrt{\lambda_n} t) v_n f_n/\sqrt{\lambda_n} \; . $$

The {\bf Heat equation}
$$ u_t = L u  \;  $$
has the solution $u(t) = e^{-L t} u(0)$ or $\sum_n e^{-\lambda_k t} u_n f_n$,
where $u(0) = \sum_n u_n f_n$ is the eigenvector expansion of $u(0)$ with
respect to eigenvectors $f_n$ of $L$. \\

The {\bf Poisson equation}
$$ L u = v  $$
has the solution $u= L^{-1} v$, where $L^{-1}$ is the pseudo inverse of $L$. We can also assume $v$ to 
be perpendicular to the kernel of $L$. \\

The {\bf Laplace equation} 
$$ L u = 0 \;  $$
has on the scalar sector $\Omega^0$ only the locally constant functions as solutions. In general, these
are {\bf harmonic functions}. \\

{\bf Maxwell equations}
$$  dF = 0, d^* F = j $$
for a 2-form $F$ called {\bf electromagnetic field}.
In the case of the rabbit, $F$ is a number attached to the triangle.
$j$, the {\bf current} is a function assigned to edges. If $F=dA$ with a {\bf vector potential} $A$.
If $A$ is Coulomb gauged so that $d^* A = 0$, then we get the Poisson equation for one forms
$$ L_1 A = j  \; . $$
This solves the problem to find the electromagnetic field $F=dA$ from the current $j$. In the rabbit
case, since it is simply connected, there is no kernel of $L_1$ and we can just invert the matrix. \\

The electromagnetic field $F=dA$ defined on the rabbit is a number attached to the triangle:
\begin{center} {\bf Let there be light! } \end{center} 

\pagebreak

\section*{Appendix: an announcement}

{\bf A case study in Interaction cohomology \hfill Oliver Knill, 3/18/2016}. \cite{CaseStudy2016}.

\ssection{Simplicial Cohomology}
{\bf Simplicial cohomology} is defined by an {\bf exterior derivative} $dF(x)=F(dx)$
on {\bf valuation forms} $F(x)$ on subgraphs $x$ of a {\bf finite simple graph} $G$, where $dx$ is the 
{\bf boundary chain} of a simplex $x$. Evaluation $F(A)$ is {\bf integration} and $dF(A)=F(dA)$ is {\bf Stokes}.
Since $d^2=0$, the kernel of $d_p: \Omega^p \to \Omega^{p+1}$ contains the image of $d_{p-1}$. 
The vector space $H^p(G) = {\rm ker}(d_p)/{\rm im}(d_{p-1})$ is the $p$'th {\bf simplicial cohomology} 
of $G$. The {\bf Betti numbers} $b_p(G)={\rm dim}(H^p(G))$ define $\sum_p (-1)^p b_p$ which is
{\bf Euler characteristic} $\chi(G) = \sum_x (-1)^{{\rm dim}(x)}$, summing 
over all complete subgraphs $x$ of $G$. If $T$ is an automorphism of $G$, the {\bf Lefschetz number}, 
the super trace $\chi_T(G)$ of the induced map $U_T$ on $H^p(G)$
is equal to the sum $\sum_{T(x)=x} i_T(x)$, where $i_T(x)=(-1)^{{\rm dim}(x)} {\rm sign}(T|x)$
is the {\bf Brouwer index}. This is the {\bf Lefschetz fixed point theorem}. The {\bf Poincar\'e polynomial}
$p_G(x)=\sum_{k=0} {\rm dim}(H^k(G)) x^k$ satisfies $p_{G \times H}(x) = p_G(x) p_H(x)$ and $\chi(G)=p_G(-1)$
so that $\chi(G \times H) = \chi(G) \cdot \chi(H)$. For $T=Id$, the Lefschetz formula is {\bf Euler-Poincar\'e}.
With the {\bf Dirac operator} $D=d+d^*$ and {\bf Laplacian} $L=D^2$, discrete {\bf Hodge} tells that
$b_p(G)$ is the nullity of $L$ restricted to $p$-forms.
By {\bf McKean Singer super symmetry}, the positive Laplace spectrum on even-forms is the 
positive Laplace spectrum 
on odd-forms. The super trace ${\rm str}(L^k)$ is therefore zero for $k>0$ and $l(t)={\rm str}(\exp(-tL) U_T)$ 
with {\bf Koopman operator} $U_Tf=f(T)$ is $t$-invariant. This heat flow argument proves Lefschetz because
$l(0) = {\rm str}(U_T)$ is $\sum_{T(x)=x} i_T(x)$ and $\lim_{t \to \infty} l(t)=\chi_T(G)$ by Hodge.

\ssection{Interaction Cohomology}

Super counting ordered pairs of intersecting simplices $(x,y)$ gives the
{\bf Wu characteristic} $\omega(G)=\sum_{x \sim y} (-1)^{{\rm dim}(x)+{\rm dim}(y)}$. Like Euler
characteristic it is {\bf multiplicative} $\omega(G \times H) = \omega(G) \omega(H)$ 
and satisfies {\bf Gauss-Bonnet} and {\bf Poincar\'e-Hopf}. {\bf Quadratic interaction cohomology} 
is defined by the {\bf exterior derivative} 
$dF(x,y) = F(dx,y) + (-1)^{{\rm dim}(x)} F(x,dy)$ on functions $F$
on ordered pairs $(x,y)$ of intersecting simplices in $G$. If $b_p$ are the Betti numbers
of these interaction cohomology groups, then $\omega(G)=\sum_p (-1)^p b_p$ and
the {\bf Lefschetz formula} $\chi_T(G)=\sum_{(x,y)=(T(x),T(y))} i_T(x,y)$ holds
where $\chi_T(G)$ is the {\bf Lefschetz number}, the super trace of $U_T$ on cohomology
and $i_T(x,y) = (-1)^{{\rm dim}(x) + {\rm dim}(y)} {\rm sign}(T|x) {\rm sign}(T|y)$ is
the {\bf Brouwer index} introduced in the discrete in \cite{brouwergraph}. 
The heat proof works too. 
The {\bf interaction Poincar\'e polynomial} $p_G(x)=\sum_{k=0} {\rm dim}(H^k(G)) x^k$ again
satisfies $p_{G \times H}(x) = p_G(x) p_H(x)$.

\ssection{The Cylinder and the M\"oebius strip}

The cylinder $G$ and M\"obius strip $H$ are homotopic but not homeomorphic. As simplicial
cohomology is a homotopy invariant, it can not distinguish $H$ and $G$ and
$p_G(x)=p_H(x)$. But interaction cohomology can see it. The interaction Poincar\'e
polynomials of $G$ and $H$ are \fbox{$p_G(x)=x^2+x^3$} and \fbox{$p_H(x)=0$}.
Like {\bf Stiefel-Whitney classes}, interaction cohomology can distinguish the graphs.
While Stiefel-Whitney is defined for vector bundles, interaction cohomologies are defined
for all finite simple graphs. As it is invariant under Barycentric refinement 
$G \to G_1 = G \times K_1$, the cohomology works for continuum geometries like
manifolds or varieties. 

\begin{center} \parbox{16.8cm}{
\parbox{4cm}{
\scalebox{0.06}{\includegraphics{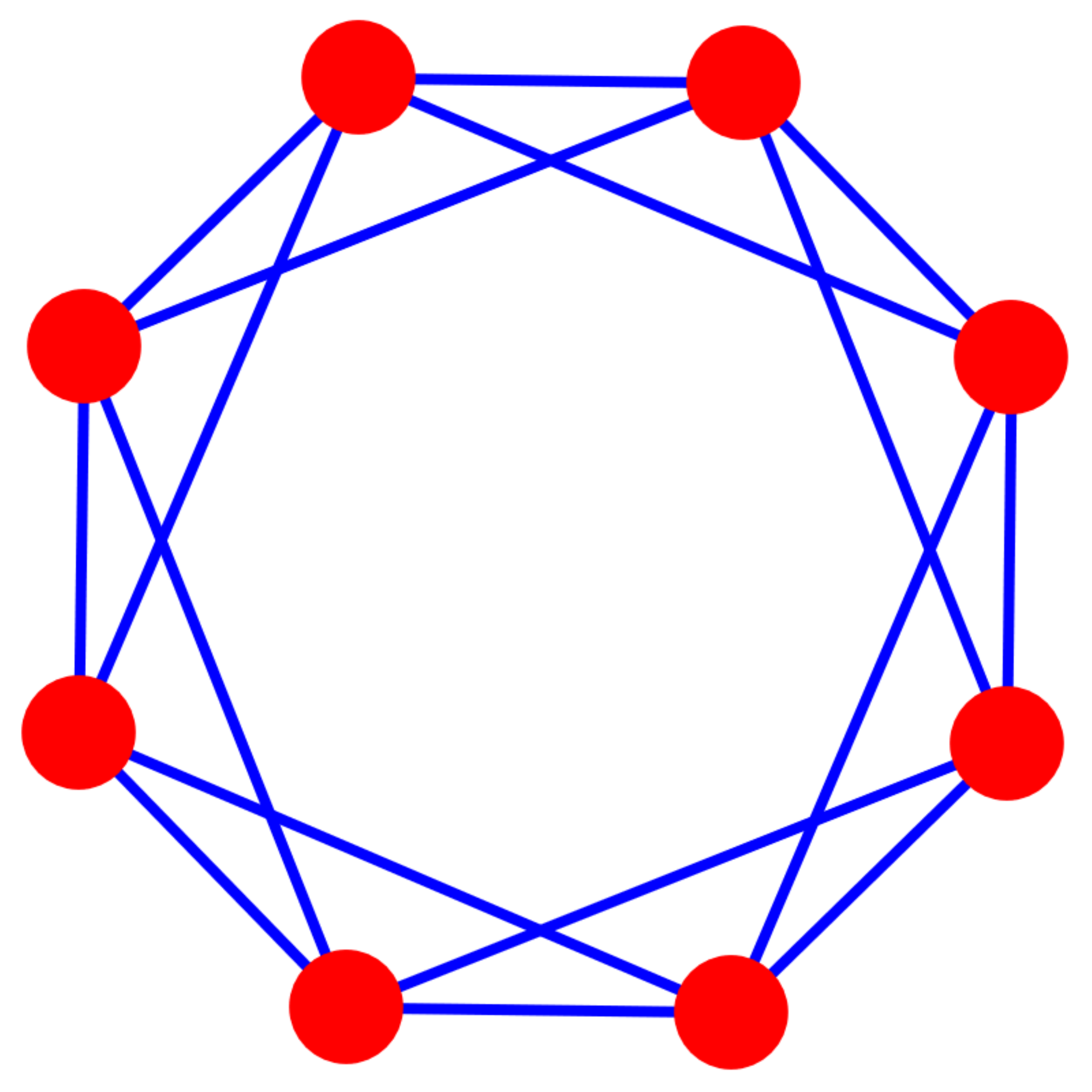}} 
\scalebox{0.10}{\includegraphics{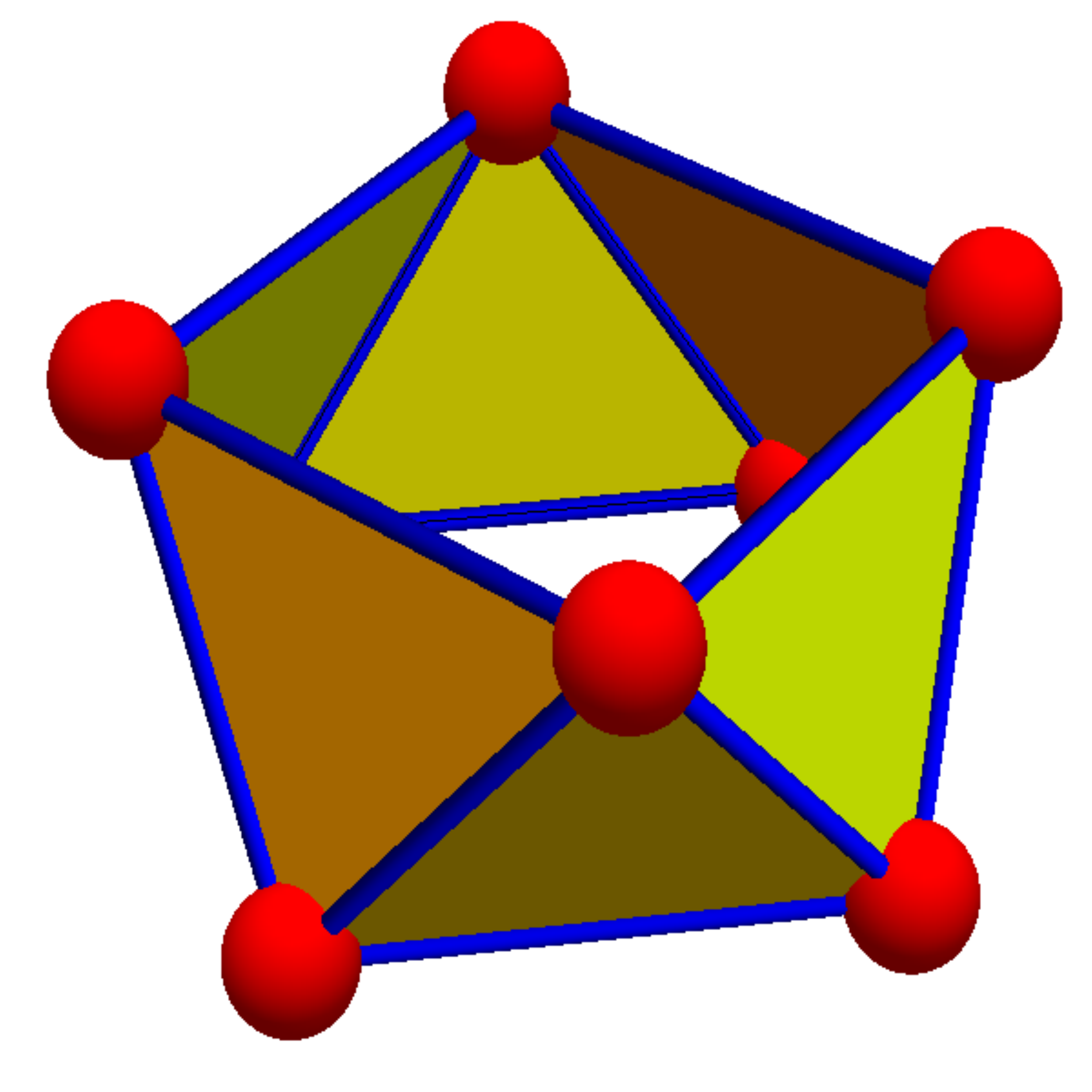}} 
}
\parbox{3.6cm}{
\begin{tiny}
The cylinder $G$ is an orientable graph with $\delta(G) = C_4 \dot{\cup} C_4$.
\end{tiny}
} \parbox{0.7cm}{ \hspace{0.7cm}}
\parbox{3.6cm}{
\begin{tiny}
The M\"obius strip $H$ is non-orientable with $\delta(H)=C_7$. 
\end{tiny}
}
\parbox{4cm}{
\scalebox{0.10}{\includegraphics{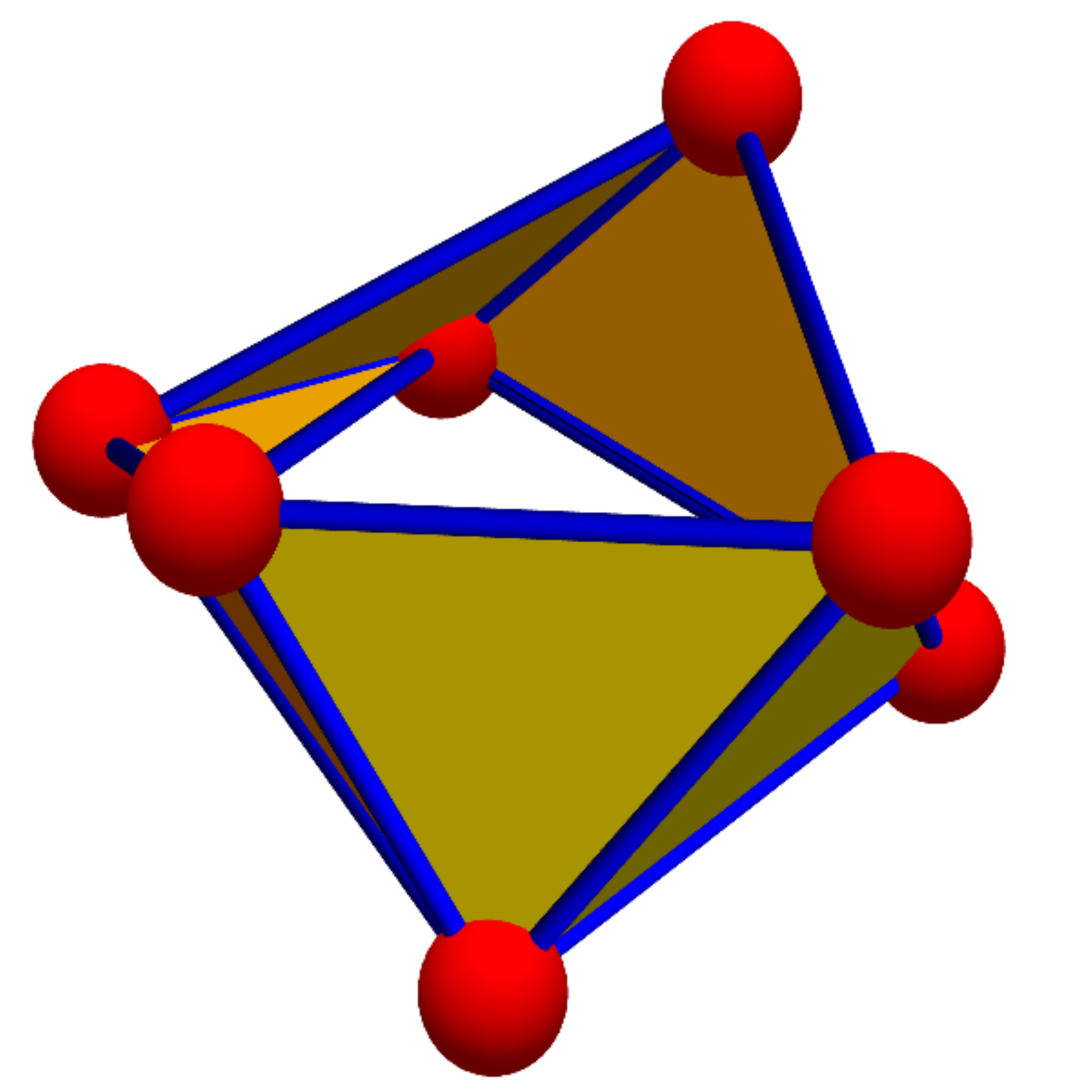}}
\scalebox{0.06}{\includegraphics{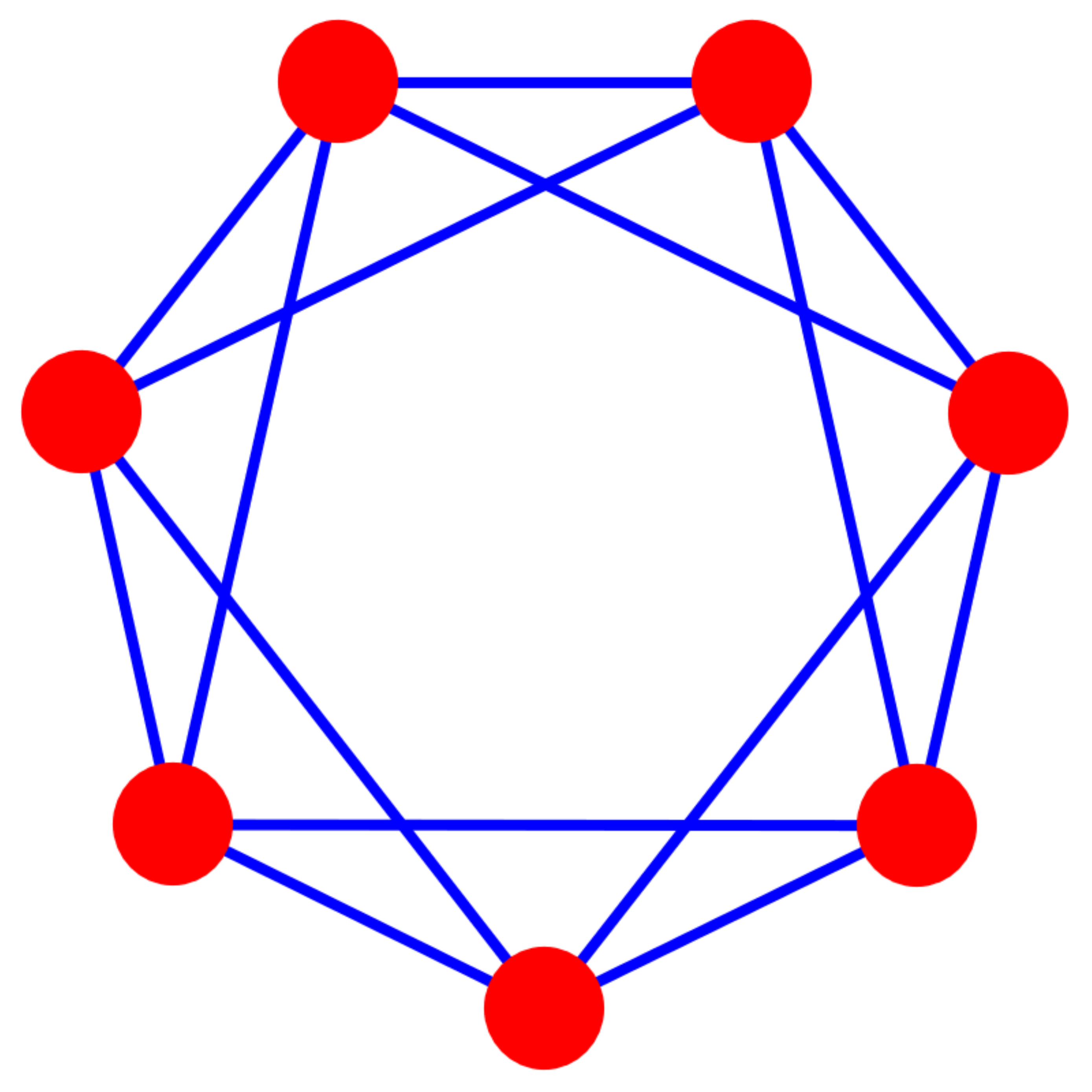}}\\
}
}
\end{center}

\ssection{Classical Calculus} 

Calculus on graphs either deals with {\bf valuations} or {\bf form valuations}. Of particular interest are
{\bf invariant linear valuations}, maps $X$ on non-oriented subgraphs $A$ of $G$ satisfying 
the {\bf valuation property} $X(A \cup B) + X(A \cap B) = X(A)+X(B)$ and $X(\emptyset)=0$ and
$X(A)=X(B)$ if $A$ and $B$ are isomorphic subgraphs. We don't assume invariance in general.
By {\bf discrete Hadwiger}, the vector space of 
invariant linear valuations has dimension $d+1$, where $d+1$ is the {\bf clique number}, the cardinality of the
vertex set a complete subgraph of $G$ can have. Linear invariant valuations are of the form $\sum_{k=0}^d X_k v_k(G)$,
where $v(G)=(v_0,v_1,\dots,v_d)$ is the $f$-vector of $G$. Examples of invariant valuations are $\chi(G)$ or $v_k(G)$
giving the number of $k$-simplices in $G$. An example of a non-invariant valuation is ${\rm deg}_a(A)$
giving the number of edges in $A$ hitting a vertex $a$. 
To define {\bf form valuations} which implements a discrete "integration of differential forms", 
one has to orient first simplices in $G$. No compatibility is required so that any graph admits 
a choice of orientation. The later is irrelevant for interesting quantities like cohomology.
A {\bf form valuation} $X$ is a function on oriented subgraphs of $G$ such that the valuation property holds and
$X(A) =-X(\overline{A})$ if $\overline{A}$ is the graph $A$ in which all orientations are reversed. 
Form valuations in particular change sign if a simplex changes orientation and when supported on $p$-simplices
play the role of $p$-forms. The defining identity $dF(A)=F(dA)$ is already Stokes theorem. If $A$ is a 
discrete $p$ surface made up of $(p+1)$-simplices with cancelling orientation so that $dA$ is a $p$-graph, then
this discretizes the continuum Stokes theorem but the result holds for all subgraphs of any graph
if one extends $F$ to {\bf chains} over $G$. For example, if $f_G=a b c + a b + b c + c a + a + b +c$ is the 
algebraic description of a triangle then $f_{dG} = b c - a c + a b + b - a + c -b  + a-c = a b +b c - a c$ is
only a chain. With the orientation $f_G=a b c + b a + b c + c a + a + b +c$ we would have got additionally 
the terms $2a-2b$. The vector space $\Omega^p(G)$ of {\bf all form valuations} on $G$ has dimension $v_p(G)$ as
we just have to give a value to each $p$-simplex to define $F$.  \\

We use a cylinder $G$, with $f$-vector $(8,16,8)$ which super sums to the Euler characteristic $\chi(G)=0$. 
An orientation on facets is fixed by giving graph algebraically like with 
$f_G = a + b + a b + c + b c + d + a d + c d + e + a e + b e + a b e + f + b f + c f + b c f + 
e f + b e f + g + c g + d g + c d g + f g + c f g + h + a h + d h + a d h + e h + a e h + g h + d g h$.
The automorphism group $A=D_8$ of $G$ has 16 elements. The Lefschetz numbers of the transformations are
$(0, 2, 2, 0, 2, 0, 0, 2, 0, 2, 2, 0, 2, 0, 0, 2)$. The average Lefschetz number is $\chi(G/A)=1$. 
The {\bf gradient} $d_0$ and {\bf curl} $d_1$ are
\begin{tiny}
$$ d_0=\left[ \begin{array}{cccccccc}
                   -1 & 1 & 0 & 0 & 0 & 0 & 0 & 0 \\
                   -1 & 0 & 0 & 1 & 0 & 0 & 0 & 0 \\
                   -1 & 0 & 0 & 0 & 1 & 0 & 0 & 0 \\
                   -1 & 0 & 0 & 0 & 0 & 0 & 0 & 1 \\
                   0 & -1 & 1 & 0 & 0 & 0 & 0 & 0 \\
                   0 & -1 & 0 & 0 & 1 & 0 & 0 & 0 \\
                   0 & -1 & 0 & 0 & 0 & 1 & 0 & 0 \\
                   0 & 0 & -1 & 1 & 0 & 0 & 0 & 0 \\
                   0 & 0 & -1 & 0 & 0 & 1 & 0 & 0 \\
                   0 & 0 & -1 & 0 & 0 & 0 & 1 & 0 \\
                   0 & 0 & 0 & -1 & 0 & 0 & 1 & 0 \\
                   0 & 0 & 0 & -1 & 0 & 0 & 0 & 1 \\
                   0 & 0 & 0 & 0 & -1 & 1 & 0 & 0 \\
                   0 & 0 & 0 & 0 & -1 & 0 & 0 & 1 \\
                   0 & 0 & 0 & 0 & 0 & -1 & 1 & 0 \\
                   0 & 0 & 0 & 0 & 0 & 0 & -1 & 1 \\
                  \end{array}
                  \right],
     d_1=\left[ 
                  \begin{array}{cccccccccccccccc}
                   -1 & 0 & 1 & 0 & 0 & -1 & 0 & 0 & 0 & 0 & 0 & 0 & 0 & 0 & 0 & 0 \\
                   0 & -1 & 0 & 1 & 0 & 0 & 0 & 0 & 0 & 0 & 0 & -1 & 0 & 0 & 0 & 0 \\
                   0 & 0 & -1 & 1 & 0 & 0 & 0 & 0 & 0 & 0 & 0 & 0 & 0 & -1 & 0 & 0 \\
                   0 & 0 & 0 & 0 & -1 & 0 & 1 & 0 & -1 & 0 & 0 & 0 & 0 & 0 & 0 & 0 \\
                   0 & 0 & 0 & 0 & 0 & -1 & 1 & 0 & 0 & 0 & 0 & 0 & -1 & 0 & 0 & 0 \\
                   0 & 0 & 0 & 0 & 0 & 0 & 0 & -1 & 0 & 1 & -1 & 0 & 0 & 0 & 0 & 0 \\
                   0 & 0 & 0 & 0 & 0 & 0 & 0 & 0 & -1 & 1 & 0 & 0 & 0 & 0 & -1 & 0 \\
                   0 & 0 & 0 & 0 & 0 & 0 & 0 & 0 & 0 & 0 & -1 & 1 & 0 & 0 & 0 & -1 \\
                  \end{array} \right] \; . $$
\end{tiny}
The Laplacian $L=(d+d^*)^2$ has blocks $L_0=d_0^* d_0$, which is the Kirchhoff matrix.
The form Laplacian $L_1=d_0 d_0^* + d_1^* d_1$ is a matrix on $1$-forms. Together with 
the form Laplacian  $L_2 = d_1 d_1^*$, the matrix $L$ is
\begin{tiny}
$$  \arraycolsep=2pt\def\arraystretch{1.0}
   \left[
      \begin{array}{cccccccccccccccccccccccccccccccc}
0&-1&4&-1&0&-1&-1&0&0&0&0&0&0&0&0&0&0&0&0&0&0&0&0&0&0&0&0&0&0&0&0&0\\
-1&0&-1&4&0&0&-1&-1&0&0&0&0&0&0&0&0&0&0&0&0&0&0&0&0&0&0&0&0&0&0&0&0\\
-1&-1&0&0&4&-1&0&-1&0&0&0&0&0&0&0&0&0&0&0&0&0&0&0&0&0&0&0&0&0&0&0&0\\
0&-1&-1&0&-1&4&-1&0&0&0&0&0&0&0&0&0&0&0&0&0&0&0&0&0&0&0&0&0&0&0&0&0\\
0&0&-1&-1&0&-1&4&-1&0&0&0&0&0&0&0&0&0&0&0&0&0&0&0&0&0&0&0&0&0&0&0&0\\
-1&0&0&-1&-1&0&-1&4&0&0&0&0&0&0&0&0&0&0&0&0&0&0&0&0&0&0&0&0&0&0&0&0\\
0&0&0&0&0&0&0&0&3&1&0&1&-1&0&-1&0&0&0&0&0&0&0&0&0&0&0&0&0&0&0&0&0\\
0&0&0&0&0&0&0&0&1&3&1&0&0&0&0&1&0&0&-1&0&0&0&0&0&0&0&0&0&0&0&0&0\\
0&0&0&0&0&0&0&0&0&1&4&0&0&0&0&0&0&0&0&0&-1&0&0&0&0&0&0&0&0&0&0&0\\
0&0&0&0&0&0&0&0&1&0&0&4&0&0&0&0&0&0&0&0&0&0&0&1&0&0&0&0&0&0&0&0\\
0&0&0&0&0&0&0&0&-1&0&0&0&3&1&0&-1&0&-1&0&0&0&0&0&0&0&0&0&0&0&0&0&0\\
0&0&0&0&0&0&0&0&0&0&0&0&1&4&0&0&0&0&0&0&0&-1&0&0&0&0&0&0&0&0&0&0\\
0&0&0&0&0&0&0&0&-1&0&0&0&0&0&4&0&0&0&0&0&0&0&-1&0&0&0&0&0&0&0&0&0\\
0&0&0&0&0&0&0&0&0&1&0&0&-1&0&0&3&1&0&0&-1&0&0&0&0&0&0&0&0&0&0&0&0\\
0&0&0&0&0&0&0&0&0&0&0&0&0&0&0&1&4&0&0&0&1&0&0&0&0&0&0&0&0&0&0&0\\
0&0&0&0&0&0&0&0&0&0&0&0&-1&0&0&0&0&4&0&0&0&0&0&-1&0&0&0&0&0&0&0&0\\
0&0&0&0&0&0&0&0&0&-1&0&0&0&0&0&0&0&0&4&0&0&0&1&0&0&0&0&0&0&0&0&0\\
0&0&0&0&0&0&0&0&0&0&0&0&0&0&0&-1&0&0&0&4&0&1&0&0&0&0&0&0&0&0&0&0\\
0&0&0&0&0&0&0&0&0&0&-1&0&0&0&0&0&1&0&0&0&3&1&-1&0&0&0&0&0&0&0&0&0\\
0&0&0&0&0&0&0&0&0&0&0&0&0&-1&0&0&0&0&0&1&1&3&0&1&0&0&0&0&0&0&0&0\\
0&0&0&0&0&0&0&0&0&0&0&0&0&0&-1&0&0&0&1&0&-1&0&3&-1&0&0&0&0&0&0&0&0\\
0&0&0&0&0&0&0&0&0&0&0&1&0&0&0&0&0&-1&0&0&0&1&-1&3&0&0&0&0&0&0&0&0\\
0&0&0&0&0&0&0&0&0&0&0&0&0&0&0&0&0&0&0&0&0&0&0&0&3&0&-1&0&1&0&0&0\\
0&0&0&0&0&0&0&0&0&0&0&0&0&0&0&0&0&0&0&0&0&0&0&0&0&3&1&0&0&0&0&-1\\
0&0&0&0&0&0&0&0&0&0&0&0&0&0&0&0&0&0&0&0&0&0&0&0&-1&1&3&0&0&0&0&0\\
0&0&0&0&0&0&0&0&0&0&0&0&0&0&0&0&0&0&0&0&0&0&0&0&0&0&0&3&1&0&1&0\\
0&0&0&0&0&0&0&0&0&0&0&0&0&0&0&0&0&0&0&0&0&0&0&0&1&0&0&1&3&0&0&0\\
0&0&0&0&0&0&0&0&0&0&0&0&0&0&0&0&0&0&0&0&0&0&0&0&0&0&0&0&0&3&1&1\\
0&0&0&0&0&0&0&0&0&0&0&0&0&0&0&0&0&0&0&0&0&0&0&0&0&0&0&1&0&1&3&0\\
0&0&0&0&0&0&0&0&0&0&0&0&0&0&0&0&0&0&0&0&0&0&0&0&0&-1&0&0&0&1&0&3\\
   \end{array} \right]  \; . $$
\end{tiny}

For the M\"obius strip $H$ with $f$-vector $(7,14,7)$ 
we could make an edge refinement to get a graph $\tilde{H}$ with 
$f$-vector $(8,16,8)$ 
which after an edge collapse goes back to $H$. The graphs $H$ and $\tilde{H}$
are topologically equivalent in the discrete topological sense (there is a covering for both graphs
for which the nerve graphs are identical and such that the elements as well as their intersections are
all two dimensional.)
Both simplicial as well as quadratic intersection cohomology are the same for $H$ and $\tilde{H}$. The 
graph $H$ has the automorphism group $D_7$ but $\tilde{H}$ only $Z_2$. The Lefschetz
numbers of $H$ are $(0, 2, 2, 0, 2, 0, 2, 0, 2, 0, 2, 0, 0, 2)$. The Lefschetz
numbers of $\tilde{H}$ are $(0,2)$. The automorphism $T=(1, 8, 7, 6, 5, 4, 3, 2)$ which generates
the automorphism group of $\tilde{H}$ has the Lefschetz number $2$. There are $4$ fixed points, they 
are the fixed vertices $1,5$ and the fixed edges $(1,5)$ and $(7,3)$. All Brouwer indices are $1$
except for $(1,5)$ which has index $-1$. (It is not flipped by the transformation and has
dimension $1$). We continue to work with the graph $H$ which has the algebraic 
description $f_H = a + b + a b + c + b c + d + a d + c d + e + a e + b e + a b e + 
d e + a d e + f + b f + c f + b c f + e f + b e f + g + a g + c g + d g + a d g + c d g + f g + c f g$. \\

Here are the gradient and curl of $H$:
\begin{tiny}
$$  
 d_0=\left[  \begin{array}{ccccccc}
                   -1 & 0 & 0 & 0 & 1 & 0 & 0 \\
                   -1 & 0 & 0 & 0 & 0 & 0 & 1 \\
                   -1 & 1 & 0 & 0 & 0 & 0 & 0 \\
                   -1 & 0 & 0 & 1 & 0 & 0 & 0 \\
                   0 & -1 & 0 & 0 & 0 & 1 & 0 \\
                   0 & -1 & 1 & 0 & 0 & 0 & 0 \\
                   0 & -1 & 0 & 0 & 1 & 0 & 0 \\
                   0 & 0 & 1 & -1 & 0 & 0 & 0 \\
                   0 & 0 & 0 & -1 & 1 & 0 & 0 \\
                   0 & 0 & 0 & -1 & 0 & 0 & 1 \\
                   0 & 0 & -1 & 0 & 0 & 0 & 1 \\
                   0 & 0 & -1 & 0 & 0 & 1 & 0 \\
                   0 & 0 & 0 & 0 & -1 & 1 & 0 \\
                   0 & 0 & 0 & 0 & 0 & -1 & 1 \\
                  \end{array} \right] \; , 
 d_1=\left[
                 \begin{array}{cccccccccccccc}
                  1 & 0 & -1 & 0 & 0 & 0 & -1 & 0 & 0 & 0 & 0 & 0 & 0 & 0 \\
                  1 & 0 & 0 & -1 & 0 & 0 & 0 & 0 & -1 & 0 & 0 & 0 & 0 & 0 \\
                  0 & 1 & 0 & -1 & 0 & 0 & 0 & 0 & 0 & -1 & 0 & 0 & 0 & 0 \\
                  0 & 0 & 0 & 0 & 1 & -1 & 0 & 0 & 0 & 0 & 0 & -1 & 0 & 0 \\
                  0 & 0 & 0 & 0 & 1 & 0 & -1 & 0 & 0 & 0 & 0 & 0 & -1 & 0 \\
                  0 & 0 & 0 & 0 & 0 & 0 & 0 & -1 & 0 & 1 & -1 & 0 & 0 & 0 \\
                  0 & 0 & 0 & 0 & 0 & 0 & 0 & 0 & 0 & 0 & 1 & -1 & 0 & -1 \\
                 \end{array} \right]  \; . $$
\end{tiny}
The Hodge Laplacian of $L$ is the $28 \times 28$ matrix (we write $i=-1$
\begin{tiny}
$$   \left[  \begin{array}{cccccccccccccccccccccccccccc}
4&i&0&i&i&0&i&0&0&0&0&0&0&0&0&0&0&0&0&0&0&0&0&0&0&0&0&0\\
i&4&i&0&i&i&0&0&0&0&0&0&0&0&0&0&0&0&0&0&0&0&0&0&0&0&0&0\\
0&i&4&i&0&i&i&0&0&0&0&0&0&0&0&0&0&0&0&0&0&0&0&0&0&0&0&0\\
i&0&i&4&i&0&i&0&0&0&0&0&0&0&0&0&0&0&0&0&0&0&0&0&0&0&0&0\\
i&i&0&i&4&i&0&0&0&0&0&0&0&0&0&0&0&0&0&0&0&0&0&0&0&0&0&0\\
0&i&i&0&i&4&i&0&0&0&0&0&0&0&0&0&0&0&0&0&0&0&0&0&0&0&0&0\\
i&0&i&i&0&i&4&0&0&0&0&0&0&0&0&0&0&0&0&0&0&0&0&0&0&0&0&0\\
0&0&0&0&0&0&0&4&1&0&0&0&0&0&0&0&0&0&0&i&0&0&0&0&0&0&0&0\\
0&0&0&0&0&0&0&1&3&1&0&0&0&0&0&0&0&1&0&0&1&0&0&0&0&0&0&0\\
0&0&0&0&0&0&0&0&1&3&1&i&i&0&0&0&0&0&0&0&0&0&0&0&0&0&0&0\\
0&0&0&0&0&0&0&0&0&1&4&0&0&0&i&0&0&0&0&0&0&0&0&0&0&0&0&0\\
0&0&0&0&0&0&0&0&0&i&0&4&0&0&0&0&0&0&0&0&i&0&0&0&0&0&0&0\\
0&0&0&0&0&0&0&0&0&i&0&0&3&1&1&0&0&i&0&0&0&0&0&0&0&0&0&0\\
0&0&0&0&0&0&0&0&0&0&0&0&1&4&0&1&0&0&0&0&0&0&0&0&0&0&0&0\\
0&0&0&0&0&0&0&0&0&0&i&0&1&0&3&1&0&0&i&0&0&0&0&0&0&0&0&0\\
0&0&0&0&0&0&0&0&0&0&0&0&0&1&1&3&1&0&0&i&0&0&0&0&0&0&0&0\\
0&0&0&0&0&0&0&0&0&0&0&0&0&0&0&1&4&0&0&0&1&0&0&0&0&0&0&0\\
0&0&0&0&0&0&0&0&1&0&0&0&i&0&0&0&0&4&0&0&0&0&0&0&0&0&0&0\\
0&0&0&0&0&0&0&0&0&0&0&0&0&0&i&0&0&0&4&1&0&0&0&0&0&0&0&0\\
0&0&0&0&0&0&0&i&0&0&0&0&0&0&0&i&0&0&1&3&i&0&0&0&0&0&0&0\\
0&0&0&0&0&0&0&0&1&0&0&i&0&0&0&0&1&0&0&i&3&0&0&0&0&0&0&0\\
0&0&0&0&0&0&0&0&0&0&0&0&0&0&0&0&0&0&0&0&0&3&1&0&0&1&0&0\\
0&0&0&0&0&0&0&0&0&0&0&0&0&0&0&0&0&0&0&0&0&1&3&1&0&0&0&0\\
0&0&0&0&0&0&0&0&0&0&0&0&0&0&0&0&0&0&0&0&0&0&1&3&0&0&i&0\\
0&0&0&0&0&0&0&0&0&0&0&0&0&0&0&0&0&0&0&0&0&0&0&0&3&1&0&1\\
0&0&0&0&0&0&0&0&0&0&0&0&0&0&0&0&0&0&0&0&0&1&0&0&1&3&0&0\\
0&0&0&0&0&0&0&0&0&0&0&0&0&0&0&0&0&0&0&0&0&0&0&i&0&0&3&i\\
0&0&0&0&0&0&0&0&0&0&0&0&0&0&0&0&0&0&0&0&0&0&0&0&1&0&i&3\\
\end{array}\right] \; . $$
\end{tiny}
It consists of three blocks. The first, $L_0 = d_0^* d_0$ is a $7 \times 7$ matrix which is the Kirchhoff matrix of $H$. 
The second, $L_1 = d_1^* d_1 + d_0 d_0^*$ is a $14 \times 14$ matrix which describes $1$-forms, functions on ordered edges.
Finally, the third $L_2 = d_1 d_1^*$ deals with $2$-forms, functions on oriented triangles of $H$. Note that an orientation on
the triangles of $H$ is not the same than having $H$ oriented. The later would mean that we find an orientation on triangles
which is compatible on intersections. We know that this is not possible for the strip. The matrices $d_0,d_1$ depend on a choice of 
orientation as usually in linear algebra where the matrix to a transformation depends on a basis. But the Laplacian $L$ does not depend
on the orientation. 

\ssection{Quadratic interaction calculus}

Quadratic interaction calculus is the next case after linear calculus. The later has to a great deal known by
Kirchhoff and definitely by Poincar\'e, who however seems not have defined $D$, the form Laplacian $L$, nor 
have used Hodge or super symmetry relating the spectra on even and odd dimensional forms. Nor did Poincar\'e work
with graphs but with simplicial complexes, a detail which is mainly language as the Barycentric refinement of any
abstract finite simplicial complex is already the Whitney complex of a finite simple graph. 
Quadratic interaction calculus deals with functions $F(x,y)$, where $x,y$ are both subgraphs. It should be seen
as a form version of quadratic valuations \cite{valuation}. The later were defined to be functions 
$X(A,B)$ on pairs of subgraphs of $G$ such that for fixed $B$, the map $A \to X(A,B)$ and for fixed $A$, 
the map $B \to X(A,B)$ are both linear (invariant) valuations. Like for any calculus 
flavor, there is a valuation version which is orientation oblivious and a form valuation version which 
cares about orientations. When fixing one variable, we want $x \to F(x,y)$ or $y \to F(y,x)$ to be {\bf valuations}: 
$F(x,y \cup z) + F(x,y \cap z) = F(x,y)+F(x,z)$ and 
$F(x,y)=-F(x,\overline{y})=-F(\overline{x},y)$ 
if $\overline{y}$ is the graph $y$ for which all orientations are reversed and the union
and intersection of graphs simultaneously intersect the vertex set and edge set. 
Note that we do {\bf not} require $F(x,y)=-F(y,x)$.
But unlike as for quadratic valuations (which are oblivious to orientations), 
we want the functions to change sign if one of the orientations of $x$ or $y$ changes. 
Again, the choice of orientation is just a choice of basis and not relevant for any 
results we care about. Calculus is defined as soon as an exterior derivative is given 
(integration is already implemented within the concept of valuation or form valuation). 
In our case, it is {\bf defined} via Stokes theorem 
$dF(x,y) = F(dx,y) + (-1)^{{\rm dim}(x)} F(x,dy)$ for simplices first and then for general subgraphs $x,y$ 
by the valuation properties. While in linear calculus, integration is the evaluation of 
the valuation on subgraphs, in quadratic interaction calculus, we evaluate=integrate on pairs of subgraphs. 
The reason for the name ``interaction cohomology"
is because for $F$ to be non-zero, the subgraphs $x,y$ are required to be {\bf interacting} meaning having a
non-empty intersection. (The word ``Intersection" would also work but the name ``Intersection cohomology" 
has been taken already in the continuum). We can think about 
$F(A,B)$ as a type of {\bf intersection number} of two oriented interacting subgraphs $A,B$ of the graph $G$
and as $F(A,A)$ as the {\bf self intersection number} $F(A,A)$. An other case is the intersection of 
the diagonal $A=\bigcup_x \{(x,x)\}$ and graph $B=\bigcup_x \{ (x,T(x))\}$ 
in the product $G \times G$ of an automorphism $T$ of $G$. The form version of the Wu intersection number 
$\omega(A,B)$ is then the Lefschetz number $\chi_T(G)$.
A particularly important example of a self-intersection number is the Wu characteristic 
$\omega(G)=\omega(G,G)$ which is quadratic valuation, not a form valuation.
Our motivation to define interaction cohomology was to get a Euler-Poincar\'e formula for
Wu characteristic. It turns out that Euler-Poincar\'e automatically and always generalizes to the
Lefschetz fixed point theorem, as the heat flow argument has shown; Euler-Poincar\'e is 
always just the case in which $T$ is the identity automorphism. \\

Lets look now at {\bf quadratic interaction cohomology} for the cylinder graph $G$ and the
M\"obius graph $H$. We believe this case demonstrates well already how quadratic interaction cohomology 
allows in an algebraic way to distinguish graphs which traditional cohomology can not. 
The $f$-matrices of the graphs are 
$$ V(G) = \left[
               \begin{array}{ccc}
                8 & 32 & 24 \\
                32 & 112 & 72 \\
                24 & 72 & 40 \\
                \end{array}
               \right],
   V(H) =  \left[
                 \begin{array}{ccc}
                  7 & 28 & 21 \\
                  28 & 98 & 63 \\
                  21 & 63 & 35 \\
                 \end{array}
                 \right], 
   V(\tilde{H}) =  \left[
                  \begin{array}{ccc}
                   8 & 32 & 24 \\
                   32 & 114 & 74 \\
                   24 & 74 & 42 \\
                  \end{array}
                  \right]   \; . $$
The $f$-matrix is the matrix $V$ for which $V_{ij}$ counts the number of pairs $(x,y)$, where
$x$ is an $i$-simplex, $y$ is a $j$-simplex and $x,y$ intersect. 
The Laplacian for quadratic intersection cohomology of 
the cylinder is a $416 \times 416$ matrix because there are a total of $416 = \sum_{i,j} V_{ij}$ pairs of simplices 
which do intersect in the graph $G$. 
For the M\"obius strip, it is a $364 \times 364$ matrix which splits into $5$ blocks $L_0,L_1,L_2,L_3,L_4$. 
Block $L_p$ corresponds to the interaction pairs $(x,y)$ for which ${\rm dim}(x)+{\rm dim}(y)$ is equal to $p$. 
The scalar interaction Laplacian $L_0$ for the cylinder is the diagonal matrix ${\rm Diag}(8,8,8,8,8,8,8)$. 
For the M\"obius strip $H$, it is the matrix $L_0(H) = {\rm Diag}(8,8,8,8,8,8)$. 
The diagonal entries of $L_0(H)$ depend only  on the vertex degrees. 
In general, for any graph, if all vertex degrees are positive, then the scalar interaction Laplacian 
$L_0$ of the graph is invertible. 
For the cylinder, the block $L_1$ is an invertible $64 \times 64$ matrix because there are 
$V_{12}+V_{21}=64$ vertex-edge or edge-vertex interactions. Its determinant is 
$2^{58} 3^8 \cdot 5 \cdot 7 \cdot 11 \cdot 13 \cdot 17^8 103^2 373^2 2089^2$. 
For the M\"obius strip, the block $L_1$ is a $56 \times 56$ matrix with determinant
$2^{46} 3^7 5 7^3 17^7 42924041^2$. \\

For the cylinder graph $G$, the block $L_2(G)$ is a $160 \times 160$ matrix, which has a 
$1$-dimensional kernel spanned by the vector
$[$ 0, -5, -5, 10, 10, 2, 2, 5, 0, -5, 2, 10, 10, 2, 5, 0, -5, -2, -2, -10, -10, 5, 5, 0, 2, 10, 
10, 2, -10, 2, 0, 9, -2, 9, 10, -10, -2, -9, 0, -9, -10, 2, -2, -10, 9, 0, 9, -10, -2, -10, -2, 
-9, 0, -2, -9, -10, 2, 10, 10, 2, 0, 5, 5, -2, -10, 9, 0, 9, -10, -2, 2, -10, -9, 0, -2, -9, -10, 
2, 10, -5, 10, 2, 0, 5, -2, 10, -9, 0, -9, 10, -2, 10, -2, 9, 9, 0, 2, -10, -10, -2, -5, -10, -2, 
0, 5, 2, 10, -5, 2, 10, -5, 0, -7, -7, -8, -7, -7, -8, 8, 7, 7, -7, -7, -8, 8, 7, 7, 7, 7, 8, 
-8, -7, -7, 8, 7, 7, -7, 7, -8, -7, -7, 8, -7, -7, 8, -7, 7, 8, -8, 7, -7, -8, 7, 7, 
-8, 7, 7, 8, -7, 7 $]$. This vector is associated to edge-edge and triangle-vertex interactions. 
So far, harmonic forms always had physical relevance. We don't know what it is for this
interaction calculus. \\

On the other hand, for the M\"obius strip $H$, the interaction form Laplacian $L_2(H)$ is a 
$140 \times 140$ matrix which is invertible! Its determinant is 
$2^{42} 3^{14} 5^3 7^6 11 \cdot 17^{14} 11087^2 212633^2 42924041^4$.

The quadratic form Laplacian $L_3$ which describes the edge-triangle interactions. 
For the cylinder $G$, the interaction form Laplacian $L_3(G)$  is a $144 \times 144$ matrix.
It has a $1$-dimensional kernel spanned by the vector 
$[$ 4, 3, -6, -3, 6, 3, 4, -6, 3, -6, -3, -3, 4, -6, 6, 3, 4, -6, -3, -6, 2, -3, -3, 2, -2, 
-3, 2, -3, 3, 2, -2, -3, -2, 3, -3, 2, 6, 6, -4, -3, 3, 3, 2, -2, -3, -2, 3, 3, 2, 6, -3, 6, 
-4, -3, -3, 2, -2, 3, 3, -2, 3, -2, -6, 3, 6, -4, 3, 6, -3, -6, 3, -4, -4, -3, 3, -2, 2, -3, 
-6, 3, 6, -3, -4, -3, 3, -2, 2, -6, -3, -6, 6, 6, 3, -2, 2, -3, 4, 3, -3, -3, 3, -4, 3, -2, 
2, -6, -3, -6, 6, 6, 3, -2, 3, 2, -3, 4, 3, 3, -4, 3, 3, -3, -2, 2, -6, 6, -6, 6, -2, 3, -3, 
2, -3, 4, -3, -6, 6, 3, -2, 3, -3, 2, -3, 4 $]$. \\

The surprise, which led us write this down quickly is that
for the M\"obius strip $H$, the interaction form Laplacian $L_3(H)$ 
is a $126 \times 126$ matrix which is {\bf invertible} and has determinant
$2^{28} 3^7 5^3 7^3 11^2 17^7 11087^4 212633^4 42924041^2$. Since $L_3$ is invertible too,
there is also no cohomological restrictions on the $p=3$ level for the M\"obius strip.
The quadratic form Laplacian is invertible. 
This is in contrast to the cylinder, where cohomological constraints exist on this level.
The interaction cohomology can detect topological features, which simplicial cohomology 
can not. \\

The fact that determinants
of $L_3,L_2$ have common prime factors is not an accident and can be explained by 
{\bf super symmetry}, as the union of nonzero eigenvalues of $L_0,L_2,L_4$ are the same than the union
of the nonzero eigenvalues of $L_1,L_3$. \\

Finally, we look at the $L_4$ block, which describes triangle-triangle interactions. 
For the cylinder $G$, this interaction Laplacian  is a $40 \times 40$ matrix which has
the  determinant $2^{26} 3^3 5 \cdot 11^3 23^2 29^2 71^2 241^2$.
For the M\"obius strip, it is a 
$35 \times 35$ matrix which has determinant 
$2^{11} 5 \cdot 11 \cdot 11087^2 212633^2$.  \\


The derivative $d_0$ for the Moebius strip $H$ is the $56 \times 7$ matrix 
\begin{tiny}
$$ \arraycolsep=1pt \def\arraystretch{1}
    d_0 = \left[
\begin{array}{ccccccc}
                   -1 & 0 & 0 & 0 & 0 & 0 & 0 \\
                   0 & 1 & 0 & 0 & 0 & 0 & 0 \\
                   0 & -1 & 0 & 0 & 0 & 0 & 0 \\
                   0 & 0 & 1 & 0 & 0 & 0 & 0 \\
                   0 & -1 & 0 & 0 & 0 & 0 & 0 \\
                   0 & 0 & 0 & 1 & 0 & 0 & 0 \\
                   -1 & 0 & 0 & 0 & 0 & 0 & 0 \\
                   0 & 0 & 0 & 1 & 0 & 0 & 0 \\
\end{array}
                  \right] \; .$$
\end{tiny}
It takes as arguments a list of function values on pairs of self-interacting vertices $(x,x)$ and gives
back a function on the list of pairs of interacting vertex-edge or edge-vertex pairs $(x,y)$.
The exterior derivative $d_1$ for $H$ is a $140 \times 56$ matrix.
We use now the notation $i=-1$ for formatting reasons:

\begin{tiny}
$$\arraycolsep=0.3pt\def\arraystretch{0.1}
\left[
\begin{array}{cccccccccccccccccccccccccccccccccccccccccccccccccccccccc}
i&1&0&0&0&0&0&0&0&0&0&0&0&0&0&0&0&0&0&0&0&0&0&0&0&0&0&0&i&0&0&0&1&0&0&0&0&0&0&0&0&0&0&0&0&0&0&0&0&0&0&0&0&0&0&0\\
0&i&0&0&0&0&0&0&0&0&0&0&0&0&0&0&0&0&0&0&0&0&0&0&0&0&0&0&0&0&0&0&0&1&0&0&0&0&0&0&0&0&0&0&0&0&0&0&0&0&0&0&0&0&0&0\\
0&i&0&0&0&0&0&0&0&0&0&0&0&0&0&0&0&0&0&0&0&0&0&0&0&0&0&0&0&0&0&0&0&0&1&0&0&0&0&0&0&0&0&0&0&0&0&0&0&0&0&0&0&0&0&0\\
i&0&0&0&0&0&0&0&0&0&0&0&0&0&0&0&0&0&0&0&0&0&0&0&0&0&0&0&0&i&0&0&0&0&0&0&0&0&0&0&0&0&0&0&0&0&0&0&0&0&0&0&0&0&0&0\\
i&0&0&0&0&0&0&0&0&0&0&0&0&0&0&0&0&0&0&0&0&0&0&0&0&0&0&0&0&0&i&0&0&0&0&0&0&0&0&0&0&0&0&0&0&0&0&0&0&0&0&0&0&0&0&0\\
0&i&0&0&0&0&0&0&0&0&0&0&0&0&0&0&0&0&0&0&0&0&0&0&0&0&0&0&0&0&0&0&0&0&0&1&0&0&0&0&0&0&0&0&0&0&0&0&0&0&0&0&0&0&0&0\\
i&0&0&0&0&0&0&0&0&0&0&0&0&0&0&0&0&0&0&0&0&0&0&0&0&0&0&0&0&0&0&i&0&0&0&0&0&0&0&0&0&0&0&0&0&0&0&0&0&0&0&0&0&0&0&0\\
0&0&1&0&0&0&0&0&0&0&0&0&0&0&0&0&0&0&0&0&0&0&0&0&0&0&0&0&0&0&0&0&i&0&0&0&0&0&0&0&0&0&0&0&0&0&0&0&0&0&0&0&0&0&0&0\\
\end{array}
\right]$$
\end{tiny}

The derivative $d_2$ for $H$ is a $126 \times 140$ matrix: \\

\hspace{-2.8cm}
\begin{tiny}
$\arraycolsep=0.02pt\def\arraystretch{0.001}
\left[
\begin{array}{cccccccccccccccccccccccccccccccccccccccccccccccccccccccccccccccccccccccccccccccccccccccccccccccccccccccccccccccccccccccccccccccccccccccccccc}
i&0&i&1&0&0&0&0&0&0&0&0&0&0&0&0&0&0&0&0&0&0&0&0&0&0&0&0&0&0&0&0&0&0&0&0&0&0&0&0&0&0&0&0&0&0&0&0&0&0&0&0&0&0&0&0&0&0&0&0&0&0&0&0&0&0&0&0&0&0&0&0&0&0&0&0&0&0&0&0&0&0&0&0&0&0&0&0&0&0&0&0&0&0&0&0&0&0&0&0&0&0&0&0&0&0&0&0&0&0&0&0&0&0&0&0&0&0&0&i&0&0&1&0&0&0&0&0&0&0&0&0&0&0&0&0&0&0&0&0\\
0&0&0&1&i&0&0&0&0&0&0&0&0&0&0&0&0&0&0&0&0&0&0&0&0&0&0&0&0&0&0&0&0&0&0&0&0&0&0&0&0&0&0&0&0&0&0&0&0&0&0&0&0&0&0&0&0&0&0&0&0&0&0&0&0&0&0&0&0&0&0&0&0&0&0&0&0&0&0&0&0&0&0&0&0&0&0&0&0&0&0&0&0&0&0&0&0&0&0&0&0&0&0&0&0&0&0&0&0&0&0&0&0&0&0&0&0&0&0&0&i&0&0&0&0&0&0&0&0&0&0&0&0&0&0&0&0&0&0&0\\
0&i&0&0&0&1&0&0&0&0&0&0&0&0&0&0&0&0&0&0&0&0&0&0&0&0&0&0&0&0&0&0&0&0&0&0&0&0&0&0&0&0&0&0&0&0&0&0&0&0&0&0&0&0&0&0&0&0&0&0&0&0&0&0&0&0&0&0&0&0&0&0&0&0&0&0&0&0&0&0&0&0&0&0&0&0&0&0&0&0&0&0&0&0&0&0&0&0&0&0&0&0&0&0&0&0&0&0&0&0&0&0&0&0&0&0&0&0&0&0&0&0&0&1&0&0&0&0&0&0&0&0&0&0&0&0&0&0&0&0\\
0&0&i&0&0&1&0&0&0&0&0&0&0&0&0&0&0&0&0&0&0&0&0&0&0&0&0&0&0&0&0&0&0&0&0&0&0&0&0&0&0&0&0&0&0&0&0&0&0&0&0&0&0&0&0&0&0&0&0&0&0&0&0&0&0&0&0&0&0&0&0&0&0&0&0&0&0&0&0&0&0&0&0&0&0&0&0&0&0&0&0&0&0&0&0&0&0&0&0&0&0&0&0&0&0&0&0&0&0&0&0&0&0&0&0&0&0&0&0&0&0&0&0&0&1&0&0&0&0&0&0&0&0&0&0&0&0&0&0&0\\
0&0&0&0&i&0&1&0&0&0&0&0&0&0&0&0&0&0&0&0&0&0&0&0&0&0&0&0&0&0&0&0&0&0&0&0&0&0&0&0&0&0&0&0&0&0&0&0&0&0&0&0&0&0&0&0&0&0&0&0&0&0&0&0&0&0&0&0&0&0&0&0&0&0&0&0&0&0&0&0&0&0&0&0&0&0&0&0&0&0&0&0&0&0&0&0&0&0&0&0&0&0&0&0&0&0&0&0&0&0&0&0&0&0&0&0&0&0&0&0&0&i&0&0&0&0&0&0&0&0&0&0&0&0&0&0&0&0&0&0\\
0&0&0&0&0&0&0&i&0&i&0&0&0&0&0&0&0&0&0&0&0&0&0&0&0&0&0&0&0&0&0&0&0&0&0&0&0&0&0&0&0&0&0&0&0&0&0&0&0&0&0&0&0&0&0&0&0&0&0&0&0&0&0&0&0&0&0&0&0&0&0&0&0&0&0&0&0&0&0&0&0&0&0&0&0&0&0&0&0&0&0&0&0&0&0&0&0&0&0&0&0&0&0&0&0&0&0&0&0&0&0&0&0&0&0&0&0&0&0&0&0&0&i&0&0&0&0&0&0&0&0&0&0&0&0&0&0&0&0&0\\
\end{array}
\right]$
\end{tiny}

The exterior derivative $d_3$ for the M\"obius strip $H$ is the $35 \times 126$ matrix

\hspace{-1.8cm}
\begin{tiny}
$\arraycolsep=0.25pt\def\arraystretch{0.04}
\left[
\begin{array}{cccccccccccccccccccccccccccccccccccccccccccccccccccccccccccccccccccccccccccccccccccccccccccccccccccccccccccccccccccccccccccccc}
i&0&0&0&0&0&0&0&0&0&i&0&0&0&1&0&0&0&0&0&0&0&0&0&0&0&0&0&0&0&0&0&0&0&0&0&0&0&0&0&0&0&0&0&0&0&0&0&0&0&0&0&0&0&0&0&0&0&0&0&0&0&0&1&0&1&i&0&0&0&0&0&0&0&0&0&0&0&0&0&0&0&0&0&0&0&0&0&0&0&0&0&0&0&0&0&0&0&0&0&0&0&0&0&0&0&0&0&0&0&0&0&0&0&0&0&0&0&0&0&0&0&0&0&0&0\\
0&i&0&0&0&0&0&0&0&0&0&i&0&0&0&1&0&0&0&0&0&0&0&0&0&0&0&0&0&0&0&0&0&0&0&0&0&0&0&0&0&0&0&0&0&0&0&0&0&0&0&0&0&0&0&0&0&0&0&0&0&0&0&0&0&0&i&1&1&0&0&0&0&0&0&0&0&0&0&0&0&0&0&0&0&0&0&0&0&0&0&0&0&0&0&0&0&0&0&0&0&0&0&0&0&0&0&0&0&0&0&0&0&0&0&0&0&0&0&0&0&0&0&0&0&0\\
0&0&i&0&0&0&0&0&0&0&0&0&i&0&0&0&0&0&0&0&0&0&0&0&0&0&0&0&0&0&0&0&0&0&0&0&0&0&0&0&0&0&0&0&0&0&0&0&0&0&0&0&0&0&0&0&0&0&0&0&0&0&0&0&1&0&0&0&0&0&i&0&0&0&0&0&0&0&0&0&0&0&0&0&0&0&0&0&0&0&0&0&0&0&0&0&0&0&0&0&0&0&0&0&0&0&0&0&0&0&0&0&0&0&0&0&0&0&0&0&0&0&0&0&0&0\\
0&0&0&i&0&0&0&0&0&0&0&0&0&i&0&0&1&0&0&0&0&0&0&0&0&0&0&0&0&0&0&0&0&0&0&0&0&0&0&0&0&0&0&0&0&0&0&0&0&0&0&0&0&0&0&0&0&0&0&0&0&0&0&0&0&1&0&0&0&1&i&0&0&0&0&0&0&0&0&0&0&0&0&0&0&0&0&0&0&0&0&0&0&0&0&0&0&0&0&0&0&0&0&0&0&0&0&0&0&0&0&0&0&0&0&0&0&0&0&0&0&0&0&0&0&0\\
0&0&0&0&i&0&0&0&0&0&0&0&0&0&0&0&0&1&0&0&0&0&0&0&0&0&0&0&0&0&0&0&0&0&0&0&0&0&0&0&0&0&0&0&0&0&0&0&0&0&0&0&0&0&0&0&0&0&0&0&0&0&0&0&0&0&0&1&0&0&0&i&0&0&0&0&0&0&0&0&0&0&0&0&0&0&0&0&0&0&0&0&0&0&0&0&0&0&0&0&0&0&0&0&0&0&0&0&0&0&0&0&0&0&0&0&0&0&0&0&0&0&0&0&0&0\\
0&0&0&0&0&0&0&0&0&0&0&0&0&0&1&0&0&0&i&0&0&0&0&0&0&0&0&i&0&0&0&0&0&0&0&0&0&0&0&0&0&0&0&0&0&0&0&0&0&0&0&0&0&0&0&0&0&0&0&0&0&0&0&0&0&0&0&0&0&0&0&0&1&1&i&0&0&0&0&0&0&0&0&0&0&0&0&0&0&0&0&0&0&0&0&0&0&0&0&0&0&0&0&0&0&0&0&0&0&0&0&0&0&0&0&0&0&0&0&0&0&0&0&0&0&0\\
0&0&0&0&0&0&0&0&0&0&0&0&0&0&0&1&0&0&0&i&0&0&0&0&0&0&0&0&i&0&0&0&0&0&0&0&0&0&0&0&0&0&0&0&0&0&0&0&0&0&0&0&0&0&0&0&0&0&0&0&0&0&0&0&0&0&0&0&0&0&0&0&0&0&i&1&0&1&0&0&0&0&0&0&0&0&0&0&0&0&0&0&0&0&0&0&0&0&0&0&0&0&0&0&0&0&0&0&0&0&0&0&0&0&0&0&0&0&0&0&0&0&0&0&0&0\\
\end{array} \right]$
\end{tiny}

It takes as an argument a $3$-form $F$, a function on pairs $(x,y)$ of simplices for which the dimension
adds to three. These are the $63$ intersecting edge-triangle pairs together with the $63$ intersecting triangle-edge pairs,
in total 126 pairs. The derivative result $dF$ is a $4$-form, a function on the $35$ triangle-triangle pairs. 

It is rare to have trivial quadratic cohomology. We let the computer search through Erd\"os-Renyi spaces of
graphs with $10$ vertices and found one in maybe 5000 graphs. Interestingly, so far, all of them were homeomorphic
to the Moebius strip. 

\ssection{Remarks}

\begin{itemize}
\item Quadratic cohomology is the second of an infinite sequence of interaction cohomologies.
In general, we can work with $k$-tuples of pairwise intersecting simplices. The basic numerical
value in the $k$-linear case is the {\bf k'th Wu characteristic} 
$\omega_k(G) = \sum_{x_1 \sim \cdots \sim x_k} (-1)^{\sum_j {\rm dim}(x_j)}$ which
sums over all pairwise intersecting ordered $k$-tuples $(x_1, \dots, x_k)$ of simplices in 
the graph $G$. The example of odd-dimensional discrete manifolds shows that it is possible that
all higher cohomologies depend on simplicial cohomology only. So, even when including all
cohomologies, we never can get a complete set of invariants. There are
always non-homeomorphic graphs with the same higher cohomology data. 
\item In the case of the example used here, the cubic interaction cohomology of the
graph is $(0,0,0,0,1,1,0)$. It is the same than the cubic interaction cohomology
of the M\"obius strip. The Laplacian $L$ of the cohomology already has 3824 entries
in the cylinder case with $|V|=8$ and 3346 entries in the M\"obius case with $|V|=7$. 
After a Barycentric refinement, the cylinder is a graph with 32 vertices and 80 edges. 
The cubic interaction Laplacian is then already a $45280 \times 45280$ matrix.
For the refinment of the M\"obius strip, a graph with 28 vertices and 70 edges, the
cubic interaction Laplacian is a $39620 \times 39620$ matrix. This is already tough for
computers. In the quadratic case, where we have a $2856 \times 2856$ Laplacian matrix in the 
M\"obius case, we can still determine that the determinant is invertible. The determinant
is a 1972 digit number but it appears already too large to be factored (we can factor 
${\rm det}(L_0),{\rm det}(L_1)$ but no more ${\rm det}(L_2)$. 

\begin{tiny}
\begin{tabular}{l}
7530638969874397037979561711563112002450663202144179001869326434789895988079316266055681741641534250\\
5520402569595314660610608902074065506357217964315818517986509791000567001716964026456737170353278527\\
1205791817169519067463372551421515830453988162735079594026130431933358246089334527639943252893043314\\
4536846307940785014527346854658200312809912922056973558229588054325241084931073085380158240709097935\\
6856545344770686703962622087593958650046156181661362355633631125541312189408443468961510149659023008\\
2253124421977553946079895639750091735324856288724091518682261237676362313177970059565027587536004200\\
2138919062124125445019983450555077106254638768447785154161762378024206587870332810577525207756811085\\
2067854380728346698788219529968766187343056661909834655700724253992870078325196771420313519612791286\\
9986050780368228721041434897398832261931316935066219252063684161315984087992975315123939475219156474\\
5064548408545070430211633945500014771746565658431354824735645943447432884777245390819277928091941305\\
0609424952327156527353420951979538240302361150261531075349177730786563980994148788098868370790909985\\
1070271552493935030583189076381732318859268620773387011789160691174099395371231597923911979872060284\\
0306733497401619307415880298487658209948973720609493409838205662853016994078695087814672005466663217\\
9379439733797746978039824274589416892649204500101226831922014937137653898387522575569863167507965108\\
4060326919043324423083543518686472734091371402379539572955982226395668169331555359933747282072900339\\
2881996143606297654761309496728269968651404538957809358589034309430857786052252496945529658134936169\\
5395066531914463700676327141913762262475351302519866878990530107809129987566574701732438272093994454\\
1389573174488922097890193235110489144406702328313353424147603306762836337219491316933490480514757457\\
9714638345039702184738990758949184452076641181016476076100416768684884112135152505650404031544520907\\
559177723438295140827322671712156466607554560000000000000000000000000000\\
\end{tabular}
\end{tiny}

\item The interaction Betti numbers satisfy the Whitney-Cartan type formula
$b_p(G \times H) = \sum_{k+l=p} b_k(G) b_l(H)$ like Stiefel-Whitney classes 
but seem unrelated to Stiefel-Whitney classes. 
\item The cylinder $G$ and M\"obius strip $H$ can be distinguished topologically:
one tool is orientation, an other is the connectedness of the boundary.
Why is an algebraic description via Betti numbers interesting? Because they 
are purely algebraically defined.  
\item Using computer search we have found examples of finite simple graphs for which the
$f$-vector and $f$-matrix are both the same, where the singular cohomologies are the same
and where the Wu characteristic is the same, but where the Betti vector for
quadratic cohomology is different.
\item For complicated networks, the computations of interaction cohomology become
heavy, already in cubic cases we reach limits of computing rather quickly. 
It can be useful therefore to find first a smaller homeomorphic
graph and compute the cohomology there. 
\item Homeomorphic graphs have the same simplicial cohomologies.
This is true also for higher cohomologies but it is less obvious. One 
can see it as follows; after some Barycentric refinements,
every point is now either a point with manifold neighborhood or a singularity, 
where various discrete manifolds with boundary come together. 
For homeomorphic graphs, the structure of these singularities must correspond. 
It follows from Gauss-Bonnet formulas that the 
higher Wu-characteristics are the same. To see it for cohomology, we will 
have to look at Meyer-Vietoris type computations for glueing cohomologie or 
invariance under deformations like edge collapse or refinement. 
\item We feel that the subject of interaction cohomology has a lot of potential also
to distinguish graphs embedded in other graphs. While Wu characteristic for embedded knots in a 3-sphere
is always 1, the cohomology could be interesting. 
\item We have not yet found a known cohomology in the continuum which comes close to 
interaction cohomology. Anything we looked at looks considerably more complicated. 
Interaction cohomology is elementary: one just defines some concrete finite matrices and 
computes their kernel. 
\item For more, see \cite{KnillKuenneth,KnillTopology,DiracKnill,KnillBaltimore,
cauchybinet,knillmckeansinger,brouwergraph,valuation}.
\end{itemize}

\scalebox{0.12}{\includegraphics{figures/moebius.pdf}} 
\scalebox{0.12}{\includegraphics{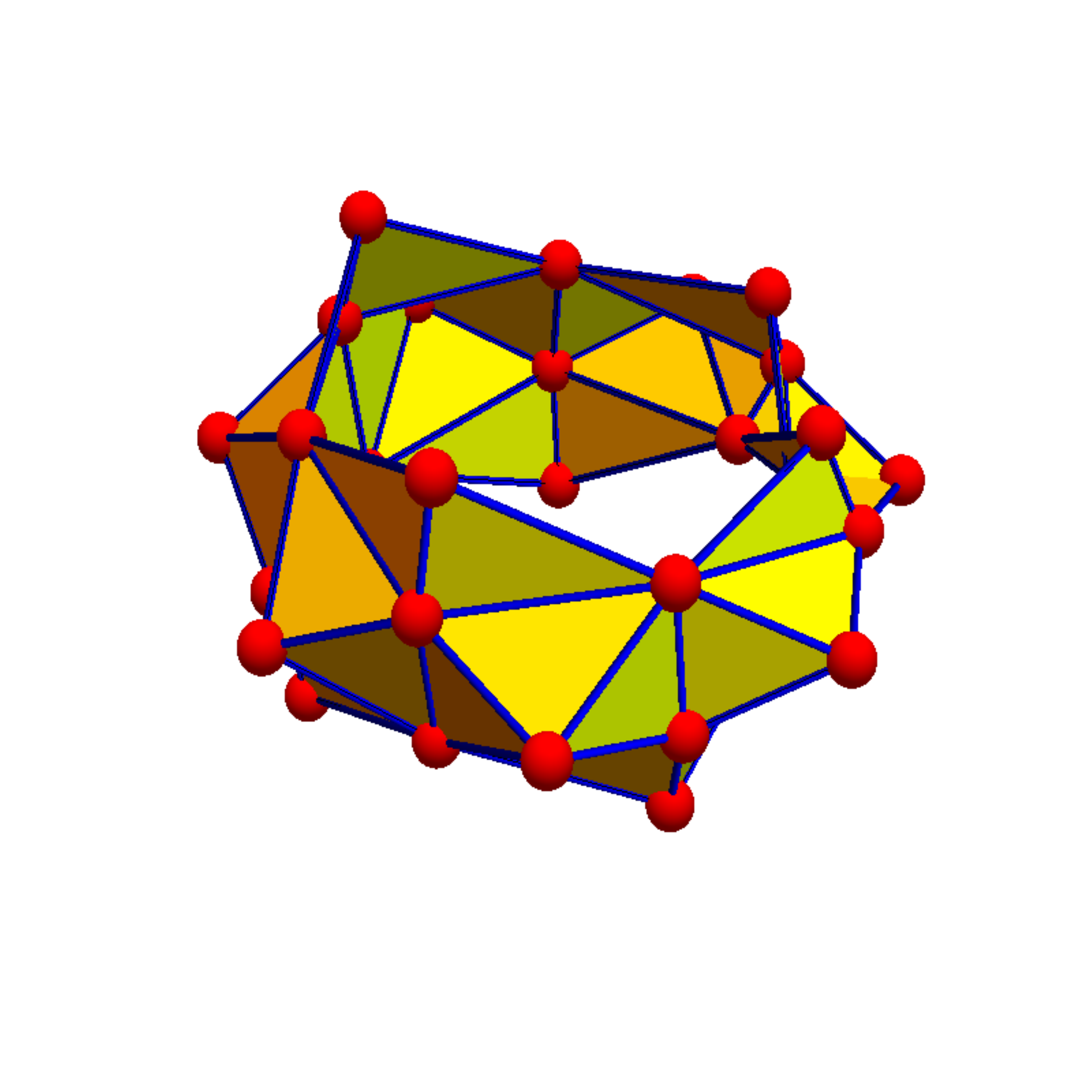}} 
\scalebox{0.12}{\includegraphics{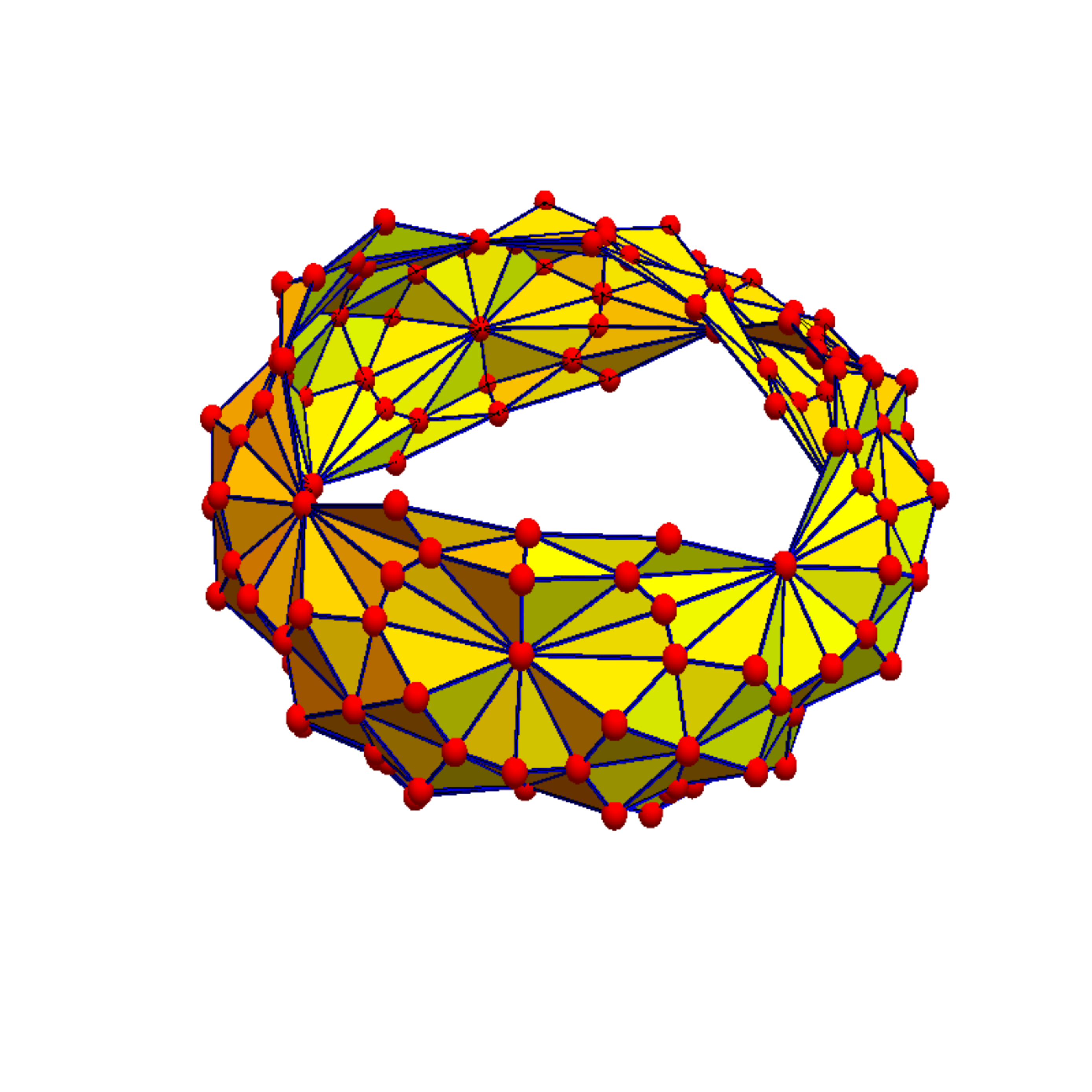}} 
\scalebox{0.12}{\includegraphics{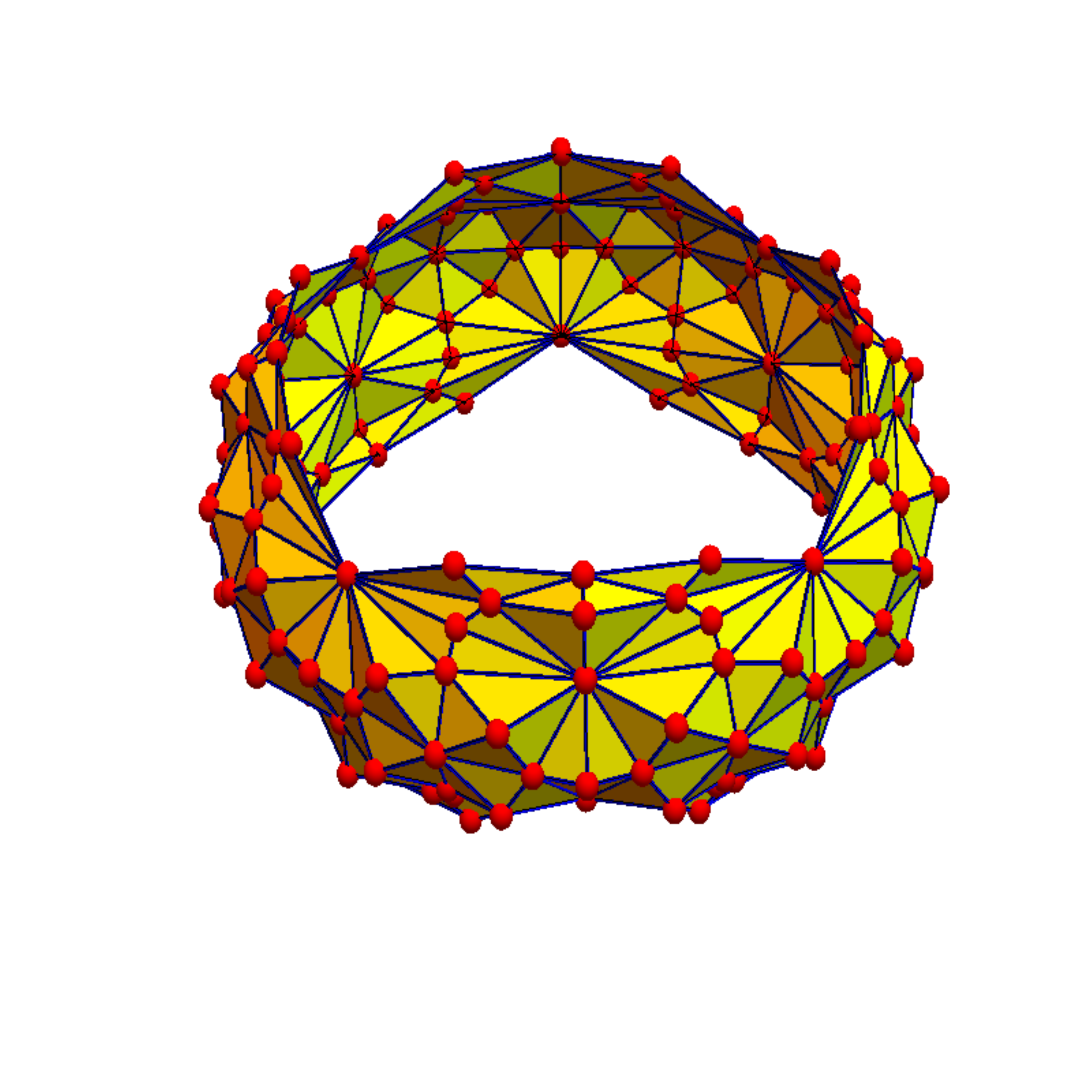}} 
\scalebox{0.12}{\includegraphics{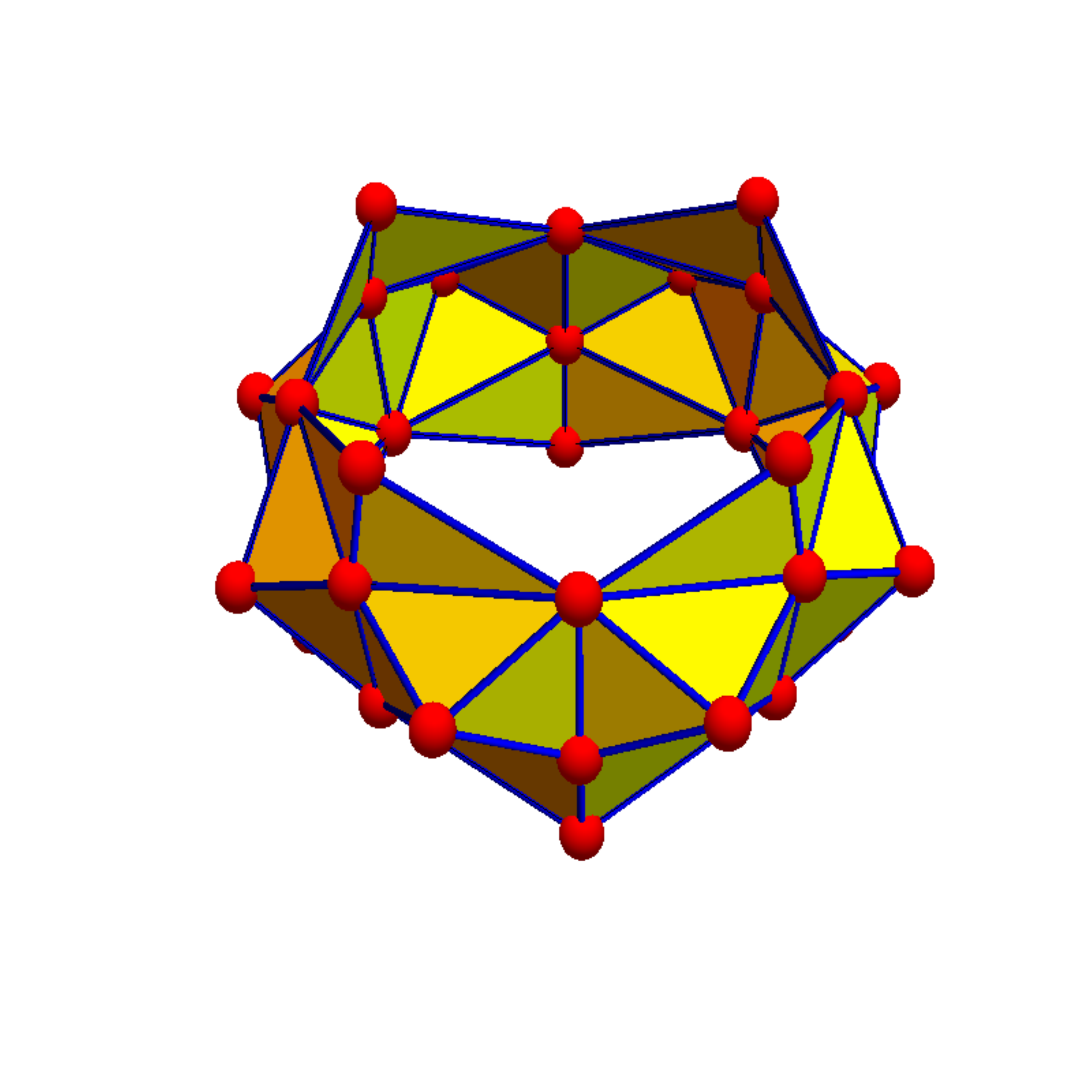}} 
\scalebox{0.12}{\includegraphics{figures/cylinder.pdf}}

\hspace{-1cm}
\scalebox{0.47}{\includegraphics{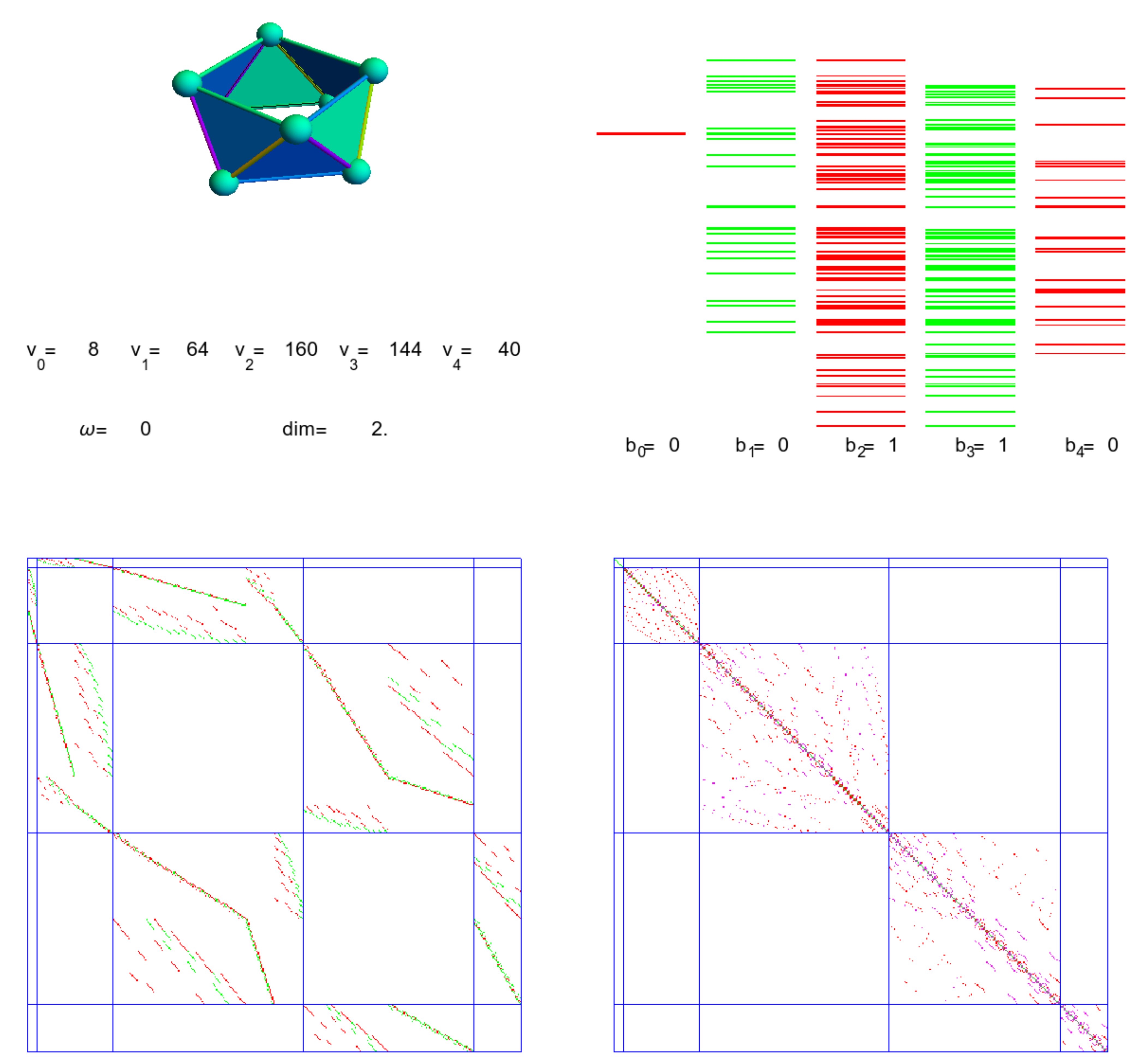}} \\
The figure shows the Dirac and Laplacian of the cylinder graph $G$.
We see the spectrum of the interaction Laplacian on 
each of the $5$ interaction form sectors $p=0, \dots, p=4$. We see also 
the super symmetry: the union of the spectra on even dimensional 
forms $p=0,2,4$ is the same than the union of the spectra on 
odd-dimensional forms $p=1,3$. 

\hspace{-1cm}
\scalebox{0.47}{\includegraphics{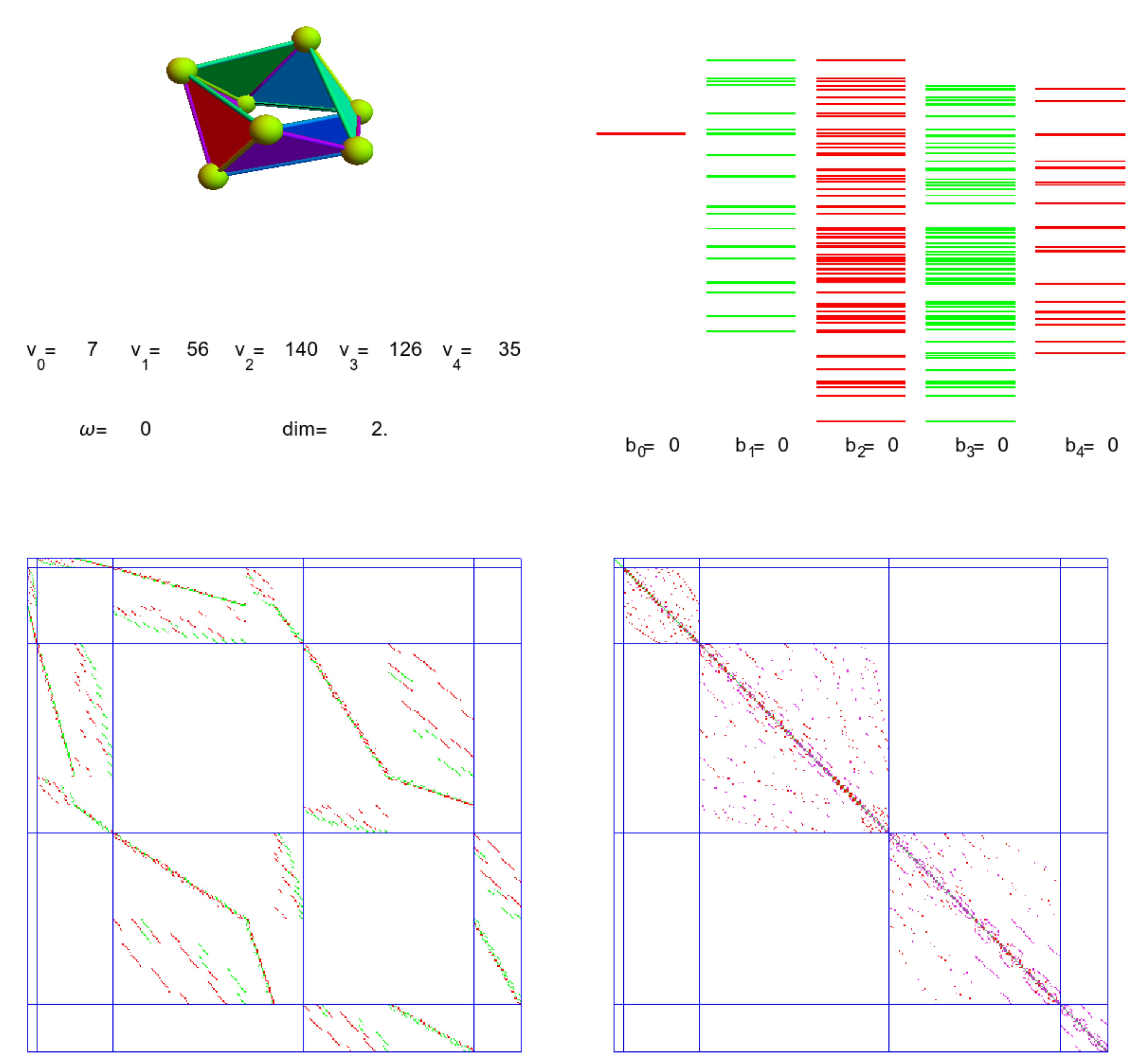}} \\
The figure shows the Dirac and Laplacian of the M\"obius graph $G$.
We see the spectrum of the interaction Laplacian on 
each of the $5$ interaction sectors $p=0, \dots, p=4$. 
There is one important difference however: the Laplacian
is invertible.  \\

\pagebreak


\bibliographystyle{plain}

\end{document}